\documentclass{report}

\usepackage{subfig}
\usepackage{verbatim}

\usepackage{epsfig}
\usepackage{graphicx}
\usepackage{amsbsy}
\usepackage{amsmath}
\usepackage{amsfonts}
\usepackage{amssymb}
\usepackage{textcomp}
\usepackage{hyperref}
\usepackage{aliascnt}

\newcommand{\mcm}[3]{\newcommand{#1}[#2]{{\ensuremath{#3}}}} 

\mcm{\tuple}{1}{\langle #1 \rangle}
\mcm{\name}{1}{\ulcorner #1 \urcorner}
\mcm{\Nbb}{0}{\mathbb{N}}
\mcm{\Zbb}{0}{\mathbb{Z}}
\mcm{\Rbb}{0}{\mathbb{R}}
\mcm{\Cbb}{0}{\mathbb{C}}
\mcm{\Qbb}{0}{\mathbb{Q}}
\mcm{\Acal}{0}{\cal A}
\mcm{\Bcal}{0}{\cal B}
\mcm{\Ccal}{0}{\cal C}
\mcm{\Dcal}{0}{\cal D}
\mcm{\Ecal}{0}{\cal E}
\mcm{\Fcal}{0}{\cal F}
\mcm{\Gcal}{0}{\cal G}
\mcm{\Hcal}{0}{\cal H}
\mcm{\Ical}{0}{\cal I}
\mcm{\Jcal}{0}{\cal J}
\mcm{\Kcal}{0}{\cal K}
\mcm{\Lcal}{0}{\cal L}
\mcm{\Mcal}{0}{\cal M}
\mcm{\Ncal}{0}{\cal N}
\mcm{\Ocal}{0}{{\cal O}}
\mcm{\Pcal}{0}{{\cal P}}
\mcm{\Qcal}{0}{{\cal Q}}
\mcm{\Rcal}{0}{{\cal R}}
\mcm{\Scal}{0}{{\cal S}}
\mcm{\Tcal}{0}{{\cal T}}
\mcm{\Ucal}{0}{{\cal U}}
\mcm{\Vcal}{0}{{\cal V}}
\mcm{\Wcal}{0}{{\cal W}}
\mcm{\Xcal}{0}{{\cal X}}
\mcm{\Ycal}{0}{{\cal Y}}
\mcm{\Zcal}{0}{{\cal Z}}
\mcm{\Mfrak}{0}{\mathfrak M}

\mcm{\restric}{0}{\upharpoonright}
\mcm{\upset}{0}{\uparrow}
\mcm{\onto}{0}{\twoheadrightarrow}
\mcm{\smallNbb}{0}{{\small \mathbb{N}}}
\DeclareMathOperator{\preop}{op}
\mcm{\op}{0}{^{\preop}}

\newcommand{\se}{\subseteq}
{\begin{array}{c}
\setlength{\unitlength}{1em}}%
{\end{array}}

\usepackage{amsthm}

\newcommand{\theoremize}[2]{\newaliascnt{#1}{thm} \newtheorem{#1}[#1]{#2} \aliascntresetthe{#1}}

\theoremstyle{plain}
\newtheorem{thm}{Theorem}[section]
\theoremize{lem}{Lemma}
\theoremize{skolem}{Skolem}
\theoremize{fact}{Fact}
\theoremize{sublem}{Claim}
\theoremize{claim}{Claim}
\theoremize{obs}{Observation}
\theoremize{prop}{Proposition}
\theoremize{cor}{Corollary}
\theoremize{que}{Question}
\theoremize{oque}{Open Question}
\theoremize{oprob}{Open Problem}
\theoremize{con}{Conjecture}

\theoremstyle{definition}
\theoremize{dfn}{Definition}
\theoremize{rem}{Remark}
\theoremize{eg}{Example}
\theoremize{exercise}{Exercise}
\theoremstyle{plain}

\newenvironment{cproof}{\noindent\textit{Proof of Claim.}}{\phantom{!}\!\hfill$\diamondsuit$}

\mcm{\Fbb}{0}{\mathbb{F}}

\DeclareMathOperator{\degree}{deg}

\DeclareMathOperator{\Sbb}{\mathbb{S}}

\newcommand{\Sthree}{$\Sbb^3$}

\newcommand{\sm}{\setminus}

\newcommand{\scom}{simplicial complex}

\usepackage{pdfpages}
\usepackage[percent]{overpic}

\usepackage{scrextend}
\begin{document}

\begin{center}
 {\LARGE Embedding simply connected 2-complexes in 3-space}

 \vspace{.6cm}Johannes Carmesin\\
 \vspace{.2cm}
 {University of Birmingham} \\\vspace{.3cm}
 \today
\end{center}

            \vspace{.55cm}

\begin{center}
 {\bf Abstract}

 \begin{addmargin}[2em]{2em}
{\small
\indent Firstly, we characterise the embeddability of simply connected locally 3-connected 2-dimensional
simplicial
complexes in 3-space in a way analogous to Kuratowski's characterisation of graph planarity, by nine
excluded minors. This answers questions of Lov\'asz, Pardon and U. Wagner.

The excluded minors are the cones over $K_5$ and $K_{3,3}$, five related constructions, and the remaining two are obtained from triangulations of the M\"obius strip by attaching a disc at its central cycle.

Secondly, we extend the above theorem to all simply connected 2-dimensional simplicial complexes.
}
\end{addmargin}
\end{center}

\theoremstyle{plain}
\newtheorem{main}{Theorem}
\def\mainautorefname{Theorem}
\newcommand{\mainmize}[2]{\newaliascnt{#1}{main} \newtheorem{#1}[#1]{#2} \aliascntresetthe{#1}}

\section*{Introduction}

In 1930, Kuratowski proved that a graph can be embedded in the plane if and only if it has none of
the two non-planar graphs $K_5$ or $K_{3,3}$ as a minor\footnote{A \emph{minor} of a graph is
obtained by deleting or contracting edges. Kuratowski's theorem is stated in terms of subdivisions but here we refer to it in the modern version in terms of minors; both statements are trivially equivalent.
The perspective in terms of the minor operation is due to K. Wagner, who used it to
prove the equivalence between the 4 Colour Theorem and Hadwiger's Conjecture \cite{MR0012237} for 4 colours; and thus founded graph minor theory \cite{wagner1937eigenschaft}.} \cite{kuratowski1930probleme}. The main result of this paper may be regarded as a
3-dimensional analogue of this theorem.

\vspace{.15cm}

Kuratowski's theorem gives a way how embeddings in the plane could be understood through the
minor relation. A far reaching extension of Kuratowski's theorem is the Robertson-Seymour theorem
\cite{MR2099147}. Any minor-closed class of graphs is characterised by the list of minor-minimal
graphs not
in the class. This theorem says that this list always must be finite.
The methods developed to prove this theorem are nowadays used in many results in the
area of structural graph theory \cite{DiestelBookCurrent} --
and beyond; recently Geelen, Gerards and Whittle extended the Robertson-Seymour theorem
to representable matroids by
proving  Rota's conjecture \cite{solving_Rota}.
Very roughly, the Robertson-Seymour structure theorem establishes a correspondence between minor
closed classes of graphs and classes of graphs almost embeddable in 2-dimensional surfaces.

In his survey on the Graph Minor project of Robertson and Seymour \cite{lovasz_gm_survey}, in 2006
Lov\'asz asked
whether there is a
meaningful analogue of the minor relation in three dimensions. Clearly, every graph can be
embedded in 3-space\footnote{Indeed, embed the vertices in general position and embed the edges
as straight lines. }.

One approach towards this question is to restrict the embeddings in question, and just consider so
called linkless embeddings of graphs, see \cite{linkless_emb_survey} for a survey. Instead of
restricting embeddings, one could also
put some additional structure on the graphs in question. Indeed, U. Wagner  asked how an analogue of
the minor relation could be defined on general simplicial complexes \cite{Wagner_minor}. All simplicial complexes considered in this paper are finite, for infinite ones see \cite{georgakopoulos2022discrete}.

Unlike in higher dimensions, a 2-dimensional simplicial complex has a topological embedding in
3-space if and only if it has a piece-wise linear embedding, which in turn holds if and only if it has a differential
embedding \cite{{Bin59},{Hatcher3notes},{moise},{Pap43}}. Matou\v sek, Sedgewick, Tancer and U. Wagner\cite{mstw_3sphere_decidable}
proved that
the embedding problem for 2-dimensional simplicial complexes in 3-space is
decidable. De Mesmay, Rieck, Sedgwick and Tancer complemented this result by showing that this problem is
NP-hard \cite{s3np_hard}.

Our lack of understanding NP-hard problems may suggest that a structural characterisation of
embeddability of 2-dimensional simplicial complexes\footnote{All graphs and simplicial complexes considered in this paper are finite.} in general is out of reach.
But what about structural characterisations for natural subclasses? In fact such
questions have
been asked: in 2011 at the internet forum `MathsOverflow' Pardon\footnote{John Pardon confirmed in
private communication that he asked that question as the user `John Pardon'.} asked
whether there are necessary and sufficient conditions for when contractible 2-dimensional
simplicial complexes embed in 3-space. The \emph{link graph} at a vertex $v$ of a
simplicial complex
is the incidence graph between edges and faces incident with $v$.
Pardon notes that if embeddable the link graph at any vertex must be planar. This leads to obstructions
for embeddability such as the cone over the complete graph $K_5$, see
\autoref{cK5}. -- But there are different
obstructions of a more global character, see \autoref{8obs}. All their link graphs are planar --
yet they are not embeddable.
   \begin{figure} [htpb]
\begin{center}
   	  \includegraphics[height=4cm]{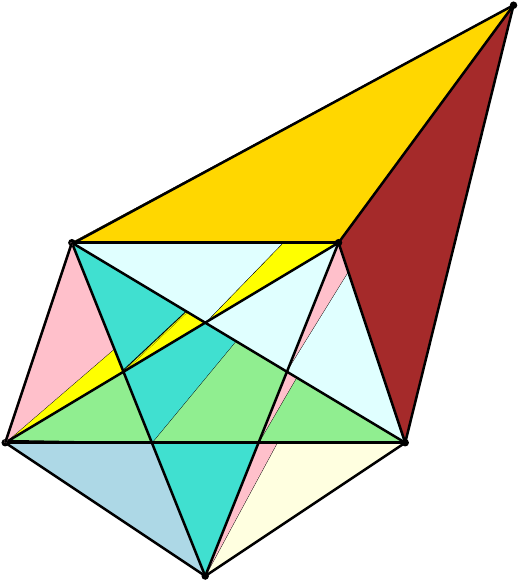}
   	  \caption{The cone over $K_5$.
Similarly as the graph $K_5$ does not embed in 2-space, the cone over
$K_5$ does not embed in 3-space. }\label{cK5}
\end{center}\vspace{-0.7cm}
   \end{figure}

   \begin{figure} [htpb]
\begin{center}
   	  \includegraphics[height=4cm]{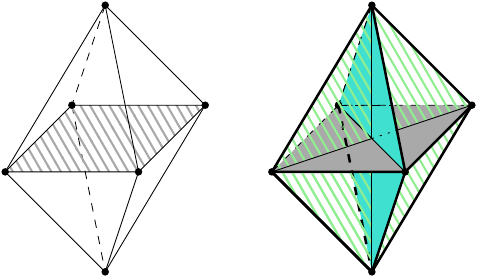}
   	  \caption{The octahedron obstruction, depicted on the right, is obtained from the
octahedron with its eight triangular faces by
adding 3 more faces of size 4 orthogonal to the three axis. If we add just one of these
4-faces
to the octahedron, the resulting 2-complex is embeddable as illustrated on the left. A
second 4-face could be added on the outside of that depicted embedding. However, it can
be shown that the octahedron with all three 4-faces is not embeddable.}\label{8obs}
\end{center}\vspace{-0.7cm}
   \end{figure}


Addressing these questions, we introduce an analogue of the minor relation
for 2-complexes and we use it to prove a 3-dimensional analogue of
Kuratowski's theorem characterising when simply connected 2-dimensional simplicial complexes
(topologically) embed
in 3-space.
\begin{figure}
\begin{center}
    \begin{tabular}{ | l | l  | l |}
    \hline
    {} & delete & contract  \\ \hline
    edge & {} &{} \\ \hline
    face & {}&{} \\ \hline
    \end{tabular}
    \caption{For each of the four corners of the above diagram we have one space minor
operation.}\label{4ops}
\end{center}\vspace{-0.7cm}
\end{figure}
More precisely, a \emph{space minor} of a 2-complex is obtained by successively
deleting or contracting edges or faces, and splitting vertices.
See \autoref{4ops} and \autoref{fig:space_minor2}. The precise details of these definitions are
given in \autoref{sec:space}; for example contraction of edges is only allowed for edges that
are not loops\footnote{\emph{Loops} are edges that have only a single endvertex. While
contraction of edges that are not loops clearly preserves embeddability in 3-space, for loops this
is not always the case.} and we only contract faces of size at most two.
   \begin{figure} [htpb]
\begin{center}
   	  \includegraphics[height=3cm]{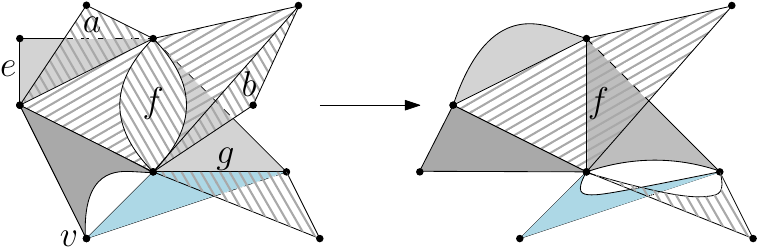}
   	  \caption{The complex on the right is a space minor of the complex on the left. To see this, delete the faces labelled $a$ and $b$, contract the edge $e$, contract the face
$f$, delete the edge $g$ and split the vertex $v$. Here deleting an edge $g$ means that we replace $g$ by a copy for every incident face, where the copy is only incident with a single face where it takes the role of $g$.}\label{fig:space_minor2}
\end{center}\vspace{-0.7cm}
   \end{figure}

It will be quite easy to see that space minors preserve embeddability in 3-space and that this
relation
is well-founded. The operations of face deletion and
face contraction correspond to the minor operations in the dual matroids of simplicial complexes in
the sense of \cite{3space4}.

\begin{eg}
 Using space minors, we can understand why the Octahedron Obstruction (\autoref{8obs}) does not
embed in 3-space. Indeed, we contract an arbitrary face of size three to a single vertex (formally,
we first contract an edge of that face, then it gets size two. So we can contract it to an edge.
Then we contract that edge to a vertex). It turns out that the link graph at the new vertex is the
non-planar graph $K_{3,3}$. Thus this space minor is not embeddable in 3-space. As space minors
preserve embeddability, we deduce that the Octahedron Obstruction cannot be embeddable.
\end{eg}

A construction of a simply connected 2-complex that is not embeddable in 3-space and has no space
minor with a non-planar link graph can be obtained from the M\"obius strip by attaching a disc at its central cycle, see \autoref{fig:moe1}. We refer to these constructions as \emph{M\"obius obstructions}.

   \begin{figure} [htpb]
\begin{center}
   	  \includegraphics[height=3cm]{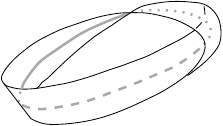}
   	  \caption{The  M\"obius-strip. Its central cycle is depicted in grey.}\label{fig:moe1}
\end{center}\vspace{-0.7cm}
   \end{figure}

The main result of \autoref{chapterI} is the following. A \emph{space restriction} is a space minor where we abstain from contracting edges or faces.

\begin{main}\label{kura_intro}
Let $C$ be a simply connected locally 3-connected 2-dimensional simplicial complex. The following
are equivalent.
\begin{itemize}
 \item $C$ embeds in 3-space;
 \item $C$ has no space restriction from an explicit list $\Zcal$.
\end{itemize}
\end{main}

The list $\Zcal$ consists of cones over subdivisions of $K_5$ and $K_{3,3}$, five similar constructions to which we refer to as \lq combined cones\rq, see \autoref{expli-computation} below for details, and M\"obius obstructions. The set of space-minor minimal elements of $\Zcal$ consists of the cones over $K_5$ and $K_{3,3}$ and a finite set of five combined cones constructions and the finite set of minimal M\"obius obstructions. In particular, this set of space-minor minimal elements of $\Zcal$ is finite.

Here a (2-dimensional) simplicial complex is \emph{locally
3-connected} if all its link graphs are 3-connected. In \autoref{chapterV}, we extend \autoref{kura_intro} to simplicial complexes that need
not be
locally
3-connected.
\vspace{.3cm}

We are able to extend \autoref{kura_intro} from simply connected
simplicial complexes to those with trivial first homology group.

\begin{main}\label{kura_intro_hom}
Let $C$ be a locally 3-connected 2-dimensional simplicial complex such that the first homology
group $H_1(C,\Fbb_p)$ is trivial for some prime $p$. The following are equivalent.
\begin{itemize}
 \item $C$ embeds in 3-space;
 \item $C$ is simply connected and has no space restriction from
the explicit list $\Zcal$.
\end{itemize}
\end{main}

In general there are infinitely many obstructions to embeddability in 3-space. Indeed, the
following infinite family consists of obstructions of trivial homology group that are not simply connected, and so these obstructions directly connect to \autoref{kura_intro_hom}.
\begin{eg}\label{q-folded}
Given a natural number $q\geq 2$, the $q$-folded cross cap consists of a single vertex,
a single edge that is a loop and a single face traversing the edge $q$-times in the same
direction. It can be shown that $q$-folded cross caps cannot be embedded in 3-space.
\end{eg}
\noindent A more sophisticated infinite family of obstructions to embeddability is constructed in \cite{3space4}.
Examples of follow-up works of this paper include: \cite{georgakopoulos20232}, \cite{georgakopoulos2022discrete}, \cite{fulek2022atomic}, \cite{carmesin2022new}.

\vspace{.3cm}

{\bf Overview over this paper.}
This paper is subdivided into three chapters.
The proof of  \autoref{kura_intro_hom} (which implies \autoref{kura_intro}) spans \autoref{chapterI} and \autoref{chapterII}. \autoref{chapterII} is self-contained and deals with the topological aspects of this paper.

The main result of  \autoref{chapterV} is an extension of \autoref{kura_intro_hom} dropping the assumption of local 3-connectivity. This result can be combined with the machinery from \cite{3space4} to extend \autoref{kura_intro_hom} beyond 2-complexes with trivial homology group. Examples demonstrating that these results cannot be extended much further can be found in \cite{3space4}.

	\tableofcontents

	\chapter{The combinatorial part of the proof of \autoref{kura_intro_hom}}\label{chapterI}

The proof of \autoref{kura_intro_hom} (which implies \autoref{kura_intro}) is subdivided into a topological and combinatorial part.
The connection between these parts is the notion of `planar rotation systems'; roughly speaking, this is some combinatorial data on a 2-complex. In the topological part we show that the existence of planar rotation systems is equivalent to the existence of an embedding; this is the topic of \autoref{chapterII}.
In \autoref{chapterI}, we characterise combinatorially which 2-complexes admit planar rotation systems.
This is done in several steps.
In \autoref{sec_vertex_sum} we make a first step towards an excluded minors characterisations by studying how planar rotation systems behave with respect to contractions of edges.
In \autoref{s_constr} we use this to find in every simply connected 2-complex without a planar rotation system an obstruction to embeddability localised at a vertex, an edge or a face.
In \autoref{restricted_section}, we use this to prove \autoref{kura_intro_hom} -- assuming a single lemma, \autoref{combi_intro_extended}, from \autoref{chapterII}. In \autoref{sec:space} we introduce space minors and prove that the list $\Zcal$ of \autoref{kura_intro_hom} has nine minimal elements and determine them. Concluding remarks follow in \autoref{concl77}.

\vspace{.3cm}

For graphs\footnote{In this chapter graphs are allowed to have loops and parallel edges. } we
follow
the notation of \cite{DiestelBookCurrent}.
To read this chapter, a minimal amount of topological background is necessary, which we summarise now.

A \emph{2-complex} is a graph $(V,E)$ together with a set $F$ of closed trails\footnote{A
\emph{trail} is sequence
$(e_i|i\leq n)$ of distinct edges such that the endvertex of $e_i$ is the starting vertex of
$e_{i+1}$ for all $i<n$.  A trail is \emph{closed} if the starting vertex of $e_1$ is equal to
the endvertex of $e_n$.}, called its
\emph{faces}. We denote 2-complexes $C$ by triples $C=(V,E,F)$.
The definition of  \emph{link graphs} naturally extends from simplicial complexes to
2-complexes with the following modification: we add two vertices in the link graph $L(v)$
for each loop incident
with $v$. We add one edge to $L(v)$ for each traversal of a face at $v$.
In this paper we suppress from our notation an injection from the vertices and edges of the link
graph to $E$ and $F$, respectively; and simply consider the vertices of the link graph as edges of
$C$, for instance.
Throughout this paper,
we denote faces of a 2-complex by $f$, edges by $e$ and vertices by $v$. Thus we denote vertices of link graphs by $e$ and edges of link graphs by $f$ (with the understanding that they come from edges or faces of a 2-complex, respectively).

Rotation systems of 2-complexes play a central role in our proof of \autoref{kura_intro}. In this
section we
introduce them and prove some basic properties of them.
A rotation system of a graph $G$ is a family $(\sigma_v|v\in V(G))$ of cyclic
orientations\footnote{A \emph{cyclic orientation} is a bijection to an oriented cycle. This notion is closely related to that of a \lq cyclic ordering\rq, which is a bijection to a cycle (which does not come with an orientation). We stress that in our notation it will be important to remember this orientation.}
$\sigma_v$ of the edges incident with the vertices $v$ \cite{MoharThomassen}. The
orientations $\sigma_v$
are called
\emph{rotators}.
Whenever a graph $G$ is embedded in an oriented  (2-dimensional) surface, the cyclic orientation in which the edges incident with a vertex $v$ appear around it (with respect to the orientation of the surface) is a rotator at $v$.
Conversely, any rotation system of a graph $G$ induces an
embedding of $G$ in an oriented surface $S$. To be precise, we obtain $S$ from $G$
by gluing faces onto (the geometric realisation of) $G$ along closed walks of $G$ as follows.
Each directed edge of $G$ is in one of these walks. Here the direction $\vec{a}$ is
directly before the direction $\vec{b}$ in a face $f$ if the endvertex $v$ of $\vec{a}$ is equal to
the starting vertex of $\vec{b}$ and $b$ is just after $a$ in the rotator at $v$.
The rotation system is \emph{planar} if that surface $S$ is a disjoint union of 2-spheres.
Note that if the graph $G$ is connected, then for any rotation system of $G$, also the
surface $S$ is connected.

A \emph{rotation system of a (directed\footnote{A \emph{directed} 2-complex is a
2-complex together with a choice of
direction at each of its edges and a choice
of one of the two orientations of each of the closed trails corresponding to the faces. All 2-complexes considered in this paper are directed. In order to simplify notation
we will
not always say that explicitly. }) 2-complex} $C$ is a  family $(\sigma_e|e\in E(C))$ of
cyclic orientations $\sigma_e$ of the faces incident\footnote{Recall that a face of a 2-complex traverses each edge at most one by definition.} with each edge $e$.
A rotation system of a 2-complex $C$ \emph{induces} a rotation system at each of its link graphs
$L(v)$ by
restricting to the edges that are vertices of the link graph $L(v)$; here we take $\sigma(e)$ if
$e$
is directed towards $v$ and the reverse of $\sigma(e)$ otherwise.
Here we follow the convention that a loop $e$ at a vertex $v$ is oriented from one attachment at $v$ to the other and hence its rotator projects to $\sigma(e)$ in $L(v)$ at the copy of $e$ corresponding to the attachment it is directed towards and the reverse of $\sigma(e)$ at the other copy.

A rotation system of a 2-complex is \emph{planar} if the induced rotation system on each link graph is planar.
\begin{eg}
Like embeddings of graphs in oriented surfaces induce rotation systems, embeddings of 2-complexes $C$ in oriented 3-manifolds induce planar rotation systems of $C$. Indeed, the cyclic orderings in which the faces appear around each edge $e$ (with respect to the orientation) form a rotator, and planarity follows from the fact that vertices of the 2-complex are embedded to points of a 3-manifold.
\end{eg}

In \autoref{chapterII} we prove the following, which we use in the proof of
\autoref{kura_intro}.

\begin{thm}\label{emb_to_rot}[\autoref{combi_intro}]
A simply connected simplicial complex has an embedding in \Sthree\  if and only if it has a planar
rotation system.
\end{thm}

\section{Vertex sums}\label{sec_vertex_sum}

In this short section we prove some elementary facts about an operation we call `vertex sum'
which
is used in the proof of \autoref{kura_intro}.

Let $H_1$ and $H_2$ be two graphs with a common vertex $v$ and a bijection $\psi$ between the
edges incident with $v$ in $H_1$ and $H_2$.
The \emph{vertex sum} of $H_1$ and $H_2$ over $v$ with respect to $\psi$ is the graph obtained from the
disjoint union of $H_1$ and $H_2$ by deleting $v$ in both $H_i$ and adding an edge between any pair
$(v_1,v_2)$ of vertices $v_1\in V(H_1)$ and  $v_2\in V(H_2)$ such that $v_1v$ and $v_2v$ are
mapped to one another by $\psi$, see \autoref{fig:vx_sum}.
   \begin{figure} [htpb]
\begin{center}
   	  \includegraphics[height=2.5cm]{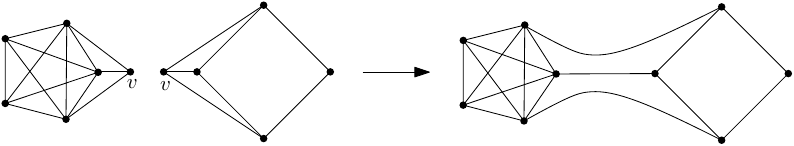}
   	  \caption{The vertex sum of the two graphs on the left is the graph on
the right.}\label{fig:vx_sum}
\end{center}\vspace{-0.7cm}
   \end{figure}

Let $C=(V,E,F)$ be a 2-complex and let $e$ be a non-loop edge of $C$, the 2-complex obtained from
$C$ by \emph{contracting $e$} (denoted by $C/e$) is obtained from $C$ by identifying the two
endvertices of $e$, removing $e$ from all faces and then removing $e$, formally:
$C/e=((V,E)/e, \{f-e|f\in F\})$.

\begin{obs}\label{obs1}
 The link graph of $C/e$ at
$e$ is the vertex sum of the link graphs $L(v)$ and $L(w)$ over the common vertex $e$.
\end{obs}
\begin{proof}
 For $\psi$ take the following bijection: given a face $f$ traversing $e$, this traversal corresponds to a unique edge in $L(v)$ and $L(w)$, and these two edges are in bijection via $\psi$.
\end{proof}

\begin{lem}\label{sum_planar1}
 Let $G$ be a graph that is a vertex sum of two graphs $H_1$ and $H_2$ over the common vertex $v$.
Let $(\sigma_x^i|x\in V(H_i))$ be a planar rotation system of $H_i$ for $i=1,2$ such that
$\sigma_v^1$ is the reverse of $\sigma_v^2$. Then  $(\sigma_x^i|x\in V(H_i)-v, i=1,2)$ is a planar
rotation system of $G$.
\end{lem}

\begin{proof}[Proof sketch.] This is a consequence of the topological fact that the connected sum
of two spheres is the sphere.
\end{proof}

\begin{lem}\label{sum_planar2}
 Let $G$ be a graph that is a vertex sum of two graphs $H_1$ and $H_2$ over the common vertex $v$.
 Assume that the vertex $v$ is not a cutvertex of $H_1$ or $H_2$.
 Assume that $G$ has a planar rotation system $\Sigma$. Then there are planar rotation systems of
$H_1$ and $H_2$ that agree with $\Sigma$ at the vertices in $V(G)\cap V(H_i)$ and that are reverse
at $v$.
\end{lem}

\begin{proof}
Since the vertex $v$ is not a cutvertex of the graph $H_2$, the graph $H_1$ can be obtained from
the graph $G$ by contracting the connected vertex set $V(H_2)-v$ onto a single
vertex.
Now let a plane embedding $\iota$ of $G$ be given that is induced by the rotation system $\Sigma$.
Since contractions can be performed within the plane embedding $\iota$, there is a planar rotation
system $\Sigma_1$ of the graph $H_1$ that agrees with $\Sigma$ at all
vertices in $V(H_1)-v$.

Since the vertex $v$ is not a cutvertex of $H_1$ or $H_2$, the cut $X$ of $G$ consisting of
the edges between $V(H_1)-v$ and $V(H_2)-v$ is actually a bond\footnote{Given a bipartition of the vertex set of a graph, its \emph{cut} is the set of edges between these two biparition classes. A \emph{bond} is a minimal nonempty cut. It is easy to see that a cut is a bond if and only if both its paritition classes are connected vertex sets.} of the graph $G$.
The bond $X$ is a circuit $o$ of the dual graph of $G$ with respect to the embedding $\iota$; and
the rotator at $v$ of the planar rotation system $\Sigma_1$ is equal (up to reversing) to the cyclic orientation
of the edges on the circuit $o$.
Similarly, we construct a planar rotation system  $\Sigma_2$ of $H_2$ that agrees with $\Sigma$ at
all
vertices in $V(H_2)-v$, and the rotator at the vertex $v$ is the other orientation of the circuit
$o$. This completes the proof.
\end{proof}

Let $C$ be a 2-complex and $e$ be a non-loop edge of $C$, and $\Sigma=(\sigma_x|x\in E(C))$ be a
rotation system of $C$. The \emph{induced} rotation system of $C/e$ is $\Sigma_e=(\sigma_x|x\in
E(C)-x)$. This is well-defined as the incidence relation between edges of $C/e$ and faces is the
same as in $C$.
Planarity of rotation systems is preserved under contractions:

\begin{lem}\label{contr_pres_planar}
Let $e$ be a nonloop edge.
If a rotation system $\Sigma$ is planar, then its induced rotation system $\Sigma_e$ is planar.

Conversely, for any planar rotation system $\Sigma'$ of $C/e$ so that $e$ is not
a cutvertex of any of the two link graphs at its endvertices in $C$, there is a planar rotation system of
$C$
inducing $\Sigma'$.
\end{lem}

\begin{proof}
If a rotation system $\Sigma$ is planar, then by \autoref{obs1} and \autoref{sum_planar1} its induced rotation system $\Sigma_e$ is planar.
Conversely, for any planar rotation system $\Sigma'$ of $C/e$ so that $e=vw$ is not
a cutvertex of any of the two link graphs at its endvertices in $C$,
by \autoref{sum_planar2} applied to the planar rotation system $\Sigma'$ induces on the link graph $L(e)$ in $C/e$,
there are planar rotation systems on the link graphs $L(v)$ and $L(w)$ that are compatible along the edge $e$ and equal to $\Sigma'$ at other edges, and so $\Sigma'$ together with the rotator at $e$ obtained from \autoref{sum_planar2} defines a planar rotation system of $C$ inducing $\Sigma'$.
\end{proof}

\begin{lem}\label{sum_3con}
 Let $G$ be a graph that is a vertex sum of two graphs $H_1$ and $H_2$ over the common vertex $v$.
 Let $k\geq 2$.
If $H_1$ and $H_2$ are $k$-connected\footnote{Given $k\geq 2$, a graph with at least $k+1$
vertices is \emph{$k$-connected} if the removal of less than $k$ vertices does not make it
disconnected. Moreover it is not allowed to have loops and if $k>2$, then it is not allowed to
have parallel edges. Whilst $K_2$ is not 2-connected, every graph with two vertices, no loops and at least two parallel edges is 2-connected, all other graphs on less than $k+1$ vertices are not $k$-connected for every $k$.}, then so is $G$.
\end{lem}

\begin{proof}
Let $Y$ be the set of edges incident with $v$ (suppressing the bijection between the edges incident
with $v$ in $H_1$ and $H_2$ in our notation). We shall also consider $Y$ as an edge set of $G$.
\begin{sublem}\label{sublem0099}
 If $Y$ does not contain a matching of size $k$, then $G$ is $k$-connected.
\end{sublem}
\begin{cproof}
As $H_1$ is $k$-connected, the set $Y$ contains at
least $k$ edges. If $k>2$, then since no $H_i$ has parallel edges, in the graph $G$ no two edges in $Y$ share a
vertex. So we may assume that $k=2$. If in $G$ the set $Y$ does not contain a matching of size two, there is a vertex $x$ covering $Y$. By symmetry assume $x\in V(H_1)$. Then all edges of $Y$ are in parallel between the vertices $x$ and $v$. Since $x$ is not a cutvertex of $H_1$, the graph $H_1$ can only have the vertices $x$ and $v$. So $G=H_2$ by the definition of vertex-sum. So $G$ is 2-connected.
\end{cproof}

By \autoref{sublem0099}, assume that $Y$ contains a set $Y'$ of $k$ edges that form a matching in the graph $G$.
 Suppose for a contradiction that there is a set of less than $k$ vertices of $G$ such that $G\sm
X$
is disconnected.
Hence by the pigeonhole principle, there is an edge $e$ in $Y'$ such that no endvertex of $e$ is in
$X$.  Let $C$ be the component of $G\sm X$
that contains $e$. Let $C'$ be a different component of $G\sm X$. Let $i$ be such that $H_i$
contains a vertex $w$ of $C'$.

In $H_i$ this vertex $w$ and an endvertex of $e$ are separated by $X+v$. As $H_i$ is $k$-connected,
we deduce that all vertices of $X$ are in $H_i$. Then the connected graph $H_{i+1}$ is a subset of
$C$. Hence the vertex $w$ and an endvertex of $e$ are separated by $X$ in $H_i$. This is a
contradiction to the assumption that $H_i$ is $k$-connected.
\end{proof}

\section{Obstructions to the existence of planar rotation systems}\label{s_constr}

In this section we prove \autoref{rot_system_exists-simplycon} below, which is used in the proof of \autoref{kura_intro}.
Intuitively, this lemma characterises simply connected locally 3-connected 2-complexes admitting planar rotation systems in terms of three obstruction; these obstructions to the existence of a planar rotation system are firstly a non-planar link graph, secondly an edge $e$ such that after contraction of $e$ the link graph at the contraction vertex is non-planar, and finally an obstruction spread around the boundary of a face, roughly speaking.
We start with some preparation. Recall that 3-connected graphs have no parallel edges.

\begin{lem}\label{locally_at_edge}
Let $C$ be a 2-complex with an edge $e$ with endvertices $v$ and $w$.
Assume that the link graphs $L(v)$ and $L(w)$ are 3-connected and that the link
graph
$L(e)$ of $C/e$ at $e$ is planar.
Then for any two planar rotation systems of $L(v)$ and $L(w)$ the rotators at $e$ are reverse of
one
another or agree.
\end{lem}

\begin{proof}
Let $\Sigma=(\sigma_x|x\in (L(v)\cup L(w))-e)$ be a planar rotation system of $L(e)$. By
\autoref{sum_planar2} there is a rotator $\tau_e$ at $e$ such that $(\sigma_x|x\in L(v)-e)$
together
with $\tau_e$ is a
planar rotation system of $L(v)$ and $(\sigma_x|x\in L(w)-e)$ together with the reverse of $\tau_e$
is a planar rotation system of $L(w)$.

Recall that Whitney proved that a 3-connected planar graph admits a planar rotation system that is unique up to reversing all rotators \cite{{whitney1992congruent},{Whitney_flip}}.
Since $L(v)$ and $L(w)$ are 3-connected, their planar rotation system are unique up to reversing
and hence the lemma follows.
\end{proof}

A \emph{rotation framework} of a 2-complex $C$ is a choice of planar rotation system at each link graph of $C$ such that for every edge $e=vw$
the two rotators at $e$, considered as a vertex in the link graphs $L(v)$ and $L(w)$, are the same or reverse of one another.
Given a 2-complex $C$ with a rotation framework. We colour an edge $e$ of $C$ \emph{green} (with respect to a rotation framework) if the two rotators at $e$ are reverse; otherwise we colour it \emph{red}.
\begin{eg} Edges incident with just two faces are always coloured green.\end{eg}

A rotation framework of a 2-complex $C$ is \emph{even} if on every cycle of $C$ the number of red edges is even; otherwise it is \emph{odd}.

\begin{lem}\label{lem101}
Assume that every edge is incident with at least three faces.
 A 2-complex admits a planar rotation system if and only if it admits an even rotation framework.
\end{lem}

\begin{proof}
Given a planar rotation system of $C$, the planar rotation systems it induces at the link graphs of $C$ form a rotation framework. All edges of $C$ are coloured green, so this rotation framework is even.

Conversely, let an even rotation framework of $C$ be given. Reversing the planar rotation system at a single vertex $v$ yields a rotation framework where the colours at all edges incident with $v$ swap -- since all vertices of $L(v)$ have degree at least three by assumption.
So this new rotation framework inherits the property that it is even. We refer to this operation as a \emph{flip}.
Now assume that $C$ is connected and pick a rooted spanning tree $T$ of (the 1-skeleton of) $C$. Each vertex of $T$ aside from the root has a unique edge to it from the direction of the root.
Starting at the neighbours of the root, we recursively do flips if necessary to ensure that the unique path in $T$ from the root to any vertex contains only green edges. So $C$ has an even rotation framework so that all edges of $T$ are green. Each edge outside $T$ forms a fundamental circuit with edges of $T$, and by evenness also must be green. So $C$ has a rotation framework where all edges are green.
This defines a planar rotation system as follows: for every directed edge $\vec{e}$ of $C$ directed to $v$, denote by $\sigma_e$ the rotator in $L(v)$ at the vertex $e$. Now $(\sigma_e| e\in E(C))$ is a planar rotation system.
\end{proof}

\begin{lem}\label{lem102}\label{rot_system_exists}
A locally 3-connected 2-complex that does not admit a planar rotation system has a vertex $v$ such that its link graph $L(v)$ is non-planar, has an edge $e$ such that in $C/e$ the link graph $L(e)$ at the contraction vertex $e$ is non-planar or $C$ admits a rotation framework so that there is a cycle containing an odd number of red edges.
\end{lem}

\begin{proof}
Assume that all link graphs of $C$ at vertices are planar and that also for every edge $e$ the link graphs $L(e)$ are planar.
So for each vertex $v$ of $C$ we can pick a planar rotation system of the link graph $L(v)$. By \autoref{locally_at_edge}, this defines a rotation framework.
Since $C$ does not have a planar rotation system, by \autoref{lem101} this rotation framework must be odd.
\end{proof}

\begin{rem}
If a 2-complex satisfies one of the three outcomes of \autoref{rot_system_exists}, then it does not admit a planar rotation system. So \autoref{rot_system_exists} is a characterisation of 2-complexes admitting planar rotation systems. We will not use this remark in our proofs.
\end{rem}

The next lemma is like \autoref{rot_system_exists} under the additional assumption of simply connectedness and this makes us find the cycle with an odd number of red edges as a face boundary.

\begin{lem}\label{rot_system_exists-simplycon}
A simply connected locally 3-connected 2-complex that does not admit a planar rotation system has a vertex $v$ such that its link graph $L(v)$ is non-planar, has an edge $e$ such that in $C/e$ the link graph $L(e)$ at the contraction vertex $e$ is non-planar or $C$ admits a rotation framework so that there is a face boundary containing an odd number of red edges.
\end{lem}

\begin{proof}
 By simply connectedness the face boundaries generate all cycles over the field $\Fbb_2$. The \emph{red parity} of an edge set is the parity of its number of red edges. Red parity is \emph{additive}; that is, the red parity of the sum (of two edge sets over $\Fbb_2$) is equal to the sum of the red parities of the summands. So take a cycle contains with odd red parity and write it as a sum of face boundaries. One of these summands must have odd parity.
 So this lemma follows from \autoref{rot_system_exists}.
\end{proof}

\section{Space restrictions}\label{restricted_section}

In this section we define space restrictions and prove \autoref{kura_intro_hom}.

Let $C$ be a 2-complex.
Given a face $f$ of $C$, we denote by $C-f$ the 2-complex obtained from $C$ by deleting the face $f$.
Given an edge $e$ of $C$, the 2-complex\footnote{Recall that the boundary of a 2-complex is a closed trail and as such contains each edge at most once.} obtained from $C$ by \emph{(topologically) deleting the edge
$e$} is obtained from $C$ by replacing $e$ by parallel edges, one for each face incident with $e$ in such a way that
the incidence of $e$ with a face is inherited by the copy corresponding to that face; for an example see the deletion of the edge $g$ in
\autoref{fig:space_minor2}).

\begin{rem}
 We think about \lq {deleting an edge}\rq\ topologically; that is, in the geometric realisation of a simplicial complex we delete the points on that edge. The resulting topological space would not be a simplicial complex as the faces incident with the deleted edge would no longer be bounded by cycles. Thus we add a private edge to each such face to repair this defect.
\end{rem}

Given a vertex $v$, the \emph{splitting} of $C$ at $v$ is obtained by replacing $v$ by one copy for each component of the link graph $L(v)$ such that each incidence of $v$ is inherited by the copy whose link graph contains the face or edge as an edge or vertex, respectively. We refer to these copies of $v$ as \emph{splitters}.

A \emph{space restriction} of $C$ is a 2-complex obtained from $C$ by deleting faces, edges and isolated\footnote{A vertex is \emph{isolated} if it has no incidences.} vertices and splitting vertices in an arbitrary order.

\begin{rem}
 We regard deleting an isolated vertex as a special case of splitting that vertex. Deleting an isolated\footnote{An edge is \emph{isolated} if it is not incident with a face.} edge makes this edge disappear.
 Deleting edges and faces commutes.

 In the context of locally connected 2-complexes, it is natural to require that after deletion of faces or edges one always splits all vertices, then these enhanced operations still commute, and there would be only two space restriction operations.
\end{rem}

\begin{obs}
 Space restrictions of 2-complexes embeddable in 3-space are embeddable in 3-space.\qed
\end{obs}

The next lemma gives a condition when topologically deleting edges can be combined with splitting vertices so that no parallel edges remain.

\begin{lem}\label{split-edge_remain_simplicial}
 Let $C$ be a simplicial complex and $W\se V(C)$ a vertex set. Let $X$ be the set of those vertices of $V(C)\sm W$ not contained in a face containing two vertices of $W$.
 Let $C'$ be the 2-complex obtained from $C$ by topologically deleting all edges with an endvertex in $X$ and no endvertex in $W$ and then splitting all vertices in $X$.

 Then $C'$ is a simplicial complex and the link graph at splitters $x'$ of vertices in $X$ is a nontrivial\footnote{A star is \emph{nontrivial} if it contains at least one edge.} star.
 Moreover, every splitter $x'$ that is in $N(W)$ sends a unique edge to $W$ and all other edges incident with $x'$ are incident with a single face.
\end{lem}

\begin{proof}
 Let $x$ be a vertex of $X$. The link graph at $x$ after topologically deleting the edges as described (but not yet split at $x$), is a disjoint union of stars, and all those stars that are not single edges are centred at a vertex whose corresponding edge has an endvertex in $W$.
 Recall that since $C$ is simplicial, no two of its faces share two edges. Hence for each edge incident with $x$ that we topologically delete, all its copies correspond to vertices in $L(x)$ that are in different components.

 Thus $C'$ has no parallel edges and thus is a simplicial complex. Above we have shown that the link graphs at splitters are as desired.
\end{proof}

\subsection{Cones}

The goal of this subsection is to find in a simplicial complex with a non-planar link graph a cone over a subdivision of $K_5$ or $K_{3,3}$ as a space restriction, compare \autoref{cone-final}.
Given a graph $G$ without loops, the \emph{cone} over $G$ is the 2-complex whose vertex set is $V(G)$ together with one additional vertex, to which we refer to as the \emph{top}. Its edge set is $E(G)$ together with one edge for every vertex $v$ of $G$ joining $v$ with the top. Its set of faces consists of one face for every edge $e$ of $G$ that is incident $e$ and the two edges from the endvertices of $e$ to the top. Note that if $G$ has no parallel edges, then its cone is a simplicial complex.

\begin{lem}\label{is:cone}
 Let $D$ be a simplicial complex with a vertex $v$. Assume that each edge not incident with $v$ is in a single face, and that this face contains $v$, and that the link graph at every vertex except for $v$ is a star whose centre corresponds to an edge from that vertex to $v$.
 Then $D$ is equal to the cone over the link graph at $v$.
\end{lem}
\begin{proof}
 $D$ has the same vertex set as the cone over $L(v)$, and the same sets of edges and faces.
\end{proof}

\begin{lem}\label{get-cone}
 Let $C$ be a simplicial complex with a vertex $v$.
 Then $C$ has a space restriction that is the cone over $L(v)$.
\end{lem}

\begin{proof}
 First delete all faces that do not contain the vertex $v$.
 Then apply \autoref{split-edge_remain_simplicial} with $W=\{v\}$ to the resulting simplicial complex. We conclude by \autoref{is:cone}.
\end{proof}

\begin{obs}\label{del-cone}
 Let $C$ be a cone over the graph $G$.
 For every edge $f$ of $G$, the cone over $G-f$ is a space restriction of $C$.
\end{obs}

\begin{proof}
 Delete the face corresponding to the edge $f$.
\end{proof}

\begin{cor}\label{expli-cone}
 Let $C$ be a cone over a non-planar graph $G$.
 Then $C$ has a space restriction that is a cone over a subdivision of $K_5$ or $K_{3,3}$.
\end{cor}

\begin{proof}
 By Kuratowski's theorem $G$ has a subgraph that is a subdivision of $K_5$ or $K_{3,3}$.
 Now apply \autoref{del-cone}.
\end{proof}

\begin{lem}\label{cone-final}
  Let $C$ be a simplicial complex with a vertex $v$ whose link graph is non-planar.
 Then $C$ has a space restriction that is the cone over a subdivision of $K_5$ or $K_{3,3}$.
\end{lem}

\begin{proof}
 Combine \autoref{get-cone} with \autoref{expli-cone}.
\end{proof}

\subsection{Combined cones}

This subsection is structured similar to the previous subsection with a few bells and whistles attached, and its goal is to find in every simplicial complex $C$ with an edge $e$ so that in $C/e$ the link graph $L(e)$ at the contraction vertex $e$ is non-planar, a space restriction similar to that of a cone over $K_5$ or $K_{3,3}$. To formalise this we introduce \lq combined cones\rq. This subsection culminates in \autoref{combined-cone-summary} below.

Given two 2-complexes $C$ and $D$ together with a bijection $f$ between two subsets of their vertices, the 2-complex obtained from $C$ by \emph{simplicial gluing} $D$ via $f$ is obtained from the disjoint union of $C$ and $D$ by identifying vertices via the bijection $f$, as well as edges and faces all of whose vertices are in the image of $f$ and that after the identification of the vertices would have the same vertex-sets.  Note that gluings of simplicial complexes are simplicial complexes.
Given a vertex-sum $G=G_1\oplus_e G_2$, the \emph{simplicial combined cone} over this vertex-sum is obtained from the cone over $G_1$ by gluing the cone over $G_2$ via the vertex set consisting of the vertices in faces that are incident with the edge $e$; here we denote the unique vertex in a face $f$ not incident with $e$ by $v_f$. The gluing bijection identifies the vertices $v_f$ from both cones and identifies the endvertices of $e$ in such a way that the endvertex of $e$ that is a top in one cone is identified with the endvertex of $e$ from the other cone that is not the top there.

\begin{eg}\label{eg:alternative-construction}
 Here we provide an alternative construction of the simplicial cone over $G=G_1\oplus_e G_2$. Let $b$ denote the cut in $G$ that separates the vertex sets of $G_1$ and $G_2$.
 Let the vertex set of $C$ consist of the endvertices of $e$ together with the vertex set of $G/b$; that is, the graph obtained from $G$ by contracting the edges of $b$.
 The edge set of $C$ consists of $e$ together with an edge for every vertex of $G$ from its contraction vertex in $G/b$ to the endvertex of $e$ corresponding to the graph $G_i$ that contains that vertex. We also add all edges of the graph $G/b$. We add a face for every edge of $G$. If the edge is in $b$, its boundary consists of the edge $e$ and the two edges to the contraction vertex of $G/b$ of that edge. If the edge is not in $b$, it is an edge of $G/b$ and the boundary of the face is that edge together with the two edges to the graph $G_i$ that are incident with that edge.

 We remark that the construction of $C$ only relies on the graph $G$ and the cut $b$ but not on the gluing isomorphism. Since this is an alternative definition of the simplicial combined cone, it follows that two simplicial combined cones are isomorphic if they have the same graph $G$ and within the same cut $b$.
\end{eg}

\begin{lem}\label{is:cone-combined}
 Let $D$ be a simplicial complex with an edge $e=vw$ so that all its faces and edges contain $v$ or $w$ and all its vertices are adjacent to $v$ or $w$, and if they are adjacent to both, then they are in common face with both of them.
  Assume that for each $u\in \{v,w\}$ the simplicial complex $D$ restricted to the vertex set consisting of $u$ and its neighbours and the faces and edges containing containing $u$ is equal to the cone over $L(u)$.

 Then $D$ is the simplicial combined cone over $L(v)\oplus_e L(w)$.
 \end{lem}
\begin{proof}
 $D$ has the same vertex set as the simplicial combined cone over $L(v)\oplus_e L(w)$, and the same sets of edges and faces.
\end{proof}

\begin{lem}\label{get-combined cone}
 Let $C$ be a simplicial complex with an edge $e=vw$.
 Then $C$ has a space restriction that is a simplicial combined cone over the vertex-sum $L(e)=L(v)\oplus_e L(w)$.
\end{lem}

\begin{proof}
 First delete all faces that do not contain an endvertex of the edge $e$, and all vertices and edges that afterwards are not in a face.
 Then apply \autoref{split-edge_remain_simplicial} with $W=\{v,w\}$ to the resulting simplicial complex.
 This lemma gives a space restriction $D$, which is a simplicial complex. Let $X$ be the set of vertices of $D$ that are not in a face with the edge $e$.
 By that lemma the link graph at every vertex $x$ of $X$ is a star, whose centre is an edge of $D$ from $x$ to $v$ or $w$.

 We obtain $D_v$ from $D$ by restricting to the vertex $v$ and its neighbours and the faces and edges containing $v$.

 \begin{sublem}\label{calc123}
  $D_v$ is the cone over $L(v)$.
 \end{sublem}
\begin{cproof}
Let $y$ be a vertex of $D_v$ other than $v$. By definition of $D_v$, the vertex $y$ is a neighbour of $v$. If $y\in X$ its link graph is a star, whose centre is an edge of $D_v$ from $y$ to $v$ or $w$.
Now let $y$ be outside $X$. By construction of $D_v$, all faces containing $y$ also contain the vertex $v$ and thus the edge $yv$. Since $D_v$ is a simplicial complex, the link graph at $y$ is a star centred at the vertex corresponding to the edge $yv$. It follows that every edge not incident with $v$ has degree one in a link graph and thus is in a single face.
So by \autoref{is:cone} $D_v$ is the cone over $L(v)$.
\end{cproof}

We define \lq $D_w$\rq\ like \lq $D_v$\rq\ with \lq $w$\rq\ in place of \lq $v$\rq. Like \autoref{calc123}, one proves that $D_w=L(w)$.
By \autoref{is:cone-combined} $D$ is the simplicial combined cone over $L(v)\oplus_e L(w)$.
\end{proof}

\begin{obs}\label{del-combined}
 Let $C$ be a simplicial combined cone over the vertex-sum $G=G_1\oplus_e G_2$.
 For every edge $f$ of $E(G_1)\sm E(G_2)$ or of $E(G_1)\cap E(G_2)$,
 the simplicial combined cone $C$ has the simplicial combined cone over $G-f=(G_1-f)\oplus_e G_2$ or $G-f=(G_1-f)\oplus_e (G_2-f)$, respectively, as a space restriction.
\end{obs}

\begin{proof}
 Delete the face corresponding to the edge $f$ in the simplicial combined cone over $G=G_1\oplus_e G_2$, and resulting edges or vertices not incident with any face. Since deletion of $f$ commutes with the gluing, we obtain a simplicial combined cone over $G-f$.
\end{proof}

\begin{cor}\label{expli}
 Let $C$ be a simplicial combined cone over the vertex-sum $G=G_1\oplus_e G_2$ such that $G$ is non-planar but $G_1$ and $G_2$ are planar.
 Then $C$ has a space restriction that is the simplicial combined cone over a vertex-sum $H=H_1\oplus_e H_2$,
 where $H$ is a subdivision of $K_5$ or $K_{3,3}$ and $H_1$ and $H_3$ have both at least three vertices of degree at least three.
\end{cor}

\begin{proof}
 By Kuratowski's theorem $G$ has a subdivision of $K_5$ or $K_{3,3}$. So via \autoref{del-combined} and by taking a suitable space restriction if necessary, assume that $G$ is a subdivision of one of these two graphs.
 Now consider the edges of $G$ that correspond to edges of $G_1$ that in there are incident to the vertex-sum vertex $e$. By the definition of vertex sum, these edges form a cut in $G$. If this cut contains only two edges or has only one vertex of degree at least three on one side, then $G_1$ or $G_2$ has one of the graphs $K_5$ or $K_{3,3}$ as a minor (indeed, these 3-connected minors live entirely one side of such a small cut), and thus is non-planar. So we may assume that each side of the cut contains at least two vertices of degree three, and that $e$ has degree at least three in $G_1$ and $G_2$.
 So $G_1$ and $G_2$ both have at least three vertices of degree at least three.
\end{proof}

\begin{lem}\label{expli-computation}
 Up to subdivision, there are exactly five vertex-sums as in \autoref{expli}, one for $K_5$ and four for $K_{3,3}$.
\end{lem}

\begin{proof}
 By \autoref{eg:alternative-construction} the simplicial combined cones over $G=G_1\oplus_e G_2$ are uniquely determined by the graph $G$ and the cut $b$ between the vertex sets of $G_1$ and $G_2$. While $K_5$ has only one such cut with at least two vertices on either side, the number of non-atomic nonempty cuts in $K_{3,3}$ is four: there are two cuts with 2 vertices versus 4 vertices and there are two cuts with 3 vertices versus 3 vertices.
\end{proof}

\begin{lem}\label{combined-cone-summary}
Let $C$ be a simplicial complex so that there is an edge $e=vw$ so that in $C/e$ the link graph $L(e)$ at the contraction vertex $e$ is non-planar but the link graphs at $v$ and $w$ in $C$ are planar.
Then $C$ has a space restriction that is a simplicial combined cone over from one of the five classes in \autoref{expli-computation}.
\end{lem}

\begin{proof}
 Combine \autoref{get-combined cone} with \autoref{expli}. Then use the explicit description of \autoref{expli-computation}.
\end{proof}

\subsection{M\"obius obstructions}

We call a face $f$ of a 2-complex $C$ \emph{red} if $C$ admits a rotation framework so that the boundary of $f$ contains an odd number of red edges; this is the last outcome of \autoref{rot_system_exists-simplycon}, and in the previous two subsection we dealt with the other two outcomes of that lemma.
In this section we deal with this final outcome.  This subsection culminates in  \autoref{moebius-summary} below.

A \emph{M\"obius framework} of a simplicial complex $C$ consists of a face $f$ of $C$ with vertices $v_1$, $v_2$, $v_3$
and a collection of six paths
$(P_i|i\in [6])$ such that $P_i$ and $P_{i+3}$ are vertex-disjoint paths in the link graph at $v_i$ for $i\in [3]$ and do not contain any of the vertices corresponding to edges of $f$, and so that the endvertex of $P_i$ and the start vertex of $P_{i+1}$ correspond to edges that together with the edge $v_iv_{i+1}$ form a face of $C$. We refer to the six faces obtained in this way as the \emph{witnesses} of the M\"obius framework. We denote M\"obius frameworks by tuples $(f,v_1,v_2,v_3, (P_i|i\in [6]))$.

\begin{lem}\label{bad_moebius}
 A locally 3-connected simplicial complex $C$ with a red face $f$ admits a M\"obius framework.
\end{lem}

\begin{proof}For $i\in [3]$, let $v_i$ be one of the three vertices of the face $f$.
By assumption the link graph $L(v_i)$ at $v_i$ is 3-connected. In there the vertices corresponding to the edges $v_iv_{i+1}$ and $v_iv_{i-1}$ are joined by an edge corresponding to $f$; recall that we denote this edge by $f$. By assumption the link graph $L(v_i)$ is planar and pick a plane embedding.
Let $D$ and $D'$ be the two face boundaries of that embedding that include the edge $f$. Since $L(v_i)$ is 2-connected, the sets $D$ and $D'$ are cycles \cite{MoharThomassen}.
We obtain the two paths $P_i$ and $P_{i+3}$ from $D$ and $D'$, respectively, by deleting the two endvertices of $f$.
Any vertex in the intersection of $P_i$ and $P_{i+3}$ together with the edge $f$ separates the graph $L(v_i)$. Since this graph is 3-connected by assumption, this is not possible, so $P_i$ and $P_{i+3}$ and are vertex-disjoint.

By construction the endvertices of these two paths correspond to edges of $C$ that together with one of the edges $v_iv_{i\pm 1}$ lie in a common face. We pick a direction on these paths so that the start vertex lies in a common face with $v_iv_{i-1}$, whilst the endvertex lies in a common face with $v_iv_{i+ 1}$. These common faces are adjacent to $f$ in the rotators at $v_iv_{i-1}$ and $v_iv_{i+ 1}$, respectively.

Now assume that $C$ admits a rotation framework. Then in the link graphs $L(v_i)$ and $L(v_{i+1})$ at the rotator at the vertex $v_iv_{i+1}$ the edge $f$ has the same two neighbours.
So by exchanging the roles of $P_2$ and $P_5$ if necessary, we ensure that the endvertex of $P_1$ and the start vertex of $P_2$ correspond to edges that together with the edge $v_1v_2$ form a face of $C$. Similarly, we ensure that
the endvertex of $P_2$ and the start vertex of $P_3$ correspond to edges that together with the edge $v_2v_3$ form a face of $C$.

Given $i\in \Zbb_6$, denote by $D_i$ the disc of the embedding of the link graph $L(v_{j})$ with $j=i \ mod \ 3$ that has the path $P_i$ and the edge $f$ in its boundary (by 3-connectivity this disc is uniquely determined and referred to $D$ or $D'$ in the above construction).

Recall that $C$ is a directed 2-complex and so the face $f$ is equipped with an orientation. When we consider the face $f$ as an edge in link graphs, this orientation induces a direction of the edge $f$. So the \emph{left side} and the \emph{right side} of the edge $f$ are well-defined.

\begin{sublem}\label{new99}
In embeddings of the link graphs $L(v_i)$ and $L(v_{i+1})$ the discs $D_i$ and $D_{i+1}$ are on the same side of the edge $f$ if and only if the rotators at the vertex $v_iv_{i+1}$ are reverse.
\end{sublem}
\begin{cproof}
The edge $f$ is incident with the vertex $v_iv_{i+1}$ in both link graphs $L(v_i)$ and $L(v_{i+1})$.
We have already constructed a face so that its edges in the link graphs  $L(v_i)$ and $L(v_{i+1})$ are incident with the vertex $v_iv_{i+1}$ and in both discs $D_i$ and $D_{i+1}$; we denote this face by $g$.
So $g$ -- considered as an edge of the link graphs $L_(v_i)$ and $L(v_{i+1})$ -- is on the same side of $f$ if and only if $D_i$ and $D_{i+1}$ are on the same side of $f$.
So the rotators at the vertex $v_iv_{i+1}$ determine on which side of $f$ the edge $g$ is.
\end{cproof}

Now assume that $C$ admits a rotation framework $\Sigma$ so that the boundary of $f$ contains an odd number of red edges.
Applying \autoref{new99} recursively at all edges of $f$ and using the fact that $f$ contains an odd number of red edges implies that the endvertex of $P_3$ and the start vertex of $P_4$ correspond to edges that together with the edge $v_1v_2$ form a face of $C$. A similar statement is true for $P_1$ and $P_6$. It follows that $(f,v_1,v_2,v_3, (P_i|i\in [6]))$ is a M\"obius framework.
\end{proof}

\begin{lem}\label{moe-reverse}
In a simplicial complex $C$ with a M\"obius framework, every rotation framework has a red face.
\end{lem}

\begin{proof}
 Let $(f,v_1,v_2,v_3, (P_i|i\in [6]))$ be a M\"obius framework of $C$. We claim that $f$ is red with respect to every rotation framework of $C$.
 Suppose for a contradiction that $C$ admits a rotation framework where $f$ is not red. Then we can flip the embeddings at the link graphs at vertices of $f$ so that all three edges of $f$ are green.
 By symmetry assume that in the embedding of the link graph $L(v_1)$ the path $P_1$ is on the left of $f$, and $P_4$ is on its right. Since the edge $v_1v_2$ is green, $P_2$ is on the left of $f$ in the embedding of $L(v_2)$. Since $v_2v_3$ is green, $P_3$ is on the left of $f$. Since $v_3v_1$ is green, we conclude that $P_4$ is on the left of $f$, a contradiction to our assumption.
\end{proof}

A triangulation of the M\"obius-strip is \emph{nice} if it has a central cycle of length three and all edges of face degree two are on the central cycle or have exactly one endvertex on the central cycle.
A \emph{M\"obius obstruction} is obtained from a nice triangulation of the M\"obiusstrip by attaching a face at the central cycle. A M\"obius obstruction is an example of a torus crossing obstruction and as such does not embed in 3-space, see \autoref{chapterV} for a proof.

\begin{lem}\label{pre-prs-moe}
 A M\"obius obstruction admits a M\"obius framework.
\end{lem}

\begin{proof}
 Take the face bounded by the central cycle as the face $f$ of the M\"obius framework and the link graphs at its endvertices are parallel graphs\footnote{A \emph{parallel graph} consists of two vertices, called the \emph{branch vertices}, and a set of
disjoint paths between them. Put another way, start with a graph with only two vertices and all
edges going between these two vertices, now subdivide these edges arbitrarily,
see \autoref{fig:para5}.} whose branching vertices have degree three and take the unique choices possible for the paths $P_i$. This defines a M\"obius framework.
\end{proof}

\begin{cor}\label{prs-moe}
 A M\"obius obstruction admits no planar rotation system.
\end{cor}

\begin{proof}
By \autoref{pre-prs-moe} this follows from \autoref{moe-reverse}.
\end{proof}

\begin{lem}\label{pre-basic-topo}[\cite{armstrong}]
Let  $S$ be a connected surface with boundary that admits a triangulation $T$ so that the Euler-characteristic $V-E-F$ is zero.
If $S$ has a single boundary component, it is the M\"obius strip. If it has two boundary components it is an untwisted strip (so homeomorphic to a disc with a hole).
\end{lem}

\begin{cor}\label{basic-topo}
Let  $S$ be a connected surface with nontrivial boundary\footnote{Recall that boundaries of 2-manifolds with boundary are always 1-dimensional, so homeomorphic to $\Sbb^1$.} that admits a triangulation $T$ so that there is a cycle $o$ such that  all edges of face degree two in $T$ are on $o$ or have exactly one endvertex on $o$, and no point on $o$ is on the boundary. Then $S$ is homeomorphic to a M\"obius-strip or an untwisted strip.
\end{cor}

\begin{proof}
 $S\sm o$ has one or two components. So $S$ has at most two boundary components, and at least one by assumption. So by \autoref{pre-basic-topo}, it suffices to show that $V-E+F=0$.
 All vertices lie on the central cycle $o$ or on one of the boundary cycles.
 We split the edges into two sets, those that lie on $o$ or on the boundary -- the number of them is equal to the number of vertices -- and the edges between $o$ and the boundary.
 All edges of the second type lie in exactly two faces and each face contains precisely two edges of type two.
 So $V-E+F=0$.
\end{proof}

\begin{lem}\label{moebius_restr}
 If $C$ has a M\"obius framework, it has a M\"obius obstruction as a space restriction.
\end{lem}

\begin{proof}
Let a M\"obius framework $(f,v_1,v_2,v_3, (P_i|i\in [6]))$ be given.
 We obtain $D$ from $C$ by deleting all faces except for $f$, the six witnesses of the M\"obius framework and those faces corresponding to edges of the paths $P_i$. Then we topologically delete all edges except for those of $f$ and those corresponding to vertices of the graphs $P_i$. We complete the definition of $D$ by splitting all vertices.

 We claim that $D$ is a M\"obius obstruction. For this, it remains to show that $D-f$ is a triangulation of a M\"obius-strip.
 Since $D-f$ is locally connected by construction, we deduce that $D-f$ is a triangulation of some connected surface with nontrivial boundary; note that splitting vertices ensures that the boundary is 1-dimensional.
 Let $o$ be the boundary cycle of $f$.
 As a small neighbourhood of $o$ includes a M\"obius-strip in $D-f$, the geometric realisation of $D-f$ is not an untwisted strip . So by \autoref{basic-topo}, $D-f$ is a triangulation of a M\"obius-strip.
\end{proof}

\begin{lem}\label{moebius-summary}
 If a locally 3-connected simplicial complex has a red face, then it has a M\"obius obstruction as a space restriction.
\end{lem}

\begin{proof}
 Combine \autoref{bad_moebius} with \autoref{moebius_restr}.
\end{proof}

\subsection{Proof of \autoref{kura_intro_hom} assuming \autoref{emb_to_rot}}

We denote by $\Zcal$ the following list:
\begin{enumerate}
 \item cones over subdivisions of $K_5$ or $K_{3,3}$;
 \item simplicial combined cones from the five families of \autoref{expli-computation};
 \item M\"obius obstructions.
\end{enumerate}

\begin{rem}
 All members of $\Zcal$ are simply connected simplicial complexes.
\end{rem}

\begin{thm}\label{subdiv-kura}\label{rot_minor}
Let $C$ be a simply connected locally 3-connected 2-dimensional simplicial complex.
Then $C$ has a planar rotation system if and only if $C$ has no space restriction from
the list $\Zcal$.
\end{thm}

\begin{proof}
First we verify that the finitely many members of $\Zcal$ do that admit a planar rotation system. Clearly, cones over subdivisions of $K_5$ or $K_{3,3}$ do not admit planar rotation systems.
If $C$ is a simplicial combined cone from $\Zcal$ then it has an edge $e$ so that the link graph $L(e)$ at the contraction vertex $e$ is not planar.
Since the existence of planar rotation system is preserved by contracting non-loop edges by \autoref{contr_pres_planar}, such $C$ also do not admit planar rotation system.
Finally a M\"obius obstruction does not admit a planar rotation systems by \autoref{prs-moe}.
Since having a planar rotation system is preserved under taking space restrictions, it follows that if $C$ has a space restriction from $\Zcal$, then it does not admit a planar rotation system.

Conversely, assume that $C$ has no planar rotation system. Our aim is to find a space restriction from $\Zcal$.
By \autoref{rot_system_exists-simplycon}, $C$ has a vertex $v$ such that its link graph $L(v)$ is non-planar, has an edge $e$ such that in $C/e$ the link graph $L(e)$ at the contraction vertex $e$ is non-planar or $C$ admits a rotation framework so that there is a face boundary containing an odd number of red edges.
In these three cases we find a member of $\Zcal$ by applying \autoref{cone-final}, \autoref{combined-cone-summary} and \autoref{moebius-summary}, respectively.
\end{proof}

\begin{proof}[Proof of \autoref{kura_intro}.]
By \autoref{emb_to_rot} a simply connected simplicial complex is
embeddable in \Sthree\
if and only if it has a planar rotation system. So \autoref{kura_intro} is implied by
\autoref{rot_minor}.
\end{proof}

\begin{proof}[Proof of \autoref{kura_intro_hom}.]
By \autoref{combi_intro_extended} a
simplicial complex
with $H_1(C,\Fbb_p)=0$ is embeddable if and
only if it is simply connected and it has a planar rotation system. So \autoref{kura_intro_hom} is
implied by
\autoref{rot_minor}.
\end{proof}

\section{Space minors}\label{sec:space}

In this section we introduce the space minor relation, prove a few basic properties and show that the subset of $\Zcal$ of its space minor-minimal elements is finite.

A \emph{space minor} of a 2-complex is obtained by successively performing one of the five
operations.
\begin{enumerate}
 \item contracting an edge that is not a loop;
 \item deleting a face (and all edges or vertices only incident with that face);
 \item contracting a face of size one\footnote{Although we do not need it in our proofs, it seems
natural to allow contractions of faces of size one. This simply removes this face from $C$ and removes its unique edge from all other faces.} or two if its two edges are not loops;
 \item splitting a vertex;
 \item topologically deleting an edge.
\end{enumerate}

\begin{rem}
 A little care is needed with contractions of faces. This can create faces traversing edges
multiple times. In this paper, however, we do not contract faces consisting of two
loops and we only perform these operations on 2-complexes whose faces have size at most three.
Hence it could only happen that after contraction some face traverses an edge twice but in
opposite direction. Since faces have size at most three, these traversals are adjacent. In this
case
we omit the two opposite traversals of the edge from the face. We delete faces incident with no
edge. This ensures that the class of 2-complexes with faces of size at most three is closed under
face contractions.
\end{rem}

A 2-complex is \emph{3-bounded} if all its faces are incident with at most three edges.
The closure
of the class of simplicial complexes by space minors is the class of 3-bounded 2-complexes.

It is easy to see that the space minor operations preserve
embeddability in \Sthree\ (or in any
other 3-dimensional manifold), see \autoref{rem:obv} below, and the first three commute when well-defined.\footnote{In order for the
contraction of a face to be defined we need the face to have at
most two edges. This may force contractions of edges to happen before the contraction of the
face.}

\begin{lem}\label{well-founded}
 The space minor relation is well-founded\footnote{A partial order is \emph{well-founded} if every strictly decreasing chain is finite. For example, a well-ordering is a well-founded total order.}.
\end{lem}
\begin{proof}
The \emph{face degree} of an edge $e$ is the number of faces incident with $e$.
We consider the sum $S$ of all face degrees ranging over all edges.
 None of the five above operations increases $S$. And 1, 2 and 3
always strictly decrease $S$. Hence we can apply 1, 2 or 3 only a bounded number of times.

Note that operations 4 and 5 do not create isolated vertices or edges. Thus after a finite number of steps, we may assume
that we do not attempt any further operations 4 or 5 to isolated vertices or edges.
Applying splitting to a vertex with connected link graph and applying topological deletion of an edge to an edge incident with a single face is the identity operation.
We refer to such applications of 4 and 5 as \emph{trivial}.

Since no operation increases the sizes of the faces, the
total number of vertices and edges incident with faces is bounded.
Operation 4, when not the identity, increases the number
of vertices and preserves the number of edges. For operation 5 it is the other way round.
Hence all operations of 4 and 5 can only be nontrivially applied a bounded number of
times.
\end{proof}

  \begin{figure} [htpb]
\begin{center}
   	  \includegraphics[height=2cm]{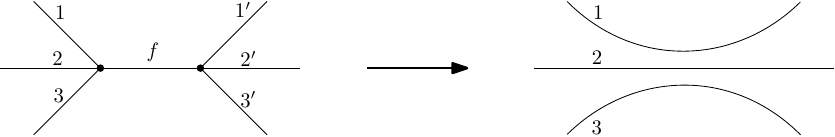}
   	  \caption{The operation that in the link graph corresponds to contracting a face $f$
only incident with a single edge $\ell$. The edge $\ell$ must be a loop. Hence in the link graph we
have two vertices for $\ell$ which are joined by the edge $f$. On the left we depicted that
configuration. Contracting $f$ in the complex yields the configuration on the right. Formally, we
delete $f$ and both its endvertices and add for each face $x$ of size at least
two  traversing $\ell$ an edge as follows. Before the contraction, the link graph contains two
edges corresponding to  the traversal of $x$ of $\ell$. These edges have precisely two distinct
endvertices that are not vertices corresponding to $\ell$. We add an edge between these two
vertices.
}\label{fig:contr_e}
\end{center}\vspace{-0.7cm}
   \end{figure}

\begin{lem}\label{rot_closed_down}
If a 2-complex $C$ has a planar rotation system, then all its space minors do.
\end{lem}

\begin{proof}
By \autoref{contr_pres_planar} existence of planar rotation systems is preserved by contracting
edges that are not loops.  Clearly the operations 2, 4 and 5 preserve planar rotation systems as
well. Since contracting a face of size two corresponds to locally in the link graph contracting the
corresponding edges, contracting faces of size two preserves planar rotation systems as noted after
\autoref{minimal_minor}. The operation that corresponds to contracting a face of size one in the
link graph is
explained in \autoref{fig:contr_e}.
   Now let $f$ be an edge in the link graph corresponding to a face of size one.
   Since the rotators at the two endvertices of $f$ are reverse of one another (by the definition of planar rotation systems of link graphs at vertices incident with loops), contracting $f$ as a face of $C$ preserves the embedding of the link graph at the unique vertex incident with $f$ in the plane. Thus contracting
a face of size one preserves planar rotation systems.
\end{proof}

\begin{rem}\label{rem:obv}
 For all space minor operations, a topological way to prove that they preserve embeddability is to remove from the embedding a small neighbourhood of the region where a single operation takes place, do the operation in there and glue it back on.
 However, there are also algebraic proofs, which perhaps are less intuitive but easier to verify. We illustrate this by formally proving that contracting a face of size one, contraction of an edge and splitting a vertex preserves embeddability.

 Firstly, let $C$ be a 2-complex with a face $f$ of size one embedded in $\Sbb^3$. Extend $C$ to a simply connected 2-complex embedded in $\Sbb^3$ by \autoref{general} (making $C$ locally connected can easily be done by considering a \lq thickening\rq); call it $D$.
 Now $D/f$ has a planar rotation system by \autoref{rot_closed_down} and is simply connected, so by \autoref{combi_intro} is embeddable in $\Sbb^3$; thus $C/f$ is embeddable.

 Secondly, let $C$ be a 2-complex with a nonloop $e$. Then $C$ embeds in $\Sbb^3$ if and only if $C/e$ embeds in $\Sbb^3$. We can argue similar as above via \autoref{general}, refering to \autoref{sum_planar1} and \autoref{sum_planar2} in place of \autoref{rot_closed_down}

Thirdly, let $D$ be a 2-complex obtained from an embedded complex $C$ by splitting a vertex $v$. We obtain $D'$ from $D$ by adding a new vertex and introducing new edges from it to all splitters of $v$. It suffices to show that $D'$ is embeddable. Contracting the newly added edges from $D'$ yields the 2-complex $C$. Now use the above fact that a 2-complex $X$ with a nonloop edge $e$ embeds in $\Sbb^3$ if and only if $X/e$ embeds in $\Sbb^3$. Since $C$ is obtained from $D'$ by contracting a tree, we deduce that $D'$ embeds in $\Sbb^3$. Hence the subcomplex $C$ embeds in $\Sbb^3$.
\end{rem}

\begin{lem}\label{cone-red}
Let $H$ and $G$ be graphs without loops so that $H$ is a minor of $G$.
The cone over $G$ has the cone over $H$ as a space minor.
\end{lem}

\begin{proof}
By \autoref{del-cone}, in order to make an inductive proof work it suffices to consider the case that $H=G/f$ for some edge $f$ of $G$.
Now in the cone over $G$, let $f$ be the face corresponding to the edge $f$ in the link graph at the top. Let $e$ be the edge of $f$ that is not incident with the top.
Now contract $e$ and then the face $f$ gets size two, and also contract it onto an edge. The resulting 2-complex is the cone over $H$.
\end{proof}

\begin{rem}(Motivation)
 Simplicial combined cones do not work so well with space minors. The problem is \autoref{cone-red} is not true for combined cones.
 To see this, note that the simplicial combined over an edge $e$ can contain a tetrahedron at the edge $e$, where the two faces containing $e$ are in both cones and each cone contributes one of the remaining two faces.
 Combined cones over $G_1\oplus_e G_2$ where every edge of $G_2$ is subdivided cannot contain tetrahedra, and in fact when embedded in 3-space the embedding has only one chamber. So such 2-complexes cannot have a space minor where the embedding has more than one chamber; for example, when the 2-complex includes a tetrahedron.
 This example shows that the minimal members of the five families of \autoref{expli-computation} are a little bit nasty to compute -- although they are clearly finite.

 However, this is completely artificial and we can reduce the constructions from \autoref{expli-computation} to five excluded minors if we give up that we want the excluded minors to be simplicial complexes, as follows.
 \end{rem}

Given two 2-complexes $C$ and $D$ with isomorphic subcomplexes $C'$ and $D'$, respectively, the 2-complex obtained from $C$ by \emph{gluing} $D$ at $C'$ is obtained from the disjoint union of $C$ and $D$ by identifying $C'$ and $D'$ via their isomorphism.
Given a vertex-sum $G=G_1\oplus_e G_2$, the \emph{combined cone} over this vertex-sum is obtained from the cone over $G_1$ by gluing the cone over $G_2$ via the subcomplex consisting of the edge $e$ and its incident faces in both cones such that in the resulting complex the tops of the cones are glued onto distinct endvertices of $e$.

\begin{rem}
 While simplicial combined cones are always simplicial complexes, the gluing process for combined cones may create parallel edges.
\end{rem}

It is straightforward to minor simplicial combined cones down to combined cones over the same graphs:

\begin{lem}
 Let $C$ be a simplicial combined cone over $G=G_1\oplus_e G_2$. Let $D$ be the 2-complex obtained from $C$ by topologically deleting all edges not incident with an endvertex of $e$.
 Then $D$ is equal to the combined cone $G=G_1\oplus_e G_2$.
 \qed
\end{lem}

\begin{lem}\label{cone-red-combined}
Let $G$ be obtained from $H$ by subdividing an edge.
Let $C$ be the combined cone over $G=G_1\oplus_e G_2$  such that the vertex $e$ has degree more than two.
Then there is a subdivision edge $f$ of $G$ such that $C$ has a space minor over one of the vertex-sums $H= (G_1/f) \oplus_e G_2$ or $H= G_1 \oplus_e (G_2/f)$.
\end{lem}

\begin{proof}
 Since the vertex $e$ has degree more than two, one of the subdivision edges is not incident with $e$; denote that edge by $f$ and assume $f$ is an edge of $G_1$.
 Note that $H= (G_1/f) \oplus_e G_2$.
 Let $x$ be the unique edge of the face $f$ that is not incident with an endvertex of the edge $e$.
 Now in $C$ contract $x$ and then the face $f$ gets size two, and also contract it onto an edge. The resulting 2-complex is the combined cone over $H= (G_1/f) \oplus_e G_2$.
\end{proof}

   \begin{figure} [htpb]
\begin{center}
   	  \includegraphics[height=1.5cm]{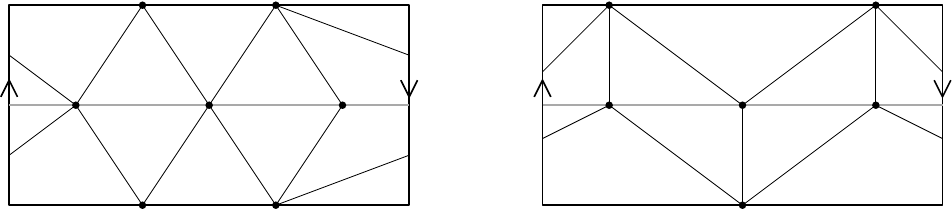}
   	  \caption{The only two nice space minor minimal triangulations of the M\"obius strip.}\label{fig:moe}
\end{center}\vspace{-0.7cm}
   \end{figure}

\begin{lem}\label{nice_moe}
 There are only two nice space minor minimal triangulations of the M\"obius strip; these are depicted in \autoref{fig:moe}
\end{lem}

\begin{proof}
Let $C$ be a nice triangulation of the M\"obius strip that is space minor minimal. So all vertices are on the central cycle $o$ or on the boundary cycle, which we call $b$. We give $b$ the structure of a directed cycle.
Since $C$ is a triangulation, it is uniquely determined by the sequence of degrees of the vertices on $b$, read in the order in which they appear on $b$.
Since $o$ contains exactly three vertices, each vertex of $b$ has degree at most 5.
If a vertex of $b$ has degree three, contract one of its incident edges on $b$ and the resulting face of size two. We obtain a simplicial complex and thus a smaller M\"obius obstruction. So by minimality of $C$, all vertices of $b$ have degree at least 4.
If two vertices that are adjacent on $b$ have degree four, then we can contract the edge joining them and the resulting face of size two and obtain a simplicial complex, and thus a smaller M\"obius obstruction.
So by minimality of $C$, on $b$ no two vertices of degree 4 are adjacent.
We denote by $d_5$ the number of degree five vertices of $b$ and by $d_4$ the number of degree four vertices of $b$.
\begin{sublem}\label{calc12}
 $12-d_5=3d_5+ 2d_4$.
\end{sublem}

\begin{cproof}
We count the edges between $o$ and $b$ in two different ways.
Since every degree five vertex of $b$ contributes three edges and every degree four vertex of $b$ contributes two edges, this number is $3d_5+ 2d_4$.

Now let $v$ be a vertex of $o$, and let $N$ be a small open neighbourhood around the point $v$ in the M\"obius strip. The cycle $o$ divides $N$ into two components; we shall refer to these components as \emph{top} and \emph{bottom}.
The \emph{top degree} of $v$ is the number of edges incident with $v$ that are in the top component.
If the top degree of a vertex of $o$ is more than two, since its top-degree neighbours are consecutive on $b$, it has a degree three neighbour, which is not possible as shown above.
So the top degree is at most two. The top degree cannot be zero as $o$ does not bound a disc as the triangulation is nice. So the top degree is one or two.
If the top degree is one, then the unique top-neighbour $x$ on $b$ has degree five and the edge from $v$ to $x$ is its middle edge in the rotator of $x$ in the embedding.
Similarly we defined the \emph{bottom degree}, and do these same analysis.
Let $t$ be the number of vertices on $o$ whose top degrees are one plus the number of vertices on $o$ whose bottom degree are one.
By the above $t=d_5$.
So the number of edges across is $12-d_5$ (indeed, we have 3 top-degrees and 3 bottom-degrees and if $d_5=0$ all count twice).
\end{cproof}

Rearranging the equation of \autoref{calc12} yields $6=2d_5+d_4$. This equation has the non-negative integer solutions $(d_5=3, d_4=0)$ and $(d_5=2,d_4=2)$. Since $b$ has length at least three, we have that the additional constraint that $d_5+d_4\geq 3$ and since no two degree four vertices can be adjacent on $b$, we conclude that $d_5\geq 2$. So there are no other solutions that are relevant to us.

The solution $(d_5=3, d_4=0)$ gives the sequence $5,5,5$ of degrees, which corresponds to the right triangulation of \autoref{fig:moe}.
Since no two vertices of degree four can be adjacent on $b$, the solution $(d_5=2,d_4=2)$ can only give the (cyclic) sequence $5,4,5,4$; this sequence corresponds to the left triangulation of \autoref{fig:moe}.

\end{proof}

\begin{cor}\label{nice_moe_cor}
 There are only two  space minor minimal M\"obius obstructions; these are obtained from the simplicial complexes of \autoref{fig:moe} by attaching a face at the central cycle.\qed
\end{cor}

Let $\Zcal'$ consist of the cones over $K_5$ and $K_{3,3}$, the five combined cones over $K_5$ and $K_{3,3}$ (compare \autoref{expli-computation}) and the M\"obius obstructions from \autoref{nice_moe_cor}.

\begin{lem}\label{final-list}
Every element of $\Zcal$ has a space minor in $\Zcal'$. The set $\Zcal'$ has only nine elements.
\end{lem}

\begin{proof}
Recall that the list $\Zcal$ is the combination of the lists given in \autoref{cone-final},  \autoref{combined-cone-summary} and \autoref{moebius-summary}.
Now apply to these lists \autoref{cone-red}, \autoref{cone-red-combined} and \autoref{nice_moe_cor}, respectively.
Finally, $|\Zcal'|=2+5+2=9$.
\end{proof}

\section{Concluding remarks}\label{concl77}

In an earlier version of this paper, which is available on arxiv \cite{3space1v3},
we use \autoref{rot_system_exists} instead of \autoref{rot_system_exists-simplycon} to compute the excluded minors.
The disadvantage of this earlier approach is that it is much more complicated and the list of excluded minors is much longer (in fact in that earlier version we did not manage to preserve simple connectedness and being a simplicial complex for all the obstructions, which added clutter to the list). However, the advantage is that with the old way one can characterise the existence of a planar rotation system combinatorically in terms of finitely many obstructions --  without the assumption of simple connectedness; details can be found in
{\cite[Theorem 7.1]{{3space1v3}}}. For general (not necessarily locally 3-connected) simply connected simplicial complexes, we take the road via \autoref{rot_system_exists} (as the stronger assumptions of \autoref{rot_system_exists-simplycon} do not seem available to us in this more general context), see \autoref{chapterV} for details.

\autoref{kura_intro_hom} is a structural characterisation of which locally 3-connected
2-dimensional simplicial complex $C$ with trivial first homology group embed in 3-space. Does
this have algorithmic consequences?
The methods of this chapter give an algorithm that checks in linear\footnote{Linear in the
number of faces of $C$.} time  whether a locally 3-connected 2-dimensional simplicial complex
has a planar rotation system. For general 2-dimensional simplicial complex we obtain a quadratic
algorithm, see \autoref{chapterV} for details. But how easy is it to check
whether
$C$ is simply connected? For simplicial complexes in general this is not decidable; indeed for
every finite presentation of a group one can build a 2-dimensional simplicial complex that has that
fundamental group.
However, for simplicial complexes that embed in some (oriented) 3-manifold; that is, that have a
planar
rotation system, this problem is known as the sphere recognition problem.
Recently it was shown that sphere recognition lies in NP \cite{{Iva01},{Sch11}} and co-NP
assuming the generalised Riemann hypothesis \cite{{HeuZen16},{Zen16}}. It is an open question
whether there is a polynomial time algorithm.

	\chapter{The topological part of the proof of \autoref{kura_intro_hom}}\label{chapterII}

\section{Abstract}
 We prove that 2-dimensional simplicial complexes with trivial first homology group have
	topological embeddings in 3-space if and only if there are embeddings of their link graphs
in the plane that are compatible at the edges and they are simply connected.

\section{Introduction}

Here we give combinatorial characterisations for when certain simplicial complexes embed in
3-space. This completes the proof of a 3-dimensional analogue of Kuratowski's characterisation
of
planarity for graphs, started in \autoref{chapterI}.

\vspace{.3cm}

A (2-dimensional) simplicial complex has a topological embedding in
3-space if and only if it has a piece-wise linear embedding if and only if it has a differential
embedding \cite{{Bin59},{Hatcher3notes},{Pap43}}.\footnote{However this is not equivalent to having
a linear embedding, see \cite{Bre83}, and \cite{mtw_hardness} for further references. }
Perelman proved that every compact simply connected 3-dimensional
manifold is isomorphic to the 3-sphere \Sthree\ \cite{{{Perelman1}, {Perelman3},
{Perelman2}}}.  In
this chapter we use Perelman's theorem to prove a combinatorial characterisation of which
simply
connected simplicial complexes can be topologically embedded into $\Sbb^3$ as follows.

The \emph{link graph} at a vertex $v$ of a simplicial complex is the graph whose vertices
are the edges incident with $v$ and
whose edges are the faces incident with $v$ and the incidence relation is as in $C$, see
\autoref{fig:intro}.    \begin{figure} [htpb]
\begin{center}
   	  \includegraphics[height=3cm]{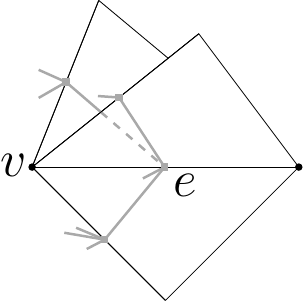}
   	  \caption{The link graph at the vertex $v$ is indicated in grey. The edge
$e$ projects down to a vertex in the link graph. The faces incident with $e$
project down to edges. }\label{fig:intro}
\end{center}\vspace{-0.7cm}
   \end{figure}
Roughly, a \emph{planar rotation system} of a simplicial complex $C$ consists of cyclic
orientations $\sigma(e)$ of the faces incident
with each edge $e$ of $C$ such that there are embeddings in the plane of the link graphs
such that at vertices $e$
the cyclic orientations of the incident edges agree with the cyclic orientations $\sigma(e)$.
It is easy to see that if a simplicial complex $C$ has a topological embedding into some oriented
3-dimensional manifold, then it has a
planar rotation system. Conversely, for simply connected simplicial complexes the
existence of
planar rotation systems is enough to characterise embeddability into \Sthree:

\begin{thm}\label{combi_intro}
 Let $C$ be a simply connected simplicial complex. Then $C$ has a topological embedding into
$\Sbb^3$ if and only if $C$ has a planar rotation system.
\end{thm}

A related result has been proved by Skopenkov for general 3-manifolds (not necessarily compact) \cite{Skopenkov94}.
The main result of this chapter is the following extension of \autoref{combi_intro}.

\begin{thm}\label{combi_intro_extended}
 Let $C$ be a simplicial complex such that the first homology group $H_1(C,\Fbb_p)$ is trivial for
some prime $p$. Then $C$ has a topological embedding into \Sthree\ if and only if $C$ is simply
connected and it has a planar rotation system.
\end{thm}

This implies characterisations of topological embeddability into \Sthree\ for the classes of
simplicial complexes with abelian fundamental group and simplicial complexes in general, see
\autoref{beyond} for details.

This chapter is organised as follows. After reviewing some
elementary definitions in \autoref{basics}, in \autoref{prelims2}, we
introduce rotation systems, related concepts and prove basic properties of them. In
Sections \ref{sec4} and \ref{sec5} we prove \autoref{combi_intro}.
The proof of \autoref{combi_intro_extended} in \autoref{sec6} makes use of \autoref{combi_intro}.
Further extensions are derived in \autoref{beyond}. In \autoref{non-or} we discuss how one could
characterise embeddability of 2-complexes in general 3-manifolds combinatorially.

\section{Basic definitions}\label{basics}

In this short section we recall some elementary definitions that are important for this chapter.

A \emph{closed trail} in a graph is a cyclically ordered sequence $(e_n|n\in \Zbb_k)$ of distinct
edges $e_n$ such that the
starting vertex of $e_{n}$ is equal to the endvertex of $e_{n-1}$.
An (abstract) (2-dimensional) \emph{complex} is a graph\footnote{In this paper graphs are allowed
to have
parallel edges and loops. All graphs and 2-complexes considered in this paper are finite.} $G$
together with a
family of
closed trails in $G$, called the \emph{faces} of the complex.
We denote complexes $C$ by triples $C=(V,E,F)$, where $V$ is the set of \emph{vertices}, $E$ the
set of \emph{edges} and $F$ the set of
faces. We assume furthermore that every vertex of a complex is incident with an edge and every edge
is incident with a face.
The \emph{1-skeleton} of a complex $C=(V,E,F)$ is the graph $(V,E)$.
A \emph{directed} complex is a complex together with a choice of
direction at each of its edges and a choice of orientation at each of its faces.
For an edge $e$, we denote the direction chosen at $e$ by $\vec{e}$.
For a face $f$, we denote the orientation chosen at $f$ by $\vec{f}$.

Examples of complexes are (abstract) (2-dimensional) simplicial complexes.
In this paper all simplicial complexes are directed -- although we will not always say it
explicitly. A \emph{(topological) embedding} of a simplicial complex $C$ into a topological space
$X$ is an injective
continuous map from (the geometric realisation of) $C$ into $X$.
In our notation we suppress the embedding map and for example write `$\Sbb^3\sm C$' for
the topological space obtained from
$\Sbb^3$ by removing all points in the image of the embedding of $C$.
A 2-complex is \emph{reasonable} if each vertex or edge is incident with a face.
\begin{eg}
 Simply connected 2-complexes are reasonable.
\end{eg}
\begin{rem}
All 2-complexes considered in this paper are reasonable, and we add this assumption to all lemmas of this paper.
This assumption is necessary for various constructions and we believe it is much simpler for the reader to assume it throughout than having to carry it along all the time, and add a few technical details to some otherwise fairly natural proofs. Graph theoretically speaking, this just means that 2-complexes and their link graphs have no isolated vertices.
\end{rem}

In this paper, a \emph{surface} is a compact 2-dimensional manifold (without boundary)\footnote{We
allow surfaces to be disconnected. }.
Given an embedding of a graph in an oriented surface, the \emph{rotation system} at a vertex $v$ is
the cyclic orientation\footnote{A \emph{cyclic orientation} is a choice of one of the two
orientations
of a cyclic ordering.}  of the edges incident with $v$ given by `walking around' $v$ in the
surface in a small circle in the direction of the orientation.
Conversely, a choice of rotation system at each vertex of a graph $G$ defines an embedding of $G$
in an oriented surface as explained in \autoref{chapterI}.

A \emph{cell complex} is a graph $G$ together with a set of directed walks such that each
direction of an edge of $G$ is in precisely one of these directed walk es. These directed walks are
called the \emph{cells}. The geometric realisation of a cell complex is obtained from (the
geometric realisation of) its graph by gluing discs so that the cells are the boundaries of these
discs. The geometric realisation is always an oriented surface.
Note that cell complexes need
not be complexes as cells are allowed to contain both directions of an edge.
The \emph{rotation system of} a cell complex $C$ is the rotation system of the graph of $C$ in the
embedding in the oriented surface given by $C$.

\section{Rotation systems}\label{prelims2}

In this section we introduce rotation systems of complexes and some related concepts.

The \emph{link graph} of a \scom\ $C$ at a vertex $v$ is the graph whose vertices are the
edges incident with $v$.
The edges are the faces incident\footnote{A face is incident with a vertex if
there is an edge incident with both of them.} with $v$.
The two endvertices of a face $f$ are those vertices corresponding to the two
edges of $C$ incident with $f$ and $v$.
We denote the link graph at $v$ by $L(v)$.

A \emph{rotation system} of a directed complex $C$ consists of for each edge $e$ of $C$ a cyclic
orientation\footnote{If the edge $e$ is only incident with a single face, then $\sigma(e)$ is
empty.} $\sigma(e)$ of the faces incident
with $e$.

Important examples of rotation systems are those \emph{induced} by topological embeddings of
2-complexes $C$ into \Sthree\  (or more generally in some oriented 3-manifold); here for an edge
$e$
of $C$,
the cyclic orientation $\sigma(e)$ of the
faces incident with $e$ is the ordering in which we see the faces when walking around some
midpoint of $e$ in a circle of small radius\footnote{Formally this means that the circle
intersects each face in a single point and that it can be contracted onto the chosen midpoint of
$e$ in such a way that the image of one such contraction map intersects each face in an
interval.} -- in the direction of the orientation of \Sthree. It can be shown that $\sigma(e)$ is
independent of the chosen circle if small enough and of the chosen midpoint.

Such rotation systems have an additional property:
let  $\Sigma=(\sigma(e)|e\in E(C))$ be a rotation system of a simplicial complex $C$ induced by
a topological embedding of $C$ in the 3-sphere.
Consider a 2-sphere of small diameter around a vertex $v$. We
may assume that each edge of $C$ intersects this 2-sphere in at most
one point and that each face intersects it in an interval or not at all. The
intersection of the 2-sphere and $C$ is a graph: the link graph at $v$. Hence link
graphs of 2-complexes with rotation systems induced from embeddings in 3-space must always be
planar. Even more is true: the
cyclic orientations $\sigma(e)$ at the edges of $C$ form -- when projected down to a link graph
to rotators at the vertices of the link graph -- a rotation system at the link graph, see
\autoref{fig:link_project}.

   \begin{figure} [htpb]
\begin{center}
   	  \includegraphics[height=3cm]{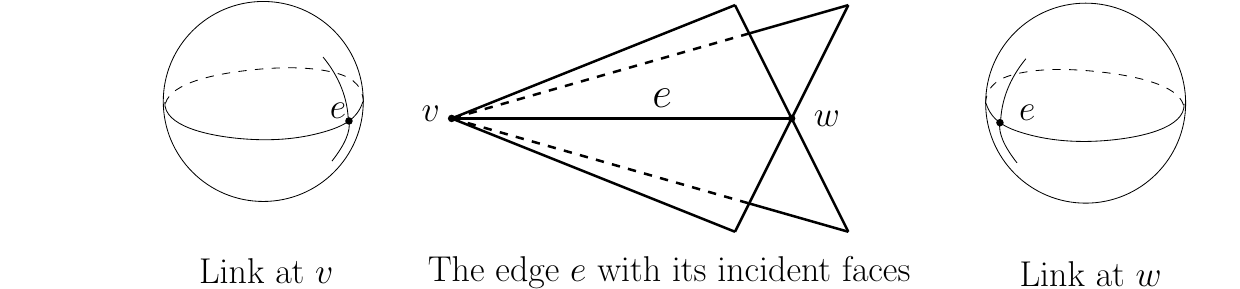}
   	  \caption{The cyclic orientation $\sigma(e)$ of the faces incident with the edge $e$ of
$C$ projects down to rotators at $e$ in the link graphs at either endvertex of $e$. In these link
graphs, the projected rotators at $e$ are reverse.
}\label{fig:link_project}
\end{center}\vspace{-0.7cm}
   \end{figure}

Next we shall define `planar rotation systems' which roughly are rotation systems satisfying
such an additional property.
The cyclic orientation $\sigma(e)$ at the edge $e$ of a rotation system defines a rotation
system $r(e,v, \Sigma)$ at each vertex $e$ of
a link graph $L(v)$: if the directed edge $\vec{e}$ is directed towards $v$ we take $r(e,v,
\Sigma)$ to be $\sigma(e)$.
Otherwise we take the reverse of $\sigma(e)$. As explained in \autoref{basics}, this
defines an embedding of the link graph into an oriented surface. The
\emph{link complex} for
$(C,\Sigma)$ at the vertex $v$ is the cell complex obtained from the link graph $L(v)$
by
adding the faces of the above
embedding of $L(v)$ into the oriented surface.
By definition,  the geometric realisation of the link complex is
always a surface. To shortcut notation, we will not distinguish between the
link complex and its geometric realisation and just say things like:
`the link complex is a sphere'.
A \emph{planar rotation system} of a directed simplicial complex $C$ is a rotation system such that
for
each vertex $v$ all link complexes are spheres -- or disjoint unions of spheres (if the link graph
is not connected).
The paragraph before shows the following. We say that a simplicial complex $C$ is
\emph{locally connected} if all link graphs are connected.

\begin{obs}\label{obo}
Rotation systems induced by topological embeddings  of locally connected\footnote{
\autoref{obo} is also true without the assumption of `local connectedness'. In that
case however
the link complex is
disconnected. Hence it is no longer directly given by the drawing of the link graph on a
ball of small radius as above.} simplicial complexes in the 3-sphere are planar.
 \qed
\end{obs}

Next we will define the \emph{local surfaces of a topological embedding} of a simplicial complex
$C$ into \Sthree.
The local surface at a connected component of $\Sbb^3\sm C$ is the following.
Pour concrete into this connected component.
The surface of the concrete is a 2-dimensional manifold.
The local surface is the simplicial complex drawn at the surface by the vertices, edges and faces
of $C$.
Note that if an edge $e$ of $G$ is incident with more than two faces that are on the
surface, then the surface will contain at least two clones of the edge $e$, see
\autoref{fig:loc_surf}.

   \begin{figure} [htpb]
\begin{center}
   	  \includegraphics[height=3cm]{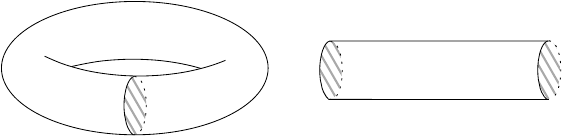}
   	  \caption{On the left we depicted the torus with an additional face
attached on a genus reducing curve in the inside. On the right we depicted the
local surface of its inside component. It is a sphere and contains two
copies of the newly added face (and its incident edges).}\label{fig:loc_surf}
\end{center}\vspace{-0.7cm}
   \end{figure}

Now we will define \emph{local surfaces for a pair $(C,\Sigma)$} consisting of a complex
$C$ and one of its rotation systems
$\Sigma$.
\autoref{topo_to_combi} below says that under fairly general circumstances the local surfaces of a
topological embedding are the local
surfaces of the rotation system induced by that topological embedding. The set of faces of
a local surface will be an equivalence
class of the set of orientations of faces of $C$.
The \emph{local-surface-equivalence relation} is the symmetric transitive closure of the following
relation.
An orientation $\vec f$ of a face $f$ is \emph{locally related} via an edge $e$ of
$C$ to an orientation $\vec{g}$ of a face $g$ if $f$ is just before $g$ in $\sigma(e)$ and
$e$
is traversed positively by $\vec f$ and
negatively by $\vec g$ and in $\sigma(e)$ the
faces $f$ and $g$ are adjacent.
Here we follow the convention that if the edge $e$ is only incident with a single face, then the
two orientations of that face
are related via $e$.
Given an equivalence class of the local-surface-equivalence relation, the \emph{local surface} at
that equivalence class is the following
 complex whose set of faces is (in bijection with) the set of orientations
in that equivalence class.
We obtain the complex from the disjoint union of the faces of these orientations by
gluing together two of these faces $f_1$ and $f_2$ along two of their edges if these edges are
copies
of the same
edge $e$ of $C$ and
$f_1$ and $f_2$ are related via $e$. Of course, we glue these two edges in a
way that endvertices are identified only with copies of the same vertex of $C$.
Hence each  edge of a local surface is incident with precisely two faces. Hence
its
geometric realisation is always a
is a surface. Similarly as for link complexes, we shall just say things like `the
local surface is a sphere'.
\begin{obs}\label{loc_is_con}
Local surfaces of planar rotation systems are always connected.
\qed
\end{obs}
A \emph{(2-dimensional) orientation} of a complex $C$ such that each edge is in precisely two faces
is a choice of orientation of each
face of $C$ such that each edge is traversed in opposite directions by the chosen orientation of
the two incident faces.
Note that a complex whose geometric realisation is a surface has an orientation if and only its
geometric realisation is orientable.

\begin{obs}\label{loc_is_orientable}
The set of orientations in a local-surface-equivalence class defines an orientation of its local
surface.

In particular, local surfaces are cell complexes.
\qed
\end{obs}

We will not use the following lemma in our proof of \autoref{combi_intro}
 However, we think that it gives a
better intuitive understanding of local surfaces.
\begin{lem}\label{topo_to_combi}
 Let $C$ be a connected and locally connected complex embedded into $\Sbb^3$ and let $\Sigma$ be
the induced planar rotation system.
 Then the local surfaces of the topological embedding are equal to the local surfaces for
$(C,\Sigma)$.\qed
\end{lem}

There is the following relation between vertices of local surfaces and faces of link
complexes.

\begin{lem}\label{loc_inc_AND_loc_surfaces}
Let $\Sigma$ be a rotation system of a simplicial complex $C$.
There is a bijection $\iota$ between the set of vertices of local
surfaces for $(C,\Sigma)$ and the set of faces of link
complexes for $(C,\Sigma)$, which maps each vertex $v'$ of a local surface cloned from the vertex
$v$ of $C$ to a face $f$ of the link complex at $v$ such that the rotation system at $v'$ is an
orientation of $f$.
\end{lem}
\begin{proof}
The set of faces of the link complex at $v$ is in bijection with the set of
$v$-equivalence classes; here the \emph{$v$-equivalence relation} on the set of orientations of
faces of $C$
incident with $v$ is
the symmetric transitive closure of the relation `locally related'. Since we work in a subset of
the orientations, every $v$-equivalence
class is contained in a local-surface-equivalence class.
On the other hand the set of all clones of a vertex $v$ of $C$ contained in a local surface $S$ is
in bijection with the set of
$v$-equivalence classes contained in the local-surface-equivalence class of $S$.
This defines a bijection $\iota$ between the set of vertices of local
surfaces for $(C,\Sigma)$ and the set of faces of link
complexes for $(C,\Sigma)$.

It is straightforward to check that $\iota$ has all the properties claimed in the lemma.
\end{proof}

\begin{cor}\label{cor7}
Given a local surface of a simplicial complex $C$ and one of its vertices $v'$ cloned from a
vertex
$v$ of $C$, there is a homeomorphism from a neighbourhood around $v'$ in the local surface to the
cone with top $v'$
over the face boundary of $\iota(v')$ that fixes $v'$ and the edges and faces incident with
$v'$ in a neighbourhood around $v'$.
\qed
\end{cor}

The definitions of link graphs and link complexes can be generalised from
simplicial complexes to
complexes as follows.
The \emph{link graph} of a complex $C$ at a vertex $v$ is the graph whose vertices are
the edges incident with $v$.
For any traversal of a face of the vertex $v$, we add an edge between the two vertices that when
considered as edges of $C$
are in the face just before and just after that traversal of $v$. We stress that we allow parallel
edges
and loops.
Given a complex $C$, any rotation system $\Sigma$ of $C$ defines rotation systems at each link
graph of $C$.
Hence the definition of link complex extends.

\section{Constructing piece-wise linear embeddings}\label{sec4}

In this section we prove \autoref{combi} below, which is used in the proof of
\autoref{combi_intro}.
   \begin{figure} [htpb]
\begin{center}
   	  \includegraphics[height=2cm]{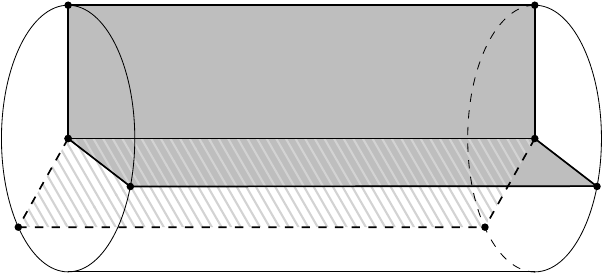}
   	  \caption{A cylinder with an embedded 2-complex.}\label{fig:cylinder}
\end{center}\vspace{-0.7cm}
   \end{figure}
\begin{eg}\label{solid_klein}
Here we give the definition of the Solid Klein Bottle and construct an embedding of a 2-complex $C$
in the Solid Klein Bottle that does not induce a planar rotation system.

Given a solid cylinder (see \autoref{fig:cylinder}), there are two ways to identify the
bounding discs: one way the boundary becomes a torus and the other way the boundary becomes a
Klein Bottle. We refer to the topological space obtained from the solid cylinder by identifying as
in the second case as the \emph{Solid Klein Bottle}. In \autoref{fig:cylinder} we embedded a
2-complex in
the solid cylinder. Extending the above identification of the cylinder to the embedded 2-complex,
induces an embedding of this new 2-complex in the Solid Klein Bottle. This embedding does not
induce
a planar rotation system.
\end{eg}

Next we show that the Solid Klein Bottle is the only obstruction that prevents embeddings of
2-complexes in 3-manifolds to induce planar rotation systems.

\begin{lem}\label{prs_klein}
 Let $M$ be a 3-manifold that does not include the Solid Klein Bottle as a submanifold.
 Then any embedding of a 2-complex $C$ in $M$ induces a planar rotation system.
\end{lem}

\begin{proof}
 By treating different connected components separately, we may assume that $C$ is connected.
 Around each vertex of $C$ we pick a small neighbourhood that is an open 3-ball, and around each
edge of $C$ we pick a small neighbourhood that is an open cylinder. Next we define orientations on
these neighbourhoods. For that we pick an arbitrary vertex and pick for its neighbourhood one of
the two orientations arbitrarily. Then we pick compatible orientations at the
neighbourhoods of the incident edges. As $C$ is connected, we can continue recursively along a
spanning tree of $C$ until we picked orientations at all neighbourhoods of vertices and edges. If
for a vertex and an incident edge, their neighbourhoods have incompatible orientations, this edge
must be outside the spanning tree and we can build a Solid Klein Bottle from its fundamental cycle
as follows. Indeed, sticking together the balls at the vertices and the cylinders at the edges gives
a Solid Klein Bottle.
By assumption, this does not occur, so all orientations of neighbourhoods are compatible.

Let $\Sigma$ be the rotation system induced by this embedding of $C$ in $M$ with respect to the
orientations chosen above. Clearly, this rotations system $\Sigma$ is planar.
\end{proof}

\begin{cor}\label{prs_orient}
 Any embedding of a 2-complex $C$ in an orientable 3-manifold $M$ induces a planar rotation system.
\end{cor}

\begin{proof}
 It is well-known that a 3-manifold is orientable if and only if it does not have the Solid Klein
Bottle as a submanifold. Hence \autoref{prs_orient} follows from \autoref{prs_klein}.
\end{proof}

Throughout this section we fix a connected and locally connected simplicial complex $C$ with a
rotation system $\Sigma$.
An associated topological space $T(C,\Sigma)$ is defined as follows.
For each local surface $S$ of $(C,\Sigma)$ we take an embedding into \Sthree\ as follows.
Let $g$ be the genus of $S$. We start with the unit ball in \Sthree\ and then identify $g$ disjoint
pairs of discs through the outside.\footnote{We have some flexibility along which paths on the
outside we do these identifications but we do not need to make particular choices for our
construction to work.} The constructed surface is isomorphic to $S$, so this defines an embedding of
$S$.
Each local surface
is oriented and we denote by $\hat S$ the topological space obtained from \Sthree\ by deleting all
points on the outside of $S$. For each local surface $S$, let $\varphi_S$ be the map which maps
every cell of $S$ to the cell of $C$ it is cloned from.
\begin{eg}
If the map $\varphi_S$ is injective, then $S$ is a subcomplex of $C$. By taking
3-dimensional baricentric
subdivisions, one could ensure that the maps $\varphi_S$ are injective. However, our proof
below also goes through smoothly without any injectivity assumptions.
\end{eg}
We obtain  $T(C,\Sigma)$ from the simplicial complex $C$ by
gluing\footnote{Given a map $f$ from a topological space $X$ to a topological space $Y$, the space
obtained from $Y$ by gluing $X$ via $f$ is the following topological space. Its point set is
obtained from the disjoint union of $X$ and $Y$
by factoring out the equivalence relation generated from $f$, where $x\in X$ and $y\in Y$ are
related if $f(x)=y$. The topology of this space is defined through the universal property that it is
the densest topology such that the injection from $X$ into this space and the lift of $f$ into this
space are continuous.}
onto each
local surface $S$ the topological space $\hat S$ along $S$ via $\varphi_S$.

We remark that associated topological spaces may depend on the chosen embeddings of the local
surfaces $S$ into \Sthree. However, if all local surfaces are spheres, then any two associated
topological spaces are isomorphic and in this case we shall talk about `the' associated topological
space.

Clearly, associated topological spaces $T(C,\Sigma)$ are compact and connected as $C$ is connected.

\begin{lem}\label{is_manifold}
The rotation system $\Sigma$ is planar if and only if the associated topological
space $T(C,\Sigma)$ is an oriented 3-dimensional manifold.
\end{lem}

 \begin{proof}
 \autoref{prs_orient} implies that if the topological space $T(C,\Sigma)$ is a 3-dimensional
orientable 3-manifold, then the rotation system $\Sigma$ is planar.
 Conversely, now assume that $\Sigma$ is a planar rotation system.
  We have to show that the topological space $T(C,\Sigma)$ is an orientable 3-manifold.
  It suffices to show that $T(C,\Sigma)$ is a 3-manifold since then orientability follows
immediately from the construction of $T(C,\Sigma)$.
So we are to show that there is a neighbourhood around any point $x$ of  $T(C,\Sigma)$ that is
isomorphic
to the closed 3-dimensional ball $B_3$.

 If $x$ is a point not in $C$, this is clear. If $x$ is an interior point of a face $f$, we obtain
a neighbourhood of $x$ by gluing together neighbourhoods of copies of $x$ in the local surfaces
that
contain an orientation of $f$.
Each orientation of $f$ is contained in local surfaces exactly once. Hence we glue together the two
orientations of $f$ and
clearly $x$ has a neighbourhood isomorphic to $B_3$.

Next we assume that $x$ is an interior point of an edge $e$.  Some open neighbourhood of
$x$ is isomorphic to the
topological space obtained from gluing
together for each copy of $e$ in a local surface, a neighbourhood around a copy $x'$ of $x$ on
those edges.
A neighbourhood around $x'$ has the shape of a piece of a cake, see \autoref{fig:piece_of_cake}

   \begin{figure} [htpb]
\begin{center}
   	  \includegraphics[height=2cm]{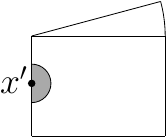}
   	  \caption{A piece of a cake. This space is obtained by taking the product of a triangle
with the unit interval. The edge $e$ is mapped to the set of points
corresponding to some vertex of the triangle. }\label{fig:piece_of_cake}
\end{center}\vspace{-0.5cm}
   \end{figure}

First we consider the case that $x$ has several copies.
As $\sigma(e)$ is a cyclic orientation, these pieces of a cake are glued together in a cyclic way
along faces. Since each cyclic
orientation of a face appears exactly once in local surfaces, we identify in each gluing step the
two cyclic orientations of a face.
Informally, the overall gluing will be a `cake' with $x$ as an interior point.
Hence a neighbourhood of $x$ is isomorphic to $B_3$.
If there is only one copy of $x'$, then the copy of $e$ containing $x'$ is incident with the two
orientations of a single face.
Then we obtain a neighbourhood of $x$ by identifying these two orientations. Hence there is a
neighbourhood of $x$ isomorphic to $B_3$.

It remains to consider the case where $x$ is a vertex of $C$. We obtain a neighbourhood of $x$ by
gluing together neighbourhoods of copies
of $x$ in local surfaces. We shall show that we have one such copy for every face of the link
complex for $(C,\Sigma)$ and a neighbourhood of $x$ in such a copy is given by the cone over that
face with $x$ being the top of the cone, see \autoref{fig:paste}.
   \begin{figure} [htpb]
\begin{center}
   	  \includegraphics[height=4cm]{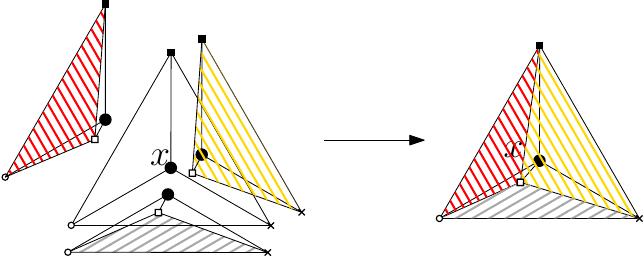}
   	  \caption{In this example the link complex of $x$ is a
tetrahedron. The three faces visible in our drawing are highlighted in red, gold and grey.
   	  On the left we see how the four cones over the faces of the link complex
are pasted together to form the
cone over the link complex depicted on the right. }\label{fig:paste}
\end{center}\vspace{-0.5cm}
   \end{figure}
We shall show that the glued together neighbourhood
is the cone over the link complex with $x$ at the top. Since $\Sigma$ is planar and $C$ is locally
connected, the link complex is isomorphic to the 2-sphere. Since the cone over the 2-sphere is a
3-ball, the neighbourhood of $x$ has the desired type.

Now we examine this plan in detail.
By \autoref{loc_inc_AND_loc_surfaces} and \autoref{cor7}, the copies are
mapped by the
bijection $\iota$ to the faces of the link complex at $x$ and a neighbourhood around such a copy
$x'$ is isomorphic to the cone
with top $x'$ over the face $\iota(x')$.
We glue these cones over the faces $\iota(x')$ on their faces that are obtained from edges of
$\iota(x')$ by adding the top $x'$.

The glued together complex is isomorphic to the cone over the complex $S$ obtained by gluing
together the faces $\iota(x')$ along edges, where we always glue the edge the way round so that
copies of the same vertex of the local
incidence graph are identified. Hence the vertex-edge-incidence relation and the
edge-face-incidence relation of $S$ are the same as for
the link complex at $x$. The same is true for the cyclic orderings of edges on faces.
So $S$ is equal to the link complex at $x$.

Hence a neighbourhood of $x$ is isomorphic to a cone with top $x$ over the link complex
at $x$. Since $\Sigma$ is a
planar rotation system, the link complex is a disjoint union of spheres. As $C$ is
locally connected, it is a sphere. Thus its
cone is isomorphic to $B_3$.
 \end{proof}

\begin{lem}\label{is simply connected}
 If $C$ is simply connected, then any associated topological space $T(C,\Sigma)$ is simply
connected.
\end{lem}
\begin{proof}
This is a consequence of Van Kampen's Theorem {\cite[Theorem 1.20]{{Hatcher}}}.
Indeed, we obtain $T$
from $T(C,\Sigma)$ by deleting all interior points of the sets $\hat S$ for local surfaces $S$ that
are not in a small open neighbourhood of $C$. This can be done in such a
way that $T$ has a deformation retract to $C$, and thus is
simply connected.
Now we recursively glue the spaces $\hat S$ back onto $T$.
In each step we glue a single space $\hat S$. Call the space obtained after $n$ gluings $T_n$.

The fundamental group of $\hat S$ is a quotient of the fundamental group of the topological space that is the
intersection of $T_n$ and $\hat S$ (this follows from the construction of the embedding of the
surface $S$ into the space $\hat S$). The fundamental group of $T_n$ is trivial by induction.
So we can apply
Van Kampen's Theorem to deduce that the gluing space $T_{n+1}$ has trivial fundamental group. Hence
the final gluing space $T(C,\Sigma)$ has trivial fundamental group. So it is simply connected.
\end{proof}

The converse of \autoref{is simply connected} is true if all local surfaces for $(C,\Sigma)$ are
spheres as follows.

\begin{lem}\label{is_simply_connected2}
 If all local surfaces for $(C,\Sigma)$ are spheres and the associated topological space
$T(C,\Sigma)$ is simply connected, then $C$ is simply connected.
\end{lem}

\begin{proof}
Let $\varphi$ an image of $\Sbb^1$ in $C$. Since  $T(C,\Sigma)$ is simply connected, there is
a homotopy from $\varphi$ to a point of $C$ in $T(C,\Sigma)$.
We can change the homotopy so that it avoids an interior point of each local
surface of the embedding.
Since each local surface is a sphere, for each local surface without the chosen point there is a
continuous projection to its
boundary. Since these projections are continuous, the concatenation of them with
the homotopy is continuous. Since this concatenation is constant on $C$ this
defines a homotopy of $\varphi$ inside $C$. Hence $C$ is simply connected.
\end{proof}

We conclude this section with the following special case of \autoref{combi_intro}.

\begin{thm}\label{combi}
Let $\Sigma$ be a rotation system of a locally connected simplicial complex $C$.
Then $\Sigma$ is planar if and
only if $T(C,\Sigma)$ is an oriented 3-manifold.
Furthermore, if $\Sigma$ is planar and $C$ is simply connected, then $T(C,\Sigma)$ must be the 3-sphere.
\end{thm}

\begin{proof}
By treating different connected components separately, we may assume that $C$ is connected.
The first part follows from \autoref{is_manifold}. The second part follows from \autoref{is simply
connected}
and Perelman's theorem \cite{Perelman1, Perelman2, Perelman3} that any compact simply connected
3-manifold is isomorphic to the 3-sphere.
\end{proof}

\begin{rem}\label{P_rem}
We used Perelman's theorem in the proof of \autoref{combi}. On the other hand it
together with Moise's theorem \cite{moise}, which says
that every compact 3-dimensional manifold has a triangulation, implies Perelman's theorem in the orientable case as
follows. Let $M$ be a simply connected oriented 3-dimensional compact manifold.
 Let $T$ be a triangulation of $M$, and let $C$ be the simplicial complex obtained from $T$ by
deleting the 3-dimensional cells (recall that 3-dimensional cells of a triangulation are homeomorphic to closed 3-balls).
Since $M$ is orientable, there is a rotation system given by the embedding of $C$ into $T$; denote it by $\Sigma$.
It is clear from that construction that $T$ is equal to the triangulation given by
the embedding of $C$ into $T(C,\Sigma)$. Hence we can apply \autoref{is_simply_connected2} to
deduce that $C$ is simply connected.
Hence by \autoref{combi} the topological space $T(C,\Sigma)$, into which $C$ embeds, is isomorphic
to the 3-sphere. Since $T(C,\Sigma)$ is isomorphic to $M$, we deduce that $M$ is isomorphic to  the
3-sphere.
We will not use this remark in our proofs.
\end{rem}

\section{Cut vertices}\label{sec5}

In this section we deduce \autoref{combi_intro} from \autoref{combi} proved in the last section.
Given a prime $p$, a \scom\ $C$ is \emph{$p$-nullhomologous} if every directed cycle of $C$ is
generated over $\Fbb_p$ by the boundaries of faces of $C$. Note that a simplicial complex $C$
is  $p$-nullhomologous if and only if the first homology group $H_1(C,\Fbb_p)$ is trivial.
Clearly, every simply connected \scom\ is $p$-nullhomologous.

A vertex $v$ in a connected complex $C$ is a \emph{cut vertex} if the 1-skeleton of $C$ without $v$
is a disconnected graph\footnote{We
define this in terms of the 1-skeleton instead of directly in terms of $C$ for a technical reason:
The object obtained from a simplicial
complex by deleting a vertex may have edges not incident with faces. So it would not be a
2-dimensional simplicial complex in the
terminology of this paper.}.
A vertex $v$ in an arbitrary, not necessarily connected, complex $C$ is \emph{a cut vertex} if it
is a
cut vertex in a connected
component of $C$.

\begin{lem}\label{is_loc_con}
 Every $p$-nullhomologous simplicial complex without a cut vertex is locally connected.
\end{lem}

\begin{proof}
 We construct for any vertex $v$ of an arbitrary simplicial complex $C$ such that the link graph
$L(v)$ at $v$ is not
connected and $v$ is not a cut vertex a cycle containing $v$ that is not generated by the face
boundaries of $C$.

Let $e$ and $g$ be two vertices in different components of $L(v)$. These are edges of $C$ and let
$w$ and $u$ be their endvertices
different from $v$. Since $v$ is not a cut vertex, there is a path in $C$ between $u$ and $w$ that
avoids $v$.
This path together with the edges $e$ and $g$ is a cycle $o$ in $C$ that contains $v$.

Our aim is to show that $o$ is not generated by the boundaries of faces of $C$.
Suppose for a contradiction that $o$ is generated. Let $F$ be a family of faces whose boundaries
sum
up to $o$.
Let $F_v$ be the subfamily of faces of $F$ that are incident with $v$. Each face in $F_v$ is an
edge of $L(v)$ and each vertex of $L(v)$
is incident with an even number (counted with multiplicities) of these edges except for $e$ and $g$
that are incident with an odd number of
these faces.
Let $X$ be the connected component of the graph $L(v)$ restricted to the edge set $F_v$ that
contains the vertex $e$.
We obtain $X'$ from $X$ by adding $k-1$ parallel edges to each edge that appears $k$ times in
$F_v$.
Since $X'$ has an even number of vertices of odd degree also $g$ must be in $X$. This is a
contradiction to the assumption that $e$ and $g$ are in different components of $L(v)$. Hence $o$
is not generated by the boundaries of
faces of $C$. This completes the proof.
\end{proof}

Given a connected complex  $C$ with a cut vertex $v$ and a connected component $K$ of the
1-skeleton of $C$ with $v$ deleted, the
\emph{complex attached} at $v$
centered at $K$ has vertex set $K+v$ and its edges and faces are those of $C$ all of whose incident
vertices are in $K+v$.

\begin{lem}\label{block-lem}
A connected simplicial complex $C$ with a cut vertex $v$ has a piece-wise linear
embedding into $\Sbb^3$
if and only if all complexes attached at
$v$ have
a piece-wise linear embedding into $\Sbb^3$.
\end{lem}

\begin{proof}
 If $C$ has an embedding into $\Sbb^3$, then clearly all complexes attached at $v$ have an
embedding. Conversely suppose that all
complexes attached at $v$ have an embedding into $\Sbb^3$. Pick one of these complexes arbitrarily,
call it $X$ and fix an embedding of it
into $\Sbb^3$. In that embedding pick for each component of $C$ remove $v$ except that for $X$ a
closed ball contained in $\Sbb^3$ that
intersects $X$ precisely in $v$ such that all these closed balls intersect pairwise only at $v$.
Each complex attached at $v$, has a
piece-wise linear embedding into the
3-dimensional unit ball as they have embeddings into $\Sbb^3$ such that some open set is disjoint
from the complex. Now we attach these
embeddings into the balls of the embedding of $X$ inside the reserved balls by identifying the
copies of $v$. This defines an embedding of
$C$.
\end{proof}

Recall that in order to prove \autoref{combi_intro} it suffices to show that any simply
connected simplicial complex $C$ has a piece-wise linear embedding into
$\Sbb^3$ if and only if $C$ has a planar rotation system.

\begin{proof}[Proof of \autoref{combi_intro}.]
Clearly if a simplicial complex is embeddable into $\Sbb^3$, then it has a planar rotation system.
For the other implication, let $C$ be a simply connected simplicial complex and $\Sigma$ be a
planar rotation system.
We prove the theorem by induction on the number of cut vertices of $C$. If $C$ has no cut vertex,
it is locally connected by \autoref{is_loc_con}. Thus it has a piece-wise linear embedding into
\Sthree\ by \autoref{combi}.

Hence we may assume that $C$ has a cut vertex $v$.
As $C$ is simply connected, every complex attached at $v$ is simply connected. Hence by the
induction hypothesis each of these complexes has a piece-wise linear embedding into \Sthree.
Thus $C$ has a piece-wise linear embedding into \Sthree\ by  \autoref{block-lem}.
\end{proof}

\section{Local surfaces of planar rotation systems}\label{sec6}

The aim of this section is to prove \autoref{combi_intro_extended}. A shorter proof is sketched
in \autoref{alg_topo} using algebraic topology. As a first step in that
direction, we first prove the following.

\begin{thm}\label{loc_are_spheres}
 Let $C$ be a locally connected $p$-nullhomologous \scom\ that has a planar rotation system. Then
all
local
surfaces of the planar rotation system are spheres.
\end{thm}

Before we can prove \autoref{loc_are_spheres} we need some preparation.
The complex \emph{dual} to a simplicial $C$ with a rotation system $\Sigma$ has as its set of
vertices the set of local
surfaces of
$\Sigma$.
Its set of edges is the set of faces of $C$, and an edge is incident with a vertex if the
corresponding face is in the corresponding local
surface. The faces of the dual are the edges of $C$. Their cyclic ordering is as given by $\Sigma$.
In particular, the
edge-face-incidence-relation of the dual is the same as that of $C$ but with the roles of edges and
faces interchanged.

Moreover, an orientation $\vec{f}$ of a face $f$ of $C$ corresponds to the direction of $f$ when
considered as an edge of the dual complex
$D$ that points
towards
the vertex of $D$ whose local-surface-equivalence class contains $\vec{f}$. Hence the direction of
the dual complex $C$ induces a direction
of
the complex $D$.
By $\Sigma_C=(\sigma_C(f)|f\in E(D))$ we denote the following rotation system for $D$: for
$\sigma_C(f)$ we take the orientation
$\vec{f}$ of $f$
in the directed complex $C$.

In this paper we follow the convention that for edges of $C$ we use the letter $e$ (with possibly
some subscripts) while for
faces of $C$ we use the letter $f$. In return, we use the letter $f$ for the edges of a dual
complex of $C$ and $e$ for its faces.

\begin{lem}\label{D_con}
 Let $C$ be a connected and locally connected simplicial complex. Then for any rotation system,
the dual complex $D$ is connected.
 \end{lem}

\begin{proof}
 Two edges of $C$ are \emph{$C$-related} if there is a face of $C$ incident with both of them.
 They are \emph{$C$-equivalent} if they are in the transitive closure of the symmetric relation
`$C$-related'.
 Clearly, any two $C$-equivalent edges of $C$ are in the same connected component. If $C$ however
is locally connected, also the converse
is true: any two edges in the same connected component are $C$-equivalent. Indeed, take a path
containing these two edges. Any
two edges incident with a common vertex are $C$-equivalent as $C$ is locally connected. Hence any
two edges on the path are $C$-equivalent.

We define \emph{$D$-equivalent} like `$C$-equivalent' with `$D$' in place of `$C$'.
Now let $f$ and $f'$ be two edges of $D$. Let $e$ and $e'$ be edges of $C$ incident with $f$ and
$f'$, respectively. Since $C$ is connected
and locally connected the edges $e$ and $e'$ are $C$-equivalent. As $C$ and $D$ have the same
edge/face incidence relation, the edges $f$
and $f'$ of $D$ are $D$-equivalent. So any two edges of $D$ are $D$-equivalent. Hence $D$ is
connected.
\end{proof}

First, we prove the following, which is reminiscent of euler's formula.

 \begin{lem}\label{geq}
Let $C$ be a locally connected $p$-nullhomologous \scom\ with a planar rotation system and $D$ the
dual complex.
Then
\[
 |V(C)|-|E|+|F|-|V(D)|\geq 0
\]

Moreover, we have equality if and only if $D$ is $p$-nullhomologous.
\end{lem}

\begin{proof}
Let $Z_C$ be the dimension over $\Fbb_p$ of the cycle
space of $C$. Similarly we define $Z_D$.
Let $r$ be the rank of the edge-face-incidence matrix over $\Fbb_p$. Note that $r\leq Z_D$ and that
$r=Z_C$ as $H_1(C,\Fbb_p)=0$.
So $Z_D-Z_C\geq 0$. Hence it suffices to prove the following.
 \begin{sublem}\label{euler_cycle_space}
\[ |V(C)|-|E|+|F|-|V(D)|= Z_D-Z_C \]
 \end{sublem}
 \begin{cproof}
 Let $k_C$ be the number of connected components of $C$ and $k_D$ be the number of connected
components of $D$.
Recall that the space orthogonal to the cycle space (over $\Fbb_p$) in a graph $G$ has dimension
$|V(G)|$ minus the number of connected
components of $G$. Hence $Z_C=|E|-|V(C)|+k_C$ and $Z_D=|F|-|V(D)|+k_D$. Subtracting the first
equation from the second yields:
\[
|V(C)|-|E|+|F|-|V(D)|+(k_D-k_C)= Z_D-Z_C
\]
 Since the dual complex of the disjoint union of two simplicial complexes (with planar rotation
systems) is the disjoint union of their
dual
complexes, $k_C\leq k_D$. By \autoref{D_con} $k_C= k_D$. Plugging this into the equation before,
proves the claim.
 \end{cproof}

 This completes the proof of the inequality. We have equality if and only if $r= Z_D$. So the
`Moreover'-part follows.
 \end{proof}

Our next goal is to prove the following, which is also reminiscent of euler's formula but here the
inequality goes the other way round.

\begin{lem}\label{euler_double_counting}
 Let $C$ be a locally connected \scom\ with a planar rotation system $\Sigma$ and $D$ the dual
complex.
 Then: \[
 |V(C)|-|E|+|F|-|V(D)|\leq 0 \] with equality if and only if all link complexes for
$(D,\Sigma_C)$
are spheres.
\end{lem}

Before we can prove this, we need some preparation.
 By $a$ we denote the sum of the faces of link complexes for
$(C,\Sigma)$. By $a'$ we denote the sum over the faces of link complexes for
$(D,\Sigma_C)$.
  Before proving that $a$ is equal to $a'$ we prove that it is useful by showing the following.

 \begin{fact}\label{claim1}
  \autoref{euler_double_counting} is true if $a=a'$ and all link complexes for
$(D,\Sigma_C)$ are connected.
 \end{fact}

\begin{proof}
Given a face $f$ of $C$, we denote the number of edges incident with $f$ by
$deg(f)$.
Our first aim is to prove that

\begin{equation}\label{eq1}
  2 |V(C)|= 2 |E|-\sum_{f\in F} \degree(f) + a
\end{equation}

To prove this equation, we apply Euler's formula \cite{DiestelBookCurrent} in the link complexes
for $(C,\Sigma)$.
Then we take the sum of all these equations over all $v\in V(C)$.
Since $\Sigma$ is a planar rotation system, all link complexes are a disjoint union of
spheres.
Since $C$ is locally connected, all link complexes are connected and hence are spheres.
So they have euler characteristic two. Thus we get the term $2 |V(C)|$ on
the left hand side.
By definition, $a$ is the sum of the faces of link complexes for
$(C,\Sigma)$.

The term $2 |E|$ is the sum over all vertices of link complexes for
$(C,\Sigma)$.
Indeed, each edge of $C$ between the two vertices $v$ and $w$ of $C$ is a vertex
of precisely the two link complexes for $v$ and $w$.

The term $\sum_{f\in F} \degree(f)$ is the sum over all edges of link complexes for $(C,\Sigma)$.
Indeed, each face $f$ of $C$ is in precisely those link complexes for
vertices on the boundary of $f$.
This completes the proof of \eqref{eq1}.

Secondly, we prove the following inequality using a similar argument.
Given an edge $e$ of $C$, we denote the number of faces incident with $e$ by
$\degree(e)$.
\begin{equation}\label{eq2}
  2 |V(D)| \geq 2 |F|-\sum_{e\in E} \degree(e) + a'
\end{equation}

To prove this, we apply Euler's formula in link complexes
for $(D,\Sigma_C)$, and take the sum over all $v\in V(D)$.
Here we have `$\geq$' instead of `$=$' as we just know by assumption that the link
complexes are connected but they may not be a sphere.
So we have $2 |V(D)|$ on the left and $a'$ is the sum over the faces of link complexes for
$(D,\Sigma_C)$.

The term $2 |F|$ is the sum over all vertices of link complexes for
$(D,\Sigma_C)$.
Indeed, each edge of $D$ between the two different vertices $v$ and $w$ of $D$ is a vertex
of precisely the two link complexes for $v$ and $w$. A loop gives rise to two vertices
in the link graph at the vertex
it is attached to.

The term $\sum_{e\in E} \degree(e)$ is the sum over all edges of link complexes for $(D,\Sigma_C)$.
Indeed, each face $e$ of $D$ is in the link complex at $v$ with multiplicity equal to
the number of times it traverses $v$.
This completes the proof of \eqref{eq2}.

By assumption, $a=a'$.
The sums  $\sum_{f\in F} \degree(f)$ and $\sum_{e\in E} \degree(e)$ both count
the number
of nonzero entries of $A$, so they are equal.
Subtracting \eqref{eq2} from \eqref{eq1}, rearranging and dividing by 2 yields:

\[
 |V(C)|-|E|+|F|-|V(D)|\leq 0
\]
 with equality if and only if all link complexes for $(D,\Sigma_C)$
are spheres.
\end{proof}

 \vspace{0.3 cm}

 Hence our next aim is to prove that $a$ is equal to $a'$.
 First we need some preparation.

 Two cell complexes $C$ and $D$ are \emph{(abstract)
surface
duals}
 if the set of vertices of $C$ is (in bijection with) the set of faces of $D$,
 the set of edges of $C$ is the set of edges of $D$ and the set of faces of $C$ is the set of
vertices of $D$.
 These three bijections preserve incidences.

 \begin{lem}\label{loc_surface_is_inc}
 Let $C$ be a simplicial complex and $\Sigma$ be a rotation system and let $D$ be the dual.
  The surface dual of a local surface $S$ for $(C,\Sigma)$ is equal to the link complex
for $(D,\Sigma_C)$ at the vertex $\ell$
of $D$ that corresponds to $S$.
 \end{lem}

 \begin{proof}
 It is immediate from the definitions that the vertices of the link complex $\bar L$ at $\ell$
are the faces of $S$.
 The edges of $S$ are triples $(e,\vec{f},\vec{g})$, where $e$ is an edge of $C$ and $\vec{f}$ and
$\vec{g}$ are orientations of faces of
$C$ that are related via $e$ and are in the local-surface-equivalence class for $S$.
Hence in $D$, these are triples $(e,\vec{f},\vec{g})$ such that $\vec{f}$ and $\vec{g}$ are
directions of edges that
point towards $\ell$ and $f$ and $g$ are adjacent in the cyclic ordering of the face $e$. This are
precisely the edges of the link graph
$L(\ell)$.
Hence the link graph $L(\ell)$ is the dual graph\footnote{ The \emph{dual graph} of a cell complex
$C$ is the graph $G$
 whose set of vertices is (in bijection with) the set of faces of $C$ and whose set of edges is the
set of edges of $C$.
 The incidence relation between the vertices and edges of $G$ is the same as the incidence
relation between the faces and edges
of $C$.} of the cell complex $S$.

Now we will use the Edmonds-Hefter-Ringel
rotation principle, see {\cite[Theorem 3.2.4]{{MoharThomassen}}}, to deduce that the link complex
$\bar L$ at $\ell$ is the surface dual of $S$. We denote the unique cell complex that is a surface
dual of $S$ by $S^*$. Above we have shown that $\bar L$ and $S^*$ have the same 1-skeleton.
Moreover, the rotation systems at the vertices of the link
complex $\bar L$ are given by the cyclic orientations
in  the local-surface-equivalence class for $S$. By \autoref{loc_is_orientable} these
local-surface-equivalence classes define an orientation of $S$. So $\bar L$ and $S^*$ have the same
rotation systems. Hence by the Edmonds-Hefter-Ringel
rotation principle $\bar L$ and $S^*$ have to be isomorphic. So $\bar L$ is a surface dual of $S$.
 \end{proof}

\begin{proof}[Proof of \autoref{euler_double_counting}.]
 Let $C$ be a locally connected simplicial complex and $\Sigma$ be a rotation system and let $D$
be the dual. Let $\Sigma_C$ be as
defined above.
By \autoref{loc_is_con} and \autoref{loc_surface_is_inc} every link complex for
$(D,\Sigma_C)$ is connected.
By \autoref{claim1}, it suffices to show that the sum over all faces of link complexes
of $C$ with respect to $\Sigma$ is equal
to
the sum over all faces of link complexes for $D$ with respect to $\Sigma_C$.
By \autoref{loc_surface_is_inc}, the second sum is equal to the sum over all vertices of local
surfaces for $(C,\Sigma)$.
This completes the proof by \autoref{loc_inc_AND_loc_surfaces}.
\end{proof}

 \begin{proof}[Proof of \autoref{loc_are_spheres}.]
 Let $C$ be a $p$-nullhomologous locally connected simplicial complex that has a planar rotation
system $\Sigma$.
 Let $D$ be the dual complex.
 Then by \autoref{euler_double_counting} and \autoref{geq}, $C$ and $D$ satisfy Euler's formula,
that is:
 \[
 |V(C)|-|E|+|F|-|V(D)|= 0
\]
 Hence by \autoref{euler_double_counting}  all link complexes for $(D,\Sigma_C)$ are
spheres. By \autoref{loc_surface_is_inc}
these are dual to the local surfaces for $(C,\Sigma)$. Hence all local surfaces for $(C,\Sigma)$
are spheres.
 \end{proof}

The following theorem gives three equivalent characterisations of the class of locally
connected simply connected simplicial complexes embeddable in
 $\Sbb^3$.

\begin{thm}\label{nullt}
Let $C$ be a locally connected simplicial complex embedded into $\Sbb^3$. The following are
equivalent.
\begin{enumerate}
 \item $C$ is simply connected;
\item $C$ is $p$-nullhomologous for some prime $p$;
\item all local surfaces of the planar rotation system induced by the topological embedding are
spheres.
\end{enumerate}
\end{thm}

\begin{proof}
Clearly, 1 implies 2. To see that 2 implies 3, we assume that $C$ is $p$-nullhomologous.
Let $\Sigma$ be the planar rotation system induced by the topological embedding of $C$ into
$\Sbb^3$.
By \autoref{loc_are_spheres} all local
surfaces for
$(C,\Sigma)$ are spheres.

It remains to prove that 3 implies 1.
So assume that $C$ has an embedding into
$\Sbb^3$ such that all local surfaces of the planar rotation system induced by the topological
embedding are spheres.
By treating different connected components separately, we may assume that $C$ is connected. By
\autoref{topo_to_combi} all local surfaces
of the topological embedding are spheres.
Thus 3 implies 1 by \autoref{is_simply_connected2}.

\end{proof}

\begin{rem}
Our proof actually proves the strengthening of \autoref{nullt} with `embedded
into $\Sbb^3$' replaced by  `embedded into a simply connected oriented 3-dimensional compact
manifold.' However this strengthening is equivalent to \autoref{nullt} by Perelman's theorem.
\end{rem}

Recall that in order to prove \autoref{combi_intro_extended}, it suffices to show that every
$p$-nullhomologous simplicial complex $C$ has a piece-wise linear embedding into
$\Sbb^3$ if and only if it is simply connected and $C$ has a planar rotation system.

\begin{proof}[Proof of \autoref{combi_intro_extended}.]
Using an induction argument on the number of cut vertices as in the proof of \autoref{combi_intro},
we may assume that $C$ is locally connected.
If $C$ has a piece-wise linear embedding into
$\Sbb^3$, then it has a planar rotation system and it is simply connected by \autoref{nullt}.
The other direction follows from \autoref{combi_intro}.
\end{proof}

\begin{rem}\label{alg_topo}
One step in proving \autoref{combi_intro_extended} was showing that if a simplicial
complex with trivial first homology group embeds in \Sthree, then it must be simply connected.
In this section we have given a proof that only uses elementary topology. We will use these methods
again in \cite{3space4}.

However there is a shorter proof of this fact, which we shall sketch in the following.
Let $C$ be a simplicial complex embedded in \Sthree\ such that one local surface of the embedding
is not a sphere. Our aim is to show that the first homology group of $C$ cannot be trivial.

We will rely on the fact that the first homology group of $X=\Sbb^3\sm \Sbb^1$ is not trivial.
It suffices to show that the homology group of $X$ is a quotient of the homology group of $C$.
Since here by Hurewicz's theorem, the homology group is the abelisation of the fundamental
group, it suffices to show that the fundamental group $\pi_1(X)$ of $X$ is a quotient of the
fundamental group $\pi_1(C)$.

We let $C_1$ be a small open neighbourhood of $C$ in the embedding of $C$ in \Sthree. Since $C_1$
has a deformation retract onto $C$, it has the same fundamental group.
We obtain $C_2$ from $C_1$ by attaching the interiors of all local surfaces of the embedding
except for one -- which is not a sphere. This can be done by attaching finitely many 3-balls.
Similar as in the proof of \autoref{is simply connected}, one can use Van Kampen's theorem to show
that the fundamental group of $C_2$ is a quotient of the fundamental group of $C_1$. By adding
finitely many spheres if necessary and arguing as above one may assume that remaining local surface
is a torus. Hence $C_2$ has the same fundamental group as $X$. This completes the sketch.
\end{rem}

\section{Embedding general simplicial complexes}\label{beyond}

There are three classes of simplicial complexes that naturally include the simply connected
simplicial complexes:
the $p$-nullhomologous ones that are included in those with abelian fundamental group that in turn
are included in
general simplicial complexes. \autoref{combi_intro_extended} characterises embeddability of
$p$-nullhomologous complexes. In this section we prove embedding results for the later two
classes. The
bigger the class gets, the stronger
assumptions we will require in order to guarantee topological embeddings into $\Sbb^3$.

A \emph{curve system} of a surface $S$ of genus $g$ is a choice of at most $g$ genus reducing
curves in $S$ that are disjoint.
An \emph{extension} of a rotation system $\Sigma$ is a choice of curve system at every local
surface of $\Sigma$.
An extension of a rotation system of a complex $C$ is \emph{simply connected} if the topological
space obtained from $C$ by
gluing\footnote{We
stress that the curves need not go through edges of $C$. `Gluing' here is on the level of
topological spaces not of complexes.} a disc at
each curve of the extension is simply connected.
The definition of a \emph{$p$-nullhomologous extension} is the same with `$p$-nullhomologous' in
place of
`simply connected'.

\begin{thm}\label{general}
 Let $C$ be a connected and locally connected simplicial complex with a rotation system
$\Sigma$.
 The following are equivalent.
 \begin{enumerate}
  \item $\Sigma$ is induced by a topological embedding of $C$ into $\Sbb^3$.
  \item $\Sigma$ is a planar rotation system that has a simply connected extension.
  \item We can subdivide edges of $C$, do baricentric subdivision of faces and add new faces such
that the resulting simplicial complex is
simply connected and has a topological embedding into \Sthree\ whose induced planar rotation system
$\Sigma'$ `induces' $\Sigma$.
 \end{enumerate}
\end{thm}
Here we define that `$\Sigma'$ induces $\Sigma$' in the obvious way as follows.
Let $C$ be a simplicial complex obtained from a simplicial complex $C'$ by deleting faces.
A rotation system $\Sigma=(\sigma(e)|e\in E(C))$ of $C$ is \emph{induced} by a rotation system
$\Sigma'=(\sigma'(e)|e\in E(C))$ of
$C'$ if $\sigma(e)$ is the restriction of $\sigma'(e)$ to the faces incident with $e$.
If $C$ is obtained from contracting edges of $C'$ instead, a rotation system $\Sigma$ of $C$ is
\emph{induced} by a rotation system
$\Sigma'$ of
$C'$ if $\Sigma$ is the restriction of $\Sigma'$ to those edges that are in $C$.
If $C'$ is obtained from $C$ by a baricentric subdivision of a face $f$ we take the same definition
of `induced', where we make the
identification between the face $f$ of $C$ and all faces of $C'$ obtained by subdividing $f$.
Now in the situation of \autoref{general}, we say that $\Sigma'$ \emph{induces} $\Sigma$ if there
is a chain of planar rotation systems
each inducing the next one starting with $\Sigma'$ and ending with $\Sigma$.

Before we can prove \autoref{general}, we need some preparation.
The following is a consequence of the Loop Theorem \cite{{Pap57},{Hatcher3notes}}.

\begin{lem}\label{cut_along_disc}
 Let $X$ be an orientable surface of genus $g\geq 1$ embedded topologically into $\Rbb^3$, then
there is a genus reducing
circle\footnote{A \emph{circle} is a topological space homeomorphic to $\mathbb{S}^1$. } $\gamma$
 of $X$ and a disc $D$ with boundary $\gamma$ and all interior points of $D$ are contained in the
interior of $X$.
\end{lem}

\begin{cor}\label{cut_along_discs}
  Let $X$ be an orientable surface of genus $g\geq 1$ embedded topologically into $\Rbb^3$, then
there are genus reducing circles
$\gamma_1$,..., $\gamma_g$ of $X$ and closed discs $D_i$ with boundary $\gamma_i$
such that the $D_i$  are disjoint and the interior points of the discs $D_i$ are contained in the
interior of $X$.
\end{cor}

\begin{proof}
 We prove this by induction on $g$. In the induction step we cut of the current surface along $D$.
Then we the apply
\autoref{cut_along_disc} to that new surface.
\end{proof}

\begin{proof}[Proof of \autoref{general}.]
1 is immediately implied by 3.

Next assume that $\Sigma$ is induced by a topological embedding of $C$ into $\Sbb^3$. Then $\Sigma$
is clearly a planar rotation system.
It has a simply connected extension by \autoref{cut_along_discs}. Hence 1 implies 2.

Next assume that $\Sigma$ is a planar rotation system that has a simply connected extension. We
can clearly subdivide edges and do
baricentric subdivision and change the curves of the curve system of the simply connected extension
such that in the resulting simplicial
complex $C'$ all the curves of the simply connected extension closed are walks in the 1-skeleton of
$C'$.
We define a planar rotation system $\Sigma'$ of $C'$ that induces $\Sigma$ as follows.
If we subdivide an edge, we assign to both copies the cyclic orientation of the original edge. If
we do a baricentric subdivision, we
assign to all new edges the unique cyclic orientation of size two.
Iterating this during the construction of $C'$ defines $\Sigma'=(\sigma'(e)|e\in E(C))$, which
clearly is a planar rotation system that
induces $\Sigma$.
By construction $\Sigma'$ has a simply connected extension such that all its curves are walks in
the 1-skeleton of $C'$ .

Informally, we obtain $C''$ from $C'$ by attaching a disc at the boundary of each curve of the
simply connected extension.
Formally, we obtain $C''$ from $C'$ by first adding a face for each curve $\gamma$ in the simply
connected extension whose boundary is
the closed walk $\gamma$. Then we do a baricentric subdivision to all these newly added faces. This
ensures that
$C''$ is a simplicial complex.
Since $C$ is locally connected, also $C''$ is locally connected.
Since the geometric realisation of $C''$ is equal to the geometric realisation of $C$, which is
simply connected, the simplicial complex $C''$
is simply connected.

Each newly added face $f$ corresponds to a traversal of a curve $\gamma$ of some edge $e$ of $C'$.
This
traversal is a unique edge of the local surface $S$ to whose curve system $\gamma$ belongs. For
later reference we denote that copy of $e$
in $S$ by $e_f$.

We define a rotation system $\Sigma''=(\sigma''(e)|e\in E(C))$ of $C''$ as follows.
All edges of $C''$ that are not edges of $C'$ are incident with precisely two faces. We take the
unique cyclic ordering of size two there.

Next we define $\sigma''(e)$ at edges $e$ of $C'$ that are incident with newly added faces.
If $e$ is only incident with a single face of $C'$,
then $e$ is only in a single
local surface and it only has one copy in that local surface. Since the curves at that local
surface are disjoint. We could have only added
a single face incident with $e$. We take for $\sigma''(e)$ the unique
cyclic orientation of size two at $e$.

So from now assume that $e$ is incident with at least two faces of $C'$.
In order to define $\sigma''(e)$, we start with $\sigma'(e)$ and define in the following for each
newly added face in
between which two cyclic orientations of faces adjacent in $\sigma'(e)$ we put it. We shall ensure
that between any two orientations we put
at most one new face. Recall that two cyclic orientations $\vec{f_1}$ and $\vec{f_2}$ of faces
$f_1$ and $f_2$, respectively, are adjacent
in $\sigma'(e)$ if and only if there is a clone $e'$ of $e$ in a local surface $S$ for
$(C',\Sigma')$ containing $\vec{f_1}$ and
$\vec{f_2}$ such that $e'$ is incident with $\vec{f_1}$ and $\vec{f_2}$ in $S$.
Let $f$ be a face newly added to $C''$ at $e$.  Let $\gamma_f$ be the curve from which $f$ is build
and let $S_f$ be the local surface
that has $\gamma_f$ in its curve system. Let $e_f$ be the copy of $e$ in $S_f$ that corresponds to
$f$ as defined above.  when we
consider $f$ has a face obtained from the disc glued at  $\gamma_f$.
We add $f$ to $\sigma'(e)$ in between the two cyclic orientations that are incident with $e_f$ in
$S_f$.
This completes the definition of $\Sigma''$.
Since the copies $e_f$ are distinct for different faces $f$, the rotation system $\Sigma''$ is
well-defined.
By construction $\Sigma''$ induces $\Sigma$. We prove the following.

\begin{sublem}\label{primeprime}
 $\Sigma''$ is a planar rotation system of $C''$.
\end{sublem}

\begin{cproof}
Let $v$ be a vertex of $C''$. If $v$ is not a vertex of $C'$, then the link graph at $v$
is a cycle. Hence the link complex at $v$ is clearly a sphere. Hence we may assume that $v$ is a
vertex of $C'$.

Our strategy to show that the link complex $S''$ at $v$ for $(C'',\Sigma'')$ is a sphere
will be to show that it is obtained
from the link complex $S'$ for $(C',\Sigma')$ by adding edges in such a way that each
newly added edge traverses a face of $S'$
and two newly added edges traverse different face of $S'$.

So let $f$ be a newly added face incident with $v$ of $C'$.
Let $x$ and $y$ be the two edges of $f$ incident with $v$.
We make use of the notations $\gamma_f$, $S_f$, $x_f$ and $y_f$ defined above.
Let $v_f$ be the unique vertex of $S_f$ traversed by $\gamma_f$ in between $x_f$ and $y_f$.
By \autoref{loc_inc_AND_loc_surfaces} there is a unique face $z_f$ of $S'$ mapped by the map
$\iota$ of that lemma to
$v_f$.
And $x$ and $y$ are vertices in the boundary of $z_f$. The edges on the boundary of $z_f$ incident
with $x$ and $y$ are the cyclic
orientations of the faces that are incident with $x_f$ and $y_f$ in $S_f$. Hence in $S''$ the edge
$f$ traverses the face $z_f$.

It remains to show that the faces $z_f$ of $S'$ are distinct for different newly added faces $f$ of
$C''$. For that it suffices by
\autoref{loc_inc_AND_loc_surfaces} to show that the vertices $v_f$ are distinct. This is true as
curves for $S_f$ traverse a vertex of $S_f$
at most once and different curves for $S_f$ are disjoint.
\end{cproof}

 Since $\Sigma''$ is a planar rotation system of the locally connected simplicial complex $C''$ and
$C''$ is simply connected, $\Sigma''$ is
induced by a topological embedding of
$C''$ into \Sthree\ by \autoref{combi}. Hence 2 implies 3.
\end{proof}

A natural weakening of the property that $C$ is simply connected is that the fundamental group of
$C$ is abelian.
Note that this is equivalent to the condition that every chain that is $p$-nullhomologous is
simply
connected.

\begin{thm}\label{abelian_alt}
 Let $C$ be a connected and locally connected simplicial complex with abelian fundamental group.
Then $C$ has a topological embedding
into $\Sbb^3$ if
and only if it has a planar rotation system $\Sigma$ that has a $p$-nullhomologous extension.
\end{thm}

In order to prove \autoref{abelian_alt}, we prove the following.

\begin{lem}\label{abi}
 A $p$-nullhomologous extension of a planar rotation system of a simplicial complex $C$ with
abelian
fundamental group is a simply connected
extension.
\end{lem}

\begin{proof}
 Let $C'$ be the topological space obtained from $C$ by gluing discs along the curves of the
$p$-nullhomologous extension. The fundamental
group $\pi'$ of $C'$ is a quotient of the fundamental group $\pi$ of $C$, see for example
{\cite[Proposition 1.26]{Hatcher}}.
Since $\pi$ is abelian by assumption, also $\pi'$ is abelian. That is, it is equal to its
abelisation, which is trivial by assumption.
Hence $C'$ is simply connected.
\end{proof}

\begin{proof}[Proof of \autoref{abelian_alt}.]
 If $C$ has a topological embedding into $\Sbb^3$, then by \autoref{general} it has a planar
rotation system that has a $p$-nullhomologous
extension.
 If $C$ has a planar rotation system that has a $p$-nullhomologous
extension, then that extension is simply connected by \autoref{abi}. Hence $C$ has a topological
embedding into $\Sbb^3$ by the other
implication of \autoref{general}.
\end{proof}

\section{Non-orientable 3-manifolds}\label{non-or}

By \autoref{prs_orient} and \autoref{is_manifold} a 2-complex $C$ is embeddable in an
orientable 3-manifold if and only if it has a planar rotation system.
Here we discuss a notion similar to planar rotation systems that can be used to characterise
embeddability in 3-manifolds combinatorially.

Recall that a rotation framework of a 2-complex $C$ is a choice of embedding of all its
link graphs in the plane so that for each edge $e$ of $C$ in the two link graphs at the endvertices
of the edge $e$, the rotators at the vertex $e$ are reverse or agree. We colour an edge \emph{red}
if they agree.
A \emph{generalised planar rotation system} is a rotation framework with the further condition that for every face $f$ of $C$ its number of red
incident edges is even.

\begin{rem}
Any planar rotation system defines a generalised planar rotation system. In this induced
generalised planar rotation system no edge is red.
\end{rem}

\begin{rem}
Let $C$ be simplicial complex whose first homology group over the binary field $\Fbb_2$ is trivial,
for example this includes the case where $C$ is simply connected. Then $C$ has a planar rotation
system if and only if it has a generalised planar rotation system by \autoref{rot_system_exists-simplycon}.
\end{rem}

\begin{rem}\label{trivial}
A locally 3-connected 2-complex has a generalised planar rotation system if and only if every choice
of embeddings of its link graphs in the plane is a generalised planar rotation system. Indeed, any
face shares an even number of edges with every vertex. By local 3-connectedness, the embeddings of a
link graph at some vertex $v$ is unique up to orientation, which reverses all rotators at edges
incident with $v$. So it flips red edges to non-red ones and vice versa. Hence the number of red
edges modulo two does not change.

Thus there is a trivial algorithm that verifies whether a locally 3-connected 2-complex has a
generalised planar rotation system: embed the link graphs in the plane arbitrarily (if possible) and
then check the condition at every edge and face.
\end{rem}

The following was proved by Skopenkov.

\begin{thm}[\cite{Skopenkov94}]
The following are equivalent for a 2-complex $C$.
\begin{itemize}
\item $C$ has a generalised planar rotation system;
\item $C$ has a thickening that is a 3-manifold;
\item $C$ has a piece-wise linear embedding in some (not necessarily compact) 3-manifold.
\end{itemize}
\end{thm}

\begin{rem}
For 2-dimensional manifolds there are two types of rotation systems: one for the orientable case
and one for the general case. In dimension three the situation is analoguous. For orientable
3-manifolds we have `planar rotation systems' and `generalised planar
rotation systems' for general 3-manifolds.
\end{rem}

\begin{rem}
Similarly as for planar rotation systems, the methods of this series of papers can also be used to
also give a polynomial time algorithm that verifies the existence of generalised planar rotation
systems. As mentioned in \autoref{trivial} the locally 3-connected case has a polynomial algorithm
and the reduction to the locally 3-connected case works the same as for planar rotation systems.
\end{rem}

\begin{rem}
 In \autoref{concl77} we explain how in an earlier version of this paper we characterise the existence of a planar rotation system, and thus embeddability in an orientable 3-manifold, in terms of excluded minors. Similarly, one can characterise the existence of a generalised planar rotation system (and thus embeddability in a 3-manifold) in terms of excluded minors, see \cite{3space2} for details.
\end{rem}

\chapter{A refined Kuratowski-type
characterisation}\label{chapterV}

\section{Abstract}
Building on earlier chapters,
we prove an analogue of Kuratowski's characterisation of graph planarity for three
dimensions.

More precisely, a simply connected 2-dimensional simplicial complex embeds in 3-space
if and only if it has no obstruction from an explicit list. This list of obstructions is finite
except for one infinite family.

\section{Introduction}

We assume that the reader is familiar with \autoref{chapterI}. In that chapter we prove that a
locally
3-connected simply connected 2-dimensional simplicial complex has a
topological embedding into 3-space if and only if it has no space minor from a finite
explicit list $\Zcal$ of obstructions (based on a topological theorem proved in \autoref{chapterII}). The purpose of this chapter is to extend that theorem beyond
locally 3-connected (2-dimensional) simplicial complexes to simply connected
simplicial complexes in general.

\vspace{0.3 cm}

The first question one might ask in this direction is whether the assumption of local
3-connectedness could simply be dropped from the result of \autoref{chapterI}. Unfortunately this is
not
true. One new obstruction can be constructed from the M\"obius-strip as follows.

\begin{eg}
One can obtain the M\"obius-strip from an untwisted strip as follows. Take one of the boundary components. This is homeomorphic to $\Sbb^1$ and now identify antipodal points.
Similarly, one can start with an untwisted strip and now on one boundary component identify every point with its rotation by $120$ degrees. Now identify these two spaces at their central cycles; that is the images of the boundary component where the gluing took place (via some homeomorphism of $\Sbb^1$).

In a few lines we explain why this topological space $X$ cannot be embedded in 3-space. Any
triangulation of $X$ gives an obstruction to embeddability. It can be shown that such
triangulations have no space minor in the finite list $\Zcal$.

Why can $X$ not be embedded in 3-space? To answer this, consider a small torus around the
central cycle. The two glued spaces each intersect that torus
in a circle. These circles however have a different homotopy class in the torus. Since any two
circles in the torus of a different homotopy class intersect\footnote{A simple way to see this is
to note that the torus with a circle removed is an annulus. }, the space $X$ cannot
be embedded in 3-space without intersections of the two glued spaces. Obstructions of
this type we call \emph{torus crossing obstructions}. A precise definition is given in
\autoref{s0}.

\end{eg}

A refined question might now be whether the result of  \autoref{chapterI} extends to simply
connected
simplicial complex if we add the list $\Tcal$ of torus crossing obstructions to the list $\Zcal$ of
obstructions. The answer to this question is `almost yes'. Indeed, we just need to add to the space
minor operation the operations of stretching defined in \autoref{s3} and add a finite set to $\Zcal$; we refer to this larger finite set as $\Ycal$.

The stretching operations are illustrated in \autoref{fig:s_pair}, \autoref{fig:stretch_branch} and
\autoref{fig:stretch_edge}.
It is not hard to show that
stretching preserves embeddability. The
main result of this chapter is the following.

\begin{thm}\label{Kura_gen}\label{Kura_simply_con}
  Let $C$ be a 2-dimensional simplicial complex such that the first homology
group $H_1(C,\Fbb_p)$ is trivial for some prime $p$. The following are equivalent.
 \begin{itemize}
  \item $C$ has a topological embedding in 3-space;
  \item $C$ is simply connected and has no stretching that has a space minor in $\Ycal\cup \Tcal$.
 \end{itemize}
\end{thm}

We deduce \autoref{Kura_gen} from the results of \autoref{chapterI} in two steps as follows.
The notion of `local almost 3-connectedness and stretched out' is slightly more general and more
technical than `local 3-connectedness', see \autoref{s0} for a definition. First we extend the
results of
\autoref{chapterI} to locally almost 3-connected and stretched out simply connected
simplicial complexes, see \autoref{Kura_simply_con2} below.
We conclude the proof by
showing that any simplicial complex can be stretched to a locally almost 3-connected and
stretched out one. More
precisely:

\begin{thm}\label{main_streching}
 For any simplicial complex $C$, there is a simplicial complex $C'$ obtained from $C$ by stretching
so that $C'$ is
locally almost 3-connected and stretched out or $C'$ has a non-planar link.

Moreover $C$ has a planar rotation system if and only if $C'$ has a planar rotation system.
\end{thm}

The overall structure of the argument is similar to that for problems in structural graph theory
with `3-connected kernel' (in such arguments one first proves the 3-connected case, then in a
second step deduces the 2-connected case and then finally deduces the general case).
In \autoref{s0} we prove
the extension of the results of \autoref{chapterI} to locally almost 3-connected and stretched out
simplicial complexes, \autoref{Kura_almost}.
In \autoref{s2} we develop the tools to extend this to the locally almost 2-connected case.
In \autoref{s1} and \autoref{s3}  we extend \autoref{Kura_almost} to general simplicial complexes,
which proves \autoref{main_streching}. Then we prove
\autoref{Kura_simply_con}. Finally in \autoref{algo_sec} we describe algorithmic consequences.

For graph theoretic definitions we refer the
reader to \cite{DiestelBookCurrent}.

\section{A Kuratowski theorem for locally almost 3-connected simply connected simplicial
complexes}\label{s0}

This section is subdivided into five subsections. In the first two we introduce the context for \autoref{Kura_simply_con2}, and in the next three subsections we prove it.

\subsection{Torus crossing obstructions}

In this section we prove \autoref{Kura_simply_con2}, which is used in the proof of the main theorem.
First we define the list $\Tcal$ of torus crossing obstructions.

Given a simplicial complex $C$, a \emph{mega face} $F=(f_i|i\in \Zbb_n)$ is a cyclic
orientation of faces $f_i$ of $C$ together with for every $i\in \Zbb_n$ an edge
$e_i$ of $C$ that is only incident with $f_i$ and $f_{i+1}$ such that the $e_i$ and $f_i$ are
locally distinct, that is, $e_i\neq e_{i+1}$ and $f_i\neq f_{i+1}$ for all $i\in \Zbb_n$. We remark
that since in a simplicial
complex any two faces can share at most one edge, the edges $e_i$ are implicitly given by the faces
$f_i$.
A \emph{boundary component} of a mega
face $F$ is a connected component of the 1-skeleton of $C$ restricted to the faces $f_i$ after we
topologically delete the edges $e_i$.
Given a cycle $o$ that is a boundary component of a mega face $F$,
we say that $F$ is
\emph{locally monotone} at $o$ if for every edge $e$ of $o$ and each face $f_i$ containing $e$, the
next face of $F$ after $f_i$ that contains an edge of $o$ contains the unique edge of $o$ that has
an endvertex in common with $e$ and $e_{i+1}$.
Under these assumptions for each edge $e$ of $o$ the number of indices $i$ such that $e$ is
incident with $f_i$ is the same. This number is called the  \emph{winding
number} of
$F$ at $o$.

A \emph{torus crossing obstruction} is a simplicial complex $C$ with a cycle $o$ (called the
\emph{base cycle}) whose faces can be partitioned into two mega faces that both have $o$ has a
boundary component and are locally monotone at $o$ but with different winding numbers. We denote
the set of torus crossing
obstructions by $\Tcal$.

\begin{rem}The set of torus crossing obstructions is infinite. Indeed, it contains at least one
member for every pair of distinct winding numbers. So it is not possible to reduce it to a finite
set. However one can further reduce torus crossing obstruction as follows. First, by working with
the class of 3-bounded 2-complexes as defined in \autoref{chapterI} instead of simplicial complexes,
one
may assume that the cycle $o$ is a loop. Secondly, one may introduce the further operation of
gluing
two faces along an edge if that edge is only incident with these two faces. This way one can glue
the two mega faces into single faces. Thirdly, one can
enlarge the holes of the mega faces to make them into one big hole (after contracting edges
one may assume that this single hole is bounded by a loop). After all these steps we
only have one torus crossing obstruction left for any pair of distinct winding numbers. This
obstruction consists of three vertex-disjoint loops and two faces, each incident with two loops.
The
loop contained by both faces is the base cycle $o$. Here the faces may have winding number greater
than one. The faces have winding number precisely one at the other loops.
\end{rem}

A \emph{parallel graph} consists of two vertices, called the \emph{branch vertices}, and a set of
disjoint paths between them. Put another way, start with a graph with only two vertices and all
edges going between these two vertices, now subdivide these edges arbitrarily,
see \autoref{fig:para5}. For example, parallel graphs where the branch vertices have degree two are cycles.

   \begin{figure} [htpb]
\begin{center}
   	  \includegraphics[height=2cm]{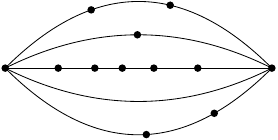}
   	  \caption{A parallel graph with five paths.}\label{fig:para5}
\end{center}\vspace{-0.7cm}
   \end{figure}
A \emph{para-cycle} in a simplicial complex $C$ is a cycle $o$ such that all its edges have face-degree at least three and all link graphs at vertices of $o$ are parallel graphs.
   A para-cycle $o$ is \emph{planar} if there is a choice of embedding of the link graph at each vertex of $o$ so that for each edge $e$ of $o$ the two rotators at $e$ in the link graphs at its endvertices are reverse of one another.

\begin{lem}\label{has_obstruction}
  Let $C$ be a simplicial complex with a non-planar para-cycle $o$.
Then a torus crossing obstruction can be obtained from $C$ by deleting faces.
\end{lem}

\begin{proof}
First we prove the lemma under the additional assumption that $o$ has length at least four and is chordless.
Our aim is to define a torus crossing obstruction with base cycle $o$. For that we define a set of
possible mega faces as follows.
By deleting faces of $C$ that do not have a vertex in the cycle $o$ if necessary, assume that all faces of $C$ have an endvertex on $o$.

Let $X$ be the set of edges of $C$ that are not on $o$ but have an endvertex on $o$.
Every edge in $X$ is incident with exactly two faces. And every face of $C$ is incident with exactly two edges of $X$ since $o$ has length at least four and no chords.
This defines a partition of the set of faces of $C$ into mega-faces, where the two edges adjacent with a face are its two unique edges in $X$.

Since $o$ is a para-cycle, the cycle $o$ is a
boundary component of each of these mega-faces. It is straightforward to check that these mega-faces are locally
monotone at $o$.

It suffices to show that two of these mega faces have distinct winding number at $o$.
Suppose not for a contradiction. Then all mega faces have the same winding number (this includes the case that there is only one mega face).

We enumerate the above mega faces and let $K$ be their total number.
We denote by $W$ the common winding number; that is, for all of the $K$ mega faces each edge of $o$ is in exactly $W$ faces of each mega face.
By $f[k,w,e]$ (or $f[k,w]$ if $e$ is implicitly given) we denote the $w$-th face incident with $e\in o$ on the
$k$-th mega face, where $k$ and $w$ are in the cyclic groups $\Zbb_K$ and $\Zbb_W$, respectively (with respect to a fixed linear order of the faces of a mega face that is induced by its unique cyclic ordering of its faces).
We define the rotator at the edge $e$ to be $f[1,1]$,
$f[2,1]$,  \ldots, $f[K,1]$, $f[1,2]$,  $f[2,2]$, \ldots, $f[K,2]$, $f[1,3]$, \ldots ,\ldots,
$f[K,W]$, $f[1,1]$. All other edges of $C$ have face degree at most two and thus a unique rotator. So these choices define a rotation system of $C$.
Since all the link graphs at vertices of $o$ are parallel graphs, this defines a plane embedding of them.
So this rotation system is planar, and witnesses that the para-cycle $o$ is planar.
This is the
desired contradiction to our assumption.
Hence two mega faces must have a different winding number. So $C$ contains a torus crossing
obstruction. This is the
desired contradiction to our assumption.
Hence two mega faces must have a different winding number. So $C$ contains a torus crossing
obstruction.

Having completed the proof under the additional assumption that $o$ has length at least four and no chords, we now deduce the general case from that special case.
We consider the operation of subdividing an edge, which subdivides an edge and all the resulting faces of size four into two triangular faces such that the newly introduced vertex is incident with all newly introduced edges. If such a subdivision has a torus crossing obstruction, then also $C$ has a torus crossing obstruction (simply replace subdivision-faces by the primal faces containing them).
Via subdivisions we first ensure that $o$ has length at least four and then we subdivide every chord, to obtain a chordless cycle. Since subdivisions of parallel graphs are parallel graphs, we can apply the special case now and deduce the general case.
\end{proof}

\subsection{Statement of \autoref{Kura_simply_con2}}

Let $G$ be a link graph of a vertex of a 2-complex $C$ incident with a single loop $\ell$. We say that $G$ is \emph{helicopter planar} if $G$ has a plane embedding so that the rotators at the two vertices of $G$ corresponding to $\ell$ are reverse of one another -- where two edges incident with these vertices are in bijection if they come from the same face of $C$ (this is unambigious as $C$ has exactly one loop).
A \emph{helicopter complex} is a 2-complex with a single loop $\ell$ such that the link graph at the unique vertex incident with $\ell$ is not helicopter planar.

\begin{obs}\label{heli_no_PRS}
 A helicopter complex does not admit a planar rotation system.
 \qed
\end{obs}

A helicopter complex is \emph{nice} if its link graph at the unique vertex incident with its loop is a subdivision of a 3-connected graph.

\begin{lem}\label{minimal-helicopters}
There is a finite set $\Xcal$ of helicopter complexes so that every nice helicopter complex has a space minor in $\Xcal$.
\end{lem}

\begin{rem}
\autoref{minimal-helicopters} will be proved in \autoref{sec:heli} below.
 The finite list $\Xcal$ of \autoref{minimal-helicopters} is computed explicitly in
 {\cite[Section: Marked graphs]{{3space1}}}; however, here we only prove existence of a finite set as this is much shorter and contains the central ideas.
\end{rem}

Denote by $\Ycal$ the list of the nine 2-complexes of \autoref{final-list} together with the finite list $\Xcal$ from \autoref{minimal-helicopters}.

   \begin{figure} [htpb]
\begin{center}
   	  \includegraphics[height=2cm]{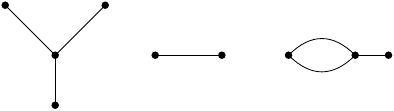}
   	  \caption{A 3-star, an edge and a 2-cycle with an
attached leaf. Free-graphs are subdivisions of these graphs.}\label{fig:starstarround}
\end{center}\vspace{-0.7cm}
   \end{figure}

A \emph{free-graph} is a subdivision of a 3-star, a path or a cycle with an attached path,
see \autoref{fig:starstarround}.
These graphs are `free' in the sense that any rotation system on them defines an embedding in the
plane.
A graph is \emph{almost 3-connected} if it is a subdivision of a 3-connected graph, a parallel
graph or a free-graph.
A simplicial complex is \emph{locally almost 3-connected} if all its link
graphs are almost 3-connected.

\begin{rem}(Motivation) Next we define `stretched out'.
 This is a technical condition, which is used only three times in the
argument,
namely in the proofs of \autoref{detour_via_5}, \autoref{x13} and \autoref{41_2} below. We remark that the notion of
stretched out as defined here is only intended to be useful for locally 3-connected 2-complexes.
In this context, it very roughly says that every edge of degree two has an endvertex whose link graph is \lq very
simple\rq.
\end{rem}

A simplicial complex is \emph{stretched out} if
 every edge incident with only two faces has an endvertex $x$ such that the link graph at $x$
is not a subdivision of a 3-connected graph and not a parallel graph whose branching vertices
have degree at least three.

\begin{thm}\label{Kura_simply_con2}\label{Kura_almost}\label{sum1234}
 Let $C$ be a nullhomologous simplicial complex that is locally almost 3-connected and stretched out. The
following are equivalent.
 \begin{itemize}
  \item $C$ has a planar rotation system;
  \item $C$ has no space minor in $\Ycal\cup \Tcal$.
 \end{itemize}
\end{thm}

Next we define `para-paths' (which are similar to para-cycles) and analyse them.
A path in a simplicial complex $C$ is a \emph{para-path} if
\begin{enumerate}
 \item the link graphs at all interior vertices of $P$ are parallel graphs, where the branching
vertices have degree at least three;
\item the link graphs at the two endvertices of $P$ are subdivisions of 3-connected graphs.
\end{enumerate}

Similar to planarity of para-cycles, we define: a para-path $P$ in a 2-complex $C$ is \emph{planar} if there is a choice of embedding at each vertex of $P$ so that for each edge $e$ of $P$the two rotators at $e$ in the link graphs at its endvertices are reverse of one another; we additionally require that if the endvertices of $P$ are joined by an edge $e'$,
then the chosen embeddings at the endvertices of $P$ evaluated at the rotators of $e'$ are reverse of one another.

As a preparation for the proof of \autoref{Kura_almost}, we prove the following analogue of
\autoref{rot_system_exists}.

\begin{lem}\label{sum123}\label{41_2}
 A null-homological locally almost 3-connected simplicial complex $C$ has one of the following:
 \begin{enumerate}
  \item a planar rotation system;
  \item a non-planar link graph;
  \item a non-planar para-path;
  \item a non-planar para-cycle;
  \item a space minor that is a nice helicopter complex.
 \end{enumerate}
\end{lem}

\subsection{Proof of \autoref{sum123}}

A vertex of a 2-complex is \emph{relevant} if its link graph is a parallel graph where the two branching vertices have degree at least three or if its link graph is a subdivision of a 3-connected graph. An edge is \emph{relevant} if both its endvertices are relevant and it is incident with at least three faces.
A \emph{pre-rotation framework} of a 2-complex $C$ is a choice of a embedding at each link graph at a relevant vertex of $C$
so that at each relevant edge $e$ the two rotators at $e$ in the chosen embeddings of the link graphs at its endvertices are reverse of one another or agree.
We colour a relevant edge \emph{green} if the two rotators are reverse, and we colour if \emph{red} otherwise, in which case they agree.
A pre-rotation framework is \emph{even} if every cycle consisting only of relevant edges contains an even number of red edges; otherwise it is \emph{odd}.
A cycle is \emph{red} if it contains an odd number of red edges, and all its edges are incident with at least three faces and it is not a para-cycle.

\begin{lem}\label{pre-PRF_to_PRS}
 If an almost locally 3-connected 2-complex $C$ admits an even pre-rotation framework, then it admits a planar rotation system.
\end{lem}

\begin{proof}
 First imitate the proof of \autoref{lem101} to modify the pre-rotation system so that all relevant edges are coloured green.
 Now pick plane embeddings at link graphs of remaining vertices, in three steps as follows. Let $X$ be the set of those non-relevant vertices that are incident with an edge of face-degree three to a relevant vertex. Since vertices of $X$ have free link graphs, they are incident with a unique edge of face-degree three and pick an embedding of their link graph so that the rotator at its unique-degree three vertex is the reverse of the rotator picked at the link graph at the relevant endvertex of that edge.
 An edge of face degree three is \emph{very irrelevant} if both endvertices have link graphs that are free. Let $Y$ be the set of endvertices of very irrelevant edges.
 Note that $Y$ is disjoint from $X$. Vertices in $Y$ are incident with a unique edge of face-degree three and it is very irrelevant. Pick an embedding of the two endvertices of that edge so that the rotator at that edge are reverse.
 Now we have defined embeddings at all link graphs of minimum degree at least three. All other link graphs have unique rotation systems and we pick them.
 This defines a rotation framework. Since edges of degree at most two are always green, all edges are green and so this rotation framework gives rise to a planar rotation system as in the proof of \autoref{lem101}.
\end{proof}

\begin{lem}\label{pre_PRF_obstructions}
If an almost locally 3-connected 2-complex does not admit a pre-rotation framework, then it has a non-planar link graph, a non-planar para-path or a non-planar para-cycle.
\end{lem}

\begin{proof}
 We prove the contra-positive: assume that all link graphs at vertices of $C$ are planar, all para-paths are planar and all para-cycles are planar. Our aim is to construct a pre-rotation framework.
 At link graphs that are subdivisions of 3-connected graphs, pick an arbitrary embedding.
 Since single edges are examples of para-paths, their planarity ensures compatibility along edges both whose endvertices have link graphs that are subdivisions of 3-connected graphs.
 Now let $X$ be the set of those relevant edges that have an endvertex whose link graph is a parallel graph. The edge set $X$ -- considered as a subset of the 1-skeleton of $C$ -- spans a graph of maximum degree two. So each connected component of $X$ is a path or a cycle. If it is a cycle, this cycle is a para-cycle, and by assumption it is planar, so pick embeddings at link graphs of its vertices accordingly. If this component is a path, this path is a para-path; call it $P$. The embeddings at the link graphs of the endvertices of $P$ are unique up to reversing, and since the para-path $P$ is planar, we can pick embeddings at link graphs of interior vertices of $P$ so that along each edge of $P$ they are reverse or agree. This defines a pre-rotation frameworks.
\end{proof}

\begin{lem}\label{tiny-arg}
 If a simplicial complex $C$ has a pre-rotation framework but no planar rotation system, then it has a red chordless cycle.
\end{lem}

\begin{proof}
 We obtain $G$ from the 1-skeleton of $C$ by deleting all edges that are not relevant; note that all remaining edges have degree at least three in $C$.
 By assumption, $G$ contains a cycle $o$ with an odd number of red edges. The chordless cycles of $G$ generate all cycles of $G$ over the field $\Fbb_2$, so in particular the cycle $o$.
 The parity of the number of red edges of the sum of two cycles is the sum of their parities. So the family of chordless cycles generating $o$ cannot all contain an even number of red edges.
 Thus in this generating family we find a red chordless cycle.
\end{proof}

\begin{lem}\label{red_to_helicopter}
 If a simplicial complex $C$ has a pre-rotation framework $\Sigma$ with a red chordless cycle $o$, then $C$ has a space minor that is a nice helicopter complex.
\end{lem}

\begin{proof}
Let $e$ be an arbitrary edge of $o$. We obtain $D$ from $C$ by contracting $o-e$. Since $o$ is chordless, $e$ is the only loop of $D$.
Let $v$ be the unique vertex of $D$ that is incident with the loop $e$.
Suppose for a contradiction that the link graph at $v$ in $D$ is helicopter planar.
The link graph $L(v)$ at $v$ is the vertex sum over all link graphs of $C$ at the vertices of the path $o-e$ via the edges of that path.
Then by \autoref{sum_planar2} applied successively along the edges of $o-e$, the embedding of $L(v)$ induces embeddings at all link graphs at vertices of $o$ in $C$ that are compatible along the edges. We denote these embeddings of the link graphs at vertices of $o$ by $(\tau(x)|x\in V(o))$.

We obtain $\Sigma'$ from the pre-rotation framework $\Sigma$ by changing the embeddings at the vertices $x$ of $o$ to $\tau(x)$. Since $o$ is red, by definition it is not a para-cycle and thus contains a vertex whose link graph is a subdivision of a 3-connected graph. Its embedding $\tau(x)$ is the same or the flip of the corresponding embedding for $\Sigma$.
Using this we prove inductively along the edges of $o$ that each embedding $\tau(x)$ is the same or the flip of the corresponding embedding for $\Sigma$.
So $\Sigma'$ is a pre-rotation framework.
The cycle $o$ has only green edges in the pre-rotation framework $\Sigma'$. Modifying the pre-rotation framework $\Sigma$ by flipping embeddings at link graphs of $C$ does not change the parity of the red edges of $o$. So the cycle $o$ contains an even number of red edges with respect to $\Sigma$, a contradiction to the assumption that $o$ is red.
So the link graph at $v$ is not helicopter planar.
The link graph at $v$ is a subdivision of a 3-connected graph, and thus the space minor $D$ is nice.
\end{proof}

We summarise this subsection with the following.

\begin{proof}[Proof of \autoref{sum123}]
By \autoref{pre_PRF_obstructions} assume that $C$ admits a pre-rotation framework. If it is even, by \autoref{pre-PRF_to_PRS}, $C$ admits a planar rotation system and we find one of the outcomes. So assume that it is odd. By \autoref{tiny-arg} $C$ has a red chordless cycle. Then by \autoref{red_to_helicopter} $C$ has a helicopter complex as a space minor; this is outcome 5.
\end{proof}

\subsection{Helicopter minors}\label{sec:heli}

The aim of this section is to prove \autoref{minimal-helicopters}; we begin by introducing our tools.
A \emph{helicopter graph} is a graph $G$ together with two of its vertices $v$ and $w$ and three pairs
$((a_i,b_i)|i=1,2,3)$ of its edges, where the edges $a_i$ are incident with $v$ and the edges $b_i$ are
incident with $w$. We stress that we allow $a_i=b_j$ for some $i,j\in[3]$ and $v=w$.

A helicopter graph $(G,v,w, ((a_i,b_i)|i=1,2,3))$ is \emph{helicopter planar} if there is a planar rotation system
$(\sigma_x|x\in V(G))$ of $G$ such that the rotator $\sigma_v$ at $v$ restricted to $(a_1,a_2,a_3)$ is the reverse
permutation of the rotator
$\sigma_w$ at $w$ restricted to $(b_1,b_2,b_3)$ -- when concatenated with the bijective map $b_i\mapsto
a_i$.
\begin{eg}\label{heli-eg}
Let $C$ be a 2-complex with a single loop $\ell$ and $x$ be the unique vertex of $C$ incident with $\ell$.
Let $(f_1,f_2,f_3)$ be a choice of three faces incident with $\ell$. Denote the vertices corresponding to $\ell$ in the link graph $L(x)$ by $v$ and $w$, and let $a_i$ be the edge of $L(x)$ incident with $v$ corresponding to $f_i$. Similarly, let $b_i$ be the edge of $L(x)$ incident with $w$ corresponding to $f_i$.
Then $(L(x),v,w, ((a_i,b_i)|i=1,2,3))$ is a helicopter graph.
\end{eg}
We refer to the helicopter graph constructed in \autoref{heli-eg} as the helicopter graph \emph{associated} to $C$ at $x$ given the choices $(f_1,f_2,f_3)$.
The connection with helicopter complexes is as follows.

\begin{lem}\label{loop_to_helicopter}
Let $C$ be a 2-complex with a single loop $\ell$ and $x$ be the unique vertex of $C$ incident with $\ell$.
Then $C$ is a helicopter complex if one of the helicopter graphs associated to $C$ at $x$ is not helicopter planar.
\end{lem}

\begin{proof}
If an associated helicopter graph is not helicopter-planar, then the link graph cannot be helicopter planar; and thus $C$ is a helicopter complex.
\end{proof}

A helicopter graph $(G,v,w, ((a_i,b_i)|i=1,2,3))$ is \emph{3-connected} if $G$ is 3-connected.
We abbreviate $A=\{a_1,a_2,a_3\}$ and $B=\{b_1,b_2,b_3\}$.

A \emph{helicopter minor} of a helicopter graph $(G,v,w, ((a_i,b_i)|i=1,2,3))$ is obtained by doing a
series of the following operations:
\begin{itemize}
 \item contracting or deleting an edge not in $A\cup B$;
 \item replacing an edge $a_i\in A\sm B$ and an edge $b_j\in B\sm A$ that are in parallel by a
single new edge which is in that parallel class. In the reduced graph, this new edge is
$a_i$ and $b_j$.
\item the above with `series' in place of `parallel'.
\item apply the bijective map\footnote{That is, we exchange the roles of the vertices \lq $v$\rq\ and \lq $w$\rq\ and the roles of  \lq $a_i$\rq\ and \lq $b_i$\rq\ are interchanged for $i\in [3]$.} $(v,A)\mapsto (w,B)$.
\end{itemize}

\begin{lem}\label{minimal_minor}
 Let $\hat G=(G,v,w, ((a_i,b_i)|i=1,2,3))$ be a helicopter graph such that G is planar.
Let $\hat H$ be a 3-connected helicopter minor of $\hat G$. Then $\hat G$ is helicopter planar if and only if
$\hat H$ is helicopter planar.
\end{lem}

Before we can prove this, we need to recall some facts about rotation systems of graphs.
Given a graph $G$ with a rotation system $\Sigma=(\sigma_v|v\in V(G))$ and an edge $e$. The
rotation
system \emph{induced} by $\Sigma$ on $G-e$ is $(\sigma_v-e|v\in V(G))$. Here $\sigma_v-e$ is
obtained from the cyclic ordering $\sigma_v$ by deleting the edge $e$.
The rotation
system \emph{induced} by $\Sigma$ on $G/e$ is $(\sigma_v|v\in V(G/e)-e)$ together with $\sigma_e$
defined as follows. Let $v$ and $w$ be the two endvertices of $e$. Then $\sigma_e$ is obtained from
the cyclic ordering $\sigma_v$ by replacing the interval $e$ by the interval $\sigma_w-e$ (in
such a way that the predecessor of $e$ in $\sigma_v$ is followed by the successor of $e$ in
$\sigma_w$). Summing up, $\Sigma$ induces a rotation system at every minor of $G$. Since the class
of plane graphs\footnote{A \emph{plane graph} is a graph together with an embedding in the plane. }
is closed under taking minors, rotation systems induced by planar rotation systems are planar.

\begin{proof}[Proof of \autoref{minimal_minor}]
Let $\Sigma$ be a planar rotation system of $G$. Let $\Sigma'$ be the rotation system of the
graph $H$ of $\hat H$ induced by $\Sigma$. As mentioned above, $\Sigma'$ is planar.

Moreover, $\Sigma$ witnesses that  $\hat G$ is helicopter planar if and only if   $\Sigma'$
witnesses that  $\hat H$ is helicopter planar. Hence if $\hat G$ is helicopter planar, so is $\hat H$. Now
assume that $\hat H$ is helicopter planar. Since $H$ is 3-connected, it must be that $\Sigma'$ witnesses that
the helicopter graph $\hat H$ is helicopter planar. Hence the helicopter graph $\hat G$ is helicopter planar.
\end{proof}

Our aim is to characterise when 3-connected helicopter graphs are helicopter planar. By \autoref{minimal_minor}
it suffices to study that question for helicopter-minor minimal 3-connected helicopter
graphs; we call such helicopter graphs  \emph{3-minimal}.

It is reasonable to expect -- and indeed true, see below -- that there are only finitely
many 3-minimal helicopter graphs, as follows.

Let $\hat G=(G,v,w, ((a_i,b_i)|i=1,2,3))$ be a helicopter graph.
We denote by $V_A$ the set of endvertices of edges in $A$ different from $v$.
We denote by $V_B$ the set of endvertices of edges in $B$ different from $w$.

\begin{lem}\label{VA}
 Let $\hat G=(G,v,w, ((a_i,b_i)|i=1,2,3))$ be 3-minimal. Unless $G$ is $K_4$, every edge
in $E(G)\sm (A\cup B)$ has its endvertices either both in $V_A$ or both in
$V_B$.
\end{lem}

 In the proof of \autoref{VA} we shall use the following lemma of Bixby.

 \begin{lem}[Bixby's lemma \cite{bixby1982simple} or {\cite[Lemma 8.7.3]{Oxley2}}]\label{bixby}
  Let $G$ be a 3-connected graph different\footnote{In the original paper, Bixby proves this lemma in the more general context of matroids and the definitions are set up so that the matroids $U_{1,3}$ and $U_{2,3}$ are 3-connected and thus the assumption $G\neq K_4$ is not necessary.} from $K_4$. For every edge $e$ of $G$ the simplification of $G/e$ or the co-simplification of $G-e$ is 3-connected.
 \end{lem}

\begin{proof}[Proof of \autoref{VA}.]
By assumption $G$ is a 3-connected graph with at least five vertices such that any proper helicopter
minor of
$\hat G$ is not 3-connected. Let $e$ be an edge of $G$ that is not in $A\cup B$. By Bixby's Lemma
(\autoref{bixby}) either $G-e$ is 3-connected after suppressing
series
edges or $G/e$ is 3-connected after suppressing parallel edges.

\begin{sublem}\label{del7}
 There is no 3-connected graph $H$ obtained from $G-e$ by suppressing
series edges.
\end{sublem}

\begin{cproof}
Suppose for a contradiction that there is such a graph $H$. As $G$ is 3-connected, every class of
series edges of $G-e$
has size at most two.
By minimality of $G$, there is no helicopter minor of $\hat G$ with graph $H$. Hence one of these
series classes has to contain two edges in $A$ or two edges in $B$. By symmetry, we may assume that
$e$ has an endvertex $x$ that is incident with two edges $e_1$ and $e_2$ in $A$. As $G$ is
3-connected these two adjacent edges of $A$ can only share the vertex $v$. Thus $x=v$. This
is a
contradiction to the assumption that $e_1$ and $e_2$ are in series as $v$ is incident with the
three edges of $A$.
\end{cproof}

By \autoref{del7} and Bixby's Lemma (\autoref{bixby}), we may assume that the graph $H$ obtained from $G/e$ by
suppressing parallel edges is 3-connected. By minimality of $G$, there is no helicopter minor of
$\hat G$ with graph $H$. Hence $G/e$ has a nontrivial parallel class; and it must contain two edges
$e_1$ and $e_2$ that are both in $A$ or both in $B$. By symmetry we may assume that $e_1$ and $e_2$
are in $A$. Since $G$ is 3-connected, the edges $e$, $e_1$
and $e_2$ form a triangle in $G$. The common vertex of $e_1$ and $e_2$ is $v$. Thus both
endvertices of $e$ are in $V_A$.
\end{proof}

A consequence of \autoref{VA} is that every 3-minimal helicopter graph has at most most 12 edges.
However, we can say more:

\begin{cor}\label{cor100}

Let $\hat G=(G,v,w, ((a_i,b_i)|i=1,2,3))$ be 3-minimal. Then $G$ has at most
five
vertices.
\end{cor}

\begin{proof}
 Let $G_A$ be the induced subgraph with vertex set $V_A+v$. Let $G_B$ be the induced subgraph
with vertex set $V_B+w$.
Note that $G=G_A\cup G_B$.
If $G_A$ and $G_B$ have at least three vertices in common, then $G$ has at most five vertices as
$G_A$ and $G_B$ both have at most four vertices. Hence we may assume that $G_A$ and $G_B$ have at
most two vertices in common. As $G$ is 3-connected, the set of common vertices cannot be a
separator of $G$. Hence $G_A\se G_B$ or $G_B\se G_A$. Hence $G$ has at most four vertices in this
case.
\end{proof}

\begin{proof}[Proof of \autoref{minimal-helicopters}.]
Let $C$ be a 2-complex with a single loop $\ell$ be given. Let $x$ be its unique vertices incident with the loop $\ell$.
Assume that the link graph at $x$ is a subdivision of a 3-connected graph.
Assume that the link graph at $x$ is not helicopter planar.
Our aim is to construct a helicopter complex of bounded size that has $C$ as a space minor.

We denote the link graph at $x$ by $G$.
If $G$ is non-planar, we are done by \autoref{cone-final} and \autoref{cone-red}, so assume it is planar.
Since $G$ is not helicopter planar but planar, the loop must be incident with at least three faces.
Since $G$ is a subdivision of a 3-connected graph, it has a unique embedding in the plane (up to reversing).
Since $G$ is not helicopter planar, no associated helicopter graph can be helicopter planar.
Let $\hat G=(G,v,w, ((a_i,b_i)|i=1,2,3))$ be some associated helicopter graph.
$\hat G$ has a 3-connected helicopter graph as a helicopter minor.
So by definition of 3-minimality, $\hat G$ has a 3-minimal minor $\hat H=(H,v,w, ((a_i,b_i)|i=1,2,3))$; say, $H$ is obtained from $G$ by contracting the edge set $Z$ and deleting the edge set $D$, in formulas: $H=G/Z\sm D$.

Amongst all choices for sets $Z$ and $D$ we pick one so that $D$ is inclusion-wise maximal.
Let $D'$ consist of those edges of $D$ that do not carry any label.
We obtain $\hat I=(I,v,w, ((a_i,b_i)|i=1,2,3))$ from $\hat G$ by deleting those edges of $D'$.
By the choice of $D$ and since $H$ has bounded size by \autoref{cor100}, all but boundedly many edges of $Z$ are subdividing edges in $I$.
We obtain $Z'$ from $Z$ by removing all edges that carry labels and removing all edges that are not subdivision edges in $I$ and removing two subdivision edges of each edge of $I$, and if there are less than two subdivision edges for some edge of $I$ we remove all of them; note that $Z\sm Z'$ has bounded size.
We obtain $\hat J=(J,v,w, ((a_i,b_i)|i=1,2,3))$ from $\hat I$ by contracting $Z'$.
This construction ensures that $J=G/Z'\sm D'$.
By \autoref{cor100} $\hat H$ has bounded size and since only a bounded number of edges carry labels and $Z\sm Z'$ is bounded, also $\hat J$ has bounded size.

\begin{eg}
 This construction ensures that in the graph $\hat J$ the vertices $v$ and $w$ are distinct as they are distinct in $\hat G$.
\end{eg}

Since $\hat G$ is not helicopter planar, by \autoref{minimal_minor} $\hat H$ is not helicopter planar. Applying the other implication of \autoref{minimal_minor} to $\hat J$ and $\hat H$ we deduce that $\hat J$ is not helicopter planar.

It remains to show that $C$ has a space minor $\hat C$ that is a helicopter 2-complex so that its unique vertex $y$ incident with the loop has the property that $\hat J$ is an associated helicopter link graph of $\hat C$ at $y$, and so that $\hat C$ has bounded size.
We obtain $C_1$ from $C$ by deleting all faces and edges that are not incident with the vertex $x$ and then splitting all vertices except for $x$.
The link graph of $C_1$ at $x$ is $G$, and all edges not incident with $x$ are in a single face, and are nonloops.
Moreover, all its vertices are equal to $x$ or adjacent to it and all its edges or faces correspond to vertices or edges of $G$.

We obtain $C_2$ from $C_1$ by deleting all faces corresponding to edges in $D'$, and then split at all vertices aside from $x$, and delete isolated edges of $C$. The link graph of $C_2$ at $x$ is $G\sm D'$.
If a non-loop edge incident with $x$ is incident with exactly two faces, then the link graph at its endvertex other than $x$ is a path consisting of two edges.
Let $Z''$ be the set of edges of $C_2$ not incident with  $x$ that are in faces corresponding to an edge of $L(x)$ in $Z'$.
We obtain $C_3$ from $C_2$ by contracting $Z''$.
The link graph of $C_3$ at $x$ is $G\sm D'$.
\begin{sublem}\label{final-sublem}
  $C_3$ has only the loop $\ell$, and $C_3$ is a space minor of $C_2$.
\end{sublem}
\begin{cproof}
Recall that $C_2$ has only the loop $\ell$, so we need to investigate whether a new loop can arise from the contraction of $Z''$.
Let $P$ be a path of $I$ all of whose internal vertices have degree two in $I$.
Let $F_P$ be the set of faces whose corresponding edges in the link graph $L(x)$ are edges in $P$.
Let $C_2'$ be the 2-complex obtained from $C_2$ by deleting the faces of $F_P$ and after that deleting edges not in a face, and split isolated vertices.
Now $C_2$ can be obtained from $C_2'$ by adding the cone over $P$ at the path of length two corresponding to the first and last vertex of $P$.
The set $Z'$ avoids all edges of $P$ or at least two of them. Since contracting these edges commutes with the gluing of the cone over $P$ onto $C_2'$, it does not create any loop.
This is true for every choice of $P$ and contracting edges on different paths does not affect the other paths. In particular, we can contract the edges of $Z''$ in $C_2$ successively so that none of them creates a loop contracted, so $C_3$ is a space minor of $C_2$.
\end{cproof}

For every edge in $Z'$ its corresponding face of $C_3$ contains two edges since $C_3=C_2/Z''$. Since $\ell$ is the only loop of $C_3$ by \autoref{final-sublem}, it cannot be in any of these faces of size two.
We obtain $\hat C$ from $C_3$ by contracting each face corresponding to an edge of $Z'$ onto a single edge.
The link graph at $x$ is $J=G/Z'\sm D'$. By \autoref{final-sublem} $\hat C$ has only the loop $\ell$ and is a space minor of $C$.
Since $\hat J$ is not helicopter planar,  by \autoref{loop_to_helicopter} $\hat C$ is a helicopter complex.
Moreover, $\hat C$ has bounded size as all its vertices are equal to $x$ or adjacent to it and all its edges or faces correspond to one of the boundedly many vertices or edges of $J$.

\end{proof}

\subsection{Proof of \autoref{Kura_simply_con2}}

\begin{lem}\label{detour_via_5}
 Let $C$ be a stretched-out simplicial complex with a non-planar para-path $P$.
 Assume that the endvertices of $P$ are joined by an edge $e$ in $C$.

 Then $e$ is a non-planar para-path or $C$ has a space minor that is a helicopter complex.
\end{lem}

\begin{proof}
Since we are done otherwise, assume that $e$ is not a para-path; that is, in $C/e$ the link graph at $e$ is planar. Let $x$ be an arbitrary edge of $P$.
The 2-complex $D$ is obtained from $C$ by contracting $P+e-x$. Since $C$ is stretched out, the edge $x$ is the only loop of $D$.
Let $v$ be the vertex of $D$ incident with $x$. If the link graph at $v$ in $D$ is helicopter planar, by \autoref{sum_planar2} applied successively along the edges $P$, the para-path $P$ is planar.
Since this is not possible by assumption, the link graph at $v$ is not helicopter planar. So $D$ is a helicopter complex.
\end{proof}

\begin{lem}\label{x13}
 Let $C$ be a stretched-out simplicial complex with a non-planar para-path $P$ so that the two endvertices of $P$ are not adjacent in $C$.

 Then $C$ has a space minor $D$ that is a simplicial complex and has an edge $e$ so that the link graph of $D/e$ at the vertex $e$ is not planar.
\end{lem}

\begin{proof}
We obtain $C_1$ from $C$ by topologically deleting all edges of $C$ that do not have an endvertex on $P$ and then we split all vertices outside $P$.
Note that $C_1$ is a simplicial complex and it is stretched out.
Let $e$ be an arbitrary edge on $P$.
We obtain $C_2$ from $C_1$ by contracting the edges of $P-e$. By construction of $C_1$, the 2-complex $C_2$ has no loops and if edges are in parallel, then they form a face of size two.
We obtain $D$ from $C_2$ by contracting all faces of size 2. By construction of $C_2$, the 2-complex $D$ is a simplicial complex.
The link graphs at the two endvertices of $e$ in $D$ are isomorphic to the link graphs at the endvertices of $P$ in $C$.
So the edge $e$ is a non-planar para-path in $D$; that is, the link graph of $D/e$ at the vertex $e$ is not planar.
\end{proof}

We summarise this section in the following.

\begin{proof}[Proof of \autoref{Kura_simply_con2}.]
Let $C$ be a nullhomologous simplicial complex that is locally almost 3-connected and stretched out.
First note that by \autoref{heli_no_PRS} no 2-complex in $\Xcal$ admits a planar rotation system, and so by \autoref{kura_intro} and \autoref{final-list} no 2-complex in $\Ycal$ admits a planar rotation system.
2-complexes in $\Tcal$ do not admit planar rotation system, as explained in the Introduction to this chapter.

Conversely, assume that $C$ has no planar rotation system. Then by \autoref{sum123} we get outcome 2, 3, 4 or 5 from that lemma.
Outcome 2 gives a cone over $K_5$ or $K_{3,3}$ as a space minor by \autoref{cone-final} and \autoref{cone-red}; such cones as in $\Ycal$.

Outcome 3 gives a non-planar para-path $P$. If $P$ consists of a single edge, we can directly infer \autoref{combined-cone-summary} and \autoref{final-list} to find a combined cone from $\Ycal$ as a space minor of $C$. If otherwise the two endvertices of $P$ are not adjacent, we use \autoref{x13} to reduce to the case that $P$ has a single edge, and so are done here as well. In the remaining cases, we rely on \autoref{detour_via_5} to find outcome 5, which is discussed below.

Outcome 4 is a non-planar para-cycle. By \autoref{has_obstruction} it gives rise to a torus crossing obstruction, and thus an element of $\Tcal$.
Outcome 5 is a helicopter complex, which contains an element of $\Xcal$ by \autoref{minimal-helicopters}; and $\Xcal \se \Ycal$ by definition.
\end{proof}

\section{Stretching local 2-separators}\label{s2}

In this section we define stretching at local 2-separators and prove basic properties of this
operation.
This operation is necessary for \autoref{main_streching}.

A \emph{2-separator} in a 2-connected graph\footnote{In this paper we will only consider
2-separators of link graphs of simplicial complexes; such link graphs do not have parallel
edges or loops. For multigraphs, it seems suitable to also consider $(a,b)$ a 2-separator
if there are two parallel edges between them and $L-a-b$ is
not empty or $a$ and $b$ have three parallel edges in between. } $L$ is a
pair of
vertices $(a,b)$ such that
$L-a-b$ has at least two connected components.

Given a simplicial complex $C$ with a vertex $v$ such that its link graph $L(v)$ is 2-connected
and has a 2-separator $(a,b)$, the simplicial
complex $C_2$ obtained from $C$ by
\emph{stretching $\{a,b\}$
at $v$} is defined as follows, see \autoref{fig:s_pair}.

   \begin{figure} [htpb]
\begin{center}
   	  \includegraphics[height=4.5cm]{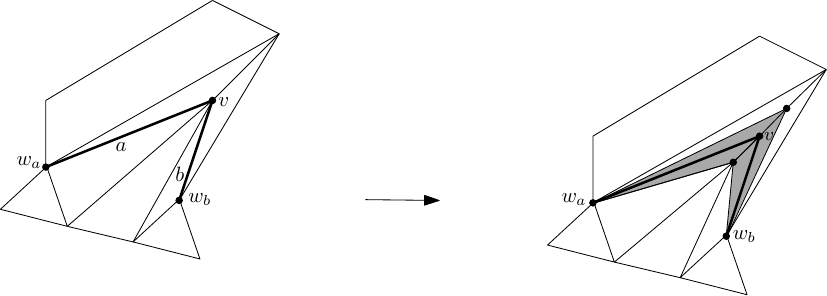}
   	  \caption{If we stretch the highlighted pair of edges in the simplicial complex on the
left,
we obtain the one on the right. The newly
added faces are
depicted in grey.}\label{fig:s_pair}
\end{center}\vspace{-0.7cm}
   \end{figure}

   \begin{figure} [htpb]
\begin{center}
   	  \includegraphics[height=1cm]{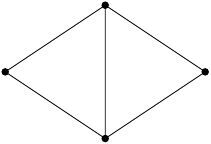}
   	  \caption{The simplicial complex $\Delta_2$.}\label{fig:Delta2}
\end{center}\vspace{-0.7cm}
   \end{figure}

We denote by $\Delta_2$ the simplicial complex obtained from two disjoint faces of size three by
gluing them together at an edge, see \autoref{fig:Delta2}.
Let $\Delta^+_n$ be the simplicial complex obtained by gluing $n$ copies of $\Delta_2$
together at a path of length two whose endvertices have degree two in $\Delta_2$ (this is uniquely
defined up to isomorphism), see \autoref{fig:Delta2N}.
   \begin{figure} [htpb]
\begin{center}
   	  \includegraphics[height=2cm]{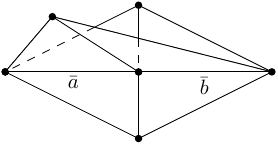}
   	  \caption{The simplicial complex $\Delta^+_3$ with the gluing edges
labelled $\bar a$ and $\bar b$.}\label{fig:Delta2N}
\end{center}\vspace{-0.7cm}
   \end{figure}

Informally, we obtain $C_2$ from $C$ by replacing the edges $a$ and $b$ by $\Delta^+_n$, where $n$
is the
number of components of $L(v)-a-b$.
More precisely, the simplicial complex $C_2$ is defined as follows.
Let $n$ be the number of components of $L(v)-a-b$.
We denote the gluing edges of $\Delta^+_n$ by
$\bar a$ and $\bar b$. We label the vertices of $\Delta^+_n$ incident with neither $\bar a$ nor
$\bar b$ by the components of $L(v)-a-b$.

In our notation we suppress a bijection between vertices of $C$ and $\Delta^+_n$ as follows.
We label the common vertex of the edges of $\bar a$ and $\bar b$ by $v$.
We denote the endvertex of the edge $a$ in $C$ different from $v$ by $w_a$; and we label the
endvertex of the edge $\bar a$ different from $v$ by $w_a$. Similarly, we denote the endvertex of
the edge $b$ in $C$ different from $v$ by $w_b$; and we label the
endvertex of the edge $\bar b$ different from $v$ by $w_b$.
\begin{itemize}
 \item The vertex set of $C_2$ is union of the vertex set of $C$ together with
the vertex set
of $\Delta^+_n$, in formulas: $V(C_2)=V(C) \cup V(\Delta^+_n)$. We stress that the sets $V(C)$ and
$V(\Delta^+_n)$ share the vertices $v$, $w_a$ and $w_b$ and hence these vertices appear in $V(C_2)$
only once as $V(C_2)$ is just a set and not a multiset;
\item the edge set of $C_2$ is (in bijection with) the edge set of $C$ with the edges $a$ and
$b$ replaced by
the set of edges of $\Delta^+_n$, in formulas: $E(C_2)=\left (E(C)-a-b\right ) \cup E(\Delta^+_n)$.
The incidences between vertices and edges are as in $C$ or $\Delta^+_n$, except for those edges of
$C$
that have the vertex $v$ as an endvertex.
This defines all incidences of edges except those of $C$ that have the endvertex $v$.
Given an edge $x$ of $C$ incident with $v$, and denote its other endvertex by $x'$.
Then its corresponding edge of $C_2$ has the endvertices $x'$ and the vertex of $\Delta^+_n$ that
is
the component of $L(v)-a-b$ containing $x$. This completes the definition of the edges of $C_2$.
We stress that the vertex $w_a$ of $C_2$ is incident with those edges of $C-a-b$ with endvertex
$w_a$ and
those edges of $\Delta^+_n$ with endvertex $w_a$;
\item the faces of $C_2$ are the faces of $C$ together with the faces of $\Delta^+_n$; in formulas:
$F(C_2)=F(C)\cup F(\Delta^+_n)$. We stress that the sets $F(C)$ and $F(\Delta^+_n)$ are disjoint.
The incidences between edges and faces are as in $C$ or $\Delta^+_n$, where defined. This defines
all
incidences of faces except for those faces $f$ of $C$ incident with the edges $a$ or $b$, which are
defined as follows.
There are three cases:

-- if $f$ is a face of $C$ incident with both the edges $a$ and $b$, then in $C_2$ these incidences
are replaced by incidences with the edges $\bar a$ and $\bar b$;

-- if $f$ is a face of $C$ incident with the edge $a$ but not $b$, then in $C_2$ the incidence of
$f$ with $a$ is replaced with an incidence with the edge $w_ax$ of $\Delta^+_n$; where
$x$
is the component of $L(v)-a-b$ such that in $L(v)$ the edge $f$ joins $a$ with a
vertex of $x$;

-- similarly, if $f$ is a face of $C$ incident with the edge $b$ but not $a$, then in $C_2$ the
incidence of
$f$ with $b$ is replaced with an incidence with the edge $w_bx$ of $\Delta^+_n$; where
$x$
is the component of $L(v)-a-b$ such that in $L(v)$ the edge $f$ joins $b$ with a
vertex of $x$.
\end{itemize}

This completes the definition of stretching a 2-separator at a
vertex.

We refer to the vertices of $C_2$ that are not in $V(C)-v$ as the \emph{new vertices}, other
vertices of $C_2$ are called \emph{old}.

The link graph at $w_a$ of $C$ is obtained
from the link graph at $w_a$ in $C_2$ by contracting all edges incident with the vertex $\bar a$.
Note that $w_a$ cannot be incident with $b$ as $C$ is a simplicial complex.

\begin{eg}
 In \autoref{fig:link2} we explain how the link graphs of
\autoref{fig:s_pair} change.
   \begin{figure} [htpb]
\begin{center}
   	  \includegraphics[height=3cm]{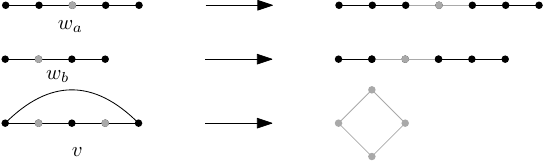}
   	  \caption{On the left we see the link graphs at the vertices $w_a$,
$w_b$ and $v$ of the simplicial complex in \autoref{fig:s_pair}. On the right we see the link
graphs after the stretching at $(a,b)$. The vertices $a$ and $b$ and the new vertices and edges in
the link graphs are depicted in grey. }\label{fig:link2}
\end{center}\vspace{-0.7cm}
   \end{figure}

\end{eg}

The \emph{(abbreviated) degree-sequence} of a graph is the sequences of degrees of its vertices,
ordered by size, where we leave out the degrees which are at most two. We compare degree-sequences
in the lexicographical order.

\begin{lem}\label{2-stretch-change}
 Let $C_2$ be a simplicial complex obtained from $C$ by stretching the 2-separator $(a,b)$ at $v$.
 Then at each vertex of $C$ aside from $v$, the degree-sequence at its link in $C_2$ is at most
the degree-sequence at its link in $C$.
At all new vertices of $C_2$ the degree-sequence at the link is strictly smaller than the
degree-sequence of the link graph at $v$ in $C$ -- unless the link graph at $v$ in $C$ is a
parallel
graph or $L(v)-a-b$ has two components and one is a path.
\end{lem}

\begin{proof}
Coadding a star at a vertex cannot increase the abbreviated degree-sequence, hence the lemma is
true at old vertices of $C_2$.
So it remains to prove the lemma for the new vertices of $C_2$ as it is obvious at the
others. As the link graph $L(v)$ at $v$ in $C$ is not a parallel graph, the degree-sequence at the
link at $v$ in $C_2$ is strictly smaller than that in $C$.
Now let $X$ be a component of $L(v)-a-b$. If $L(v)-a-b$ has at least three components, then
 the degree-sequence at the
link at $X$ in $C_2$ is strictly smaller the degree-sequence of the link graph at $v$ in
$C$. This is also true if $L(v)-a-b$ has only two components and the other component has a vertex
of degree greater than two; that is, is not a path. This completes the proof of the lemma.
\end{proof}

The \emph{degree-parameter} of a 2-complex is the sequence of degree-sequences of all its link
graphs, ordered by size.  We compare degree-parameters
in the lexicographical order.

\begin{lem}\label{stretch_loc_con}
 Let $C$ be a simplicial complex such that all link graphs are 2-connected or free-graphs. Then we
can apply stretchings at local
2-separators of 2-connected link graphs such that the resulting simplicial complex is locally
almost 3-connected.
\end{lem}

Before we prove this, we need a definition. A 2-separator $(x,y)$ in a graph $G$ is \emph{proper}
unless $G-x-y$ has precisely
two components
and one of them is a path and $xy$ is not an edge.

\begin{proof}
 If $C$ has a 2-connected link graph that is not a parallel graph or a
subdivision of a 3-connected graph, it contains a proper 2-separator and
we stretch at that 2-separator.
Link graphs at other vertices remain 2-connected or free graphs, respectively.
By \autoref{2-stretch-change} the degree-parameter goes down and hence this process has to stop
after finitely many steps -- with the desired simplicial complex.
\end{proof}

\vspace{.3cm}

Until the rest of this section we fix a simplicial complex $C$ with a vertex $v$ such that the link
$L(v)$ is 2-connected and let $(a,b)$ be a 2-separator of $L(v)$. We
denote the simplicial complex obtained
from $C$ by stretching $(a,b)$ at $v$ by $C_2$.

\begin{rem}\label{inverse2}
 $C$ can be obtained from $C_2$ as follows. First we contract the edges incident with the
vertex $v$ except for $\bar a$ and $\bar b$. We relabel $\bar a$ by $a$ and $\bar b$ by $b$.
We obtain some faces of size two, we refer to these faces as \emph{tiny} faces. Then we contract
all these tiny faces. This gives $C$.
\end{rem}

We say that an operation, such as contracting an edge, is an \emph{equivalence} for a property,
such as the existence of a planar rotation systems, if a simplicial complex has that property if
and only if the simplicial complex after applying this operation has this property.

In \autoref{contr_pres_planar} it is
shown that contracting a non-loop edge where the link graph at both endvertices are 2-connected is
an equivalence for the
existence of planar rotation systems. Contracting a face of size two is not always an equivalence
for the
existence of planar rotation systems but here the contracted faces have the following additional
property.

A face $f$ incident with only two edges $e_1$ and $e_2$ is \emph{redundant} if there is a vertex
$v$ incident with $f$ such that in $C/f$ in any planar rotation system of the link graph $L(v)$
at the rotator at $f$, the edges incident with $e_1$ in the link at $v$ for $C$ form an
interval. (This implies that also the edges incident with $e_2$ in the link at $v$ for $C$ form an
interval.)

The following is obvious.
\begin{obs}\label{is_eq}
Let $C'$ be obtained from $C$ by contracting a redundant face.
If $C'$ has a planar rotation system, then $C$ has a planar rotation system.
\qed
\end{obs}

\begin{obs}\label{tiny}
Tiny faces (as defined in \autoref{inverse2}) are redundant.
\end{obs}
\begin{proof}
 Let $X$ be a component of $L(v)-a-b$. Since $L(v)$ is 2-connected, the edges between $a$ and $X$
form an interval in any rotator at $a$ for any embedding of $L(v)$ in the plane. The same is true
for `$b$' in place of `$a$'.
\end{proof}

\begin{lem}\label{PRS2}
The simplicial complex $C$ has a planar rotation system if and only if the simplicial complex
$C_2$
has a planar rotation system.
\end{lem}

\begin{proof}
 It is shown in \autoref{contr_pres_planar} that contracting a non-loop edge
where the link graph at both endvertices are
2-connected is an equivalence for the
existence of planar rotation systems.

By \autoref{rot_closed_down} contracting a
face of
size two preserves the
existence of planar rotation
systems. So contracting tiny faces is an equivalence for the existence of planar rotation systems
by \autoref{is_eq} and \autoref{tiny}.

Hence all the operation that transform the simplicial complex $C_2$ to $C$ as described in
\autoref{inverse2} are equivalences. Thus stretching at local 2-separators is an equivalence for
planar rotation systems.
\end{proof}

The following is geometrically clear, see \autoref{fig:s_pair}, and we will not use it in
our proofs.

\begin{lem}\label{stretch_22_embed} If $C$ embeds in 3-space, then also $C_2$
embeds in 3-space.
\qed
\end{lem}

\begin{rem}
 Also the converse of \autoref{stretch_22_embed} is true.
\end{rem}

\section{Stretching a local branch}\label{s1}

In this section we define stretching local branches and prove basic properties of this operation.
This operation is necessary for \autoref{main_streching}.

Given a connected graph $G$ with a cut-vertex $v$, a \emph{branch} at $v$ is a connected component
$X$ of $G-v$ together with the vertex $v$ (and all edges between $X$ and $v$). A \emph{branch of
$G$} is a branch at some cut-vertex of $G$. For any branch $B$, there is a unique vertex $v$ such
that $B$ is a branch at $v$; we refer to that vertex $v$ as \emph{the cut-vertex} of the branch $B$.

Given a 2-complex $C$ with a vertex $v$ such that the link graph $L(v)$ at $v$ is connected and a
branch $B$ of $L(v)$, the complex $C[B]$ obtained from $C$ by \emph{pre-stretching} $B$ is
defined as follows, see \autoref{fig:stretch_branch}.
We denote the cut-vertex of the branch $B$ by $e$; and remark that $e$ is an edge of the simplicial
complex $C$.

   \begin{figure} [htpb]
\begin{center}
   	  \includegraphics[height=5cm]{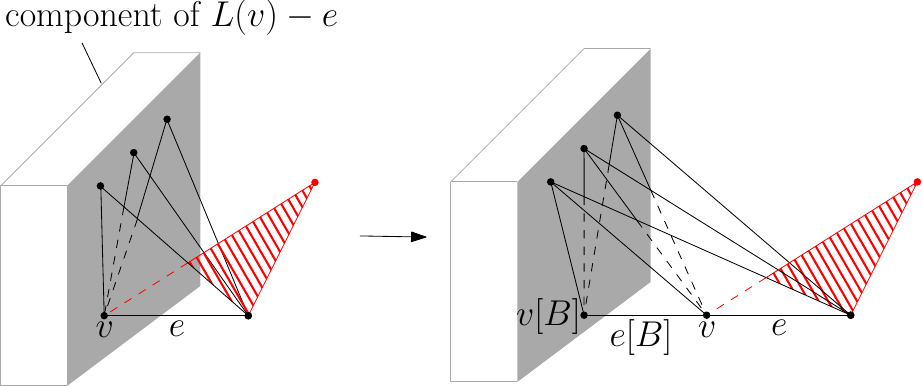}
   	  \caption{The 2-complex on the right is obtained from the 2-complex on the left by
stretching the branch $B$, which consists of the grey box together with the three faces attaching
at the grey box. Stretching is defined like pre-stretching but there
we additionally subdivide faces to make them all have size three.}\label{fig:stretch_branch}
\end{center}\vspace{-0.7cm}
   \end{figure}

\begin{itemize}
 \item The vertex set of $C[B]$ is that of $C$ together with one new vertex, which we denote
by $v[B]$, in formulas: $V(C_1)=V(C) \cup \{v[B]\}$;
\item the edge set of $C[B]$ is (in bijection with) the edge set of $C$ together with one
additional edge, which we denote by $e[B]$, in formulas: $E(C[B])=E(C)\cup \{e[B]\}$;
The incidences between edges and vertices are as in $C$ except for those edges $z\neq e$ of $C$
that are vertices of the branch $B$. Such edges are incident with the new vertex $v[B]$ in place of
$v$, the other endvertex is not changed. The edge $e[B]$ has the endvertices $v$ and $v[B]$;

\item the faces of $C[B]$ are (in bijection with) the faces of $C$; in formulas:
$F(C[B])=F(C)$.
The incidences between faces and edges are as in $C$ except for those faces of $C$ that are
incident with the edge $e$ and are in the link graph $L(v)$ edges of the branch $B$. These faces
now have size four. They are now additionally incident with the edge $e[B]$.
\end{itemize}
This completes the definition of pre-stretching the branch $B$ at $v$.
\emph{Stretching} the branch $B$ is defined the same way except that we additionally subdivide each
face $f$ of size four once. Namely we add a subdivision-edge between the vertex $v$ and the
unique vertex of the face that is not in the edge $e$ and different from $v[B]$. Hence for any
simplicial complex $C$ any stretching at a branch is again a simplicial complex.

See \autoref{fig:stretch_branch_link} for an
example illustrating how the link changes at the vertex $v$ and how the link looks like at the
vertex $v[B]$.
   \begin{figure} [htpb]
\begin{center}
   	  \includegraphics[height=4cm]{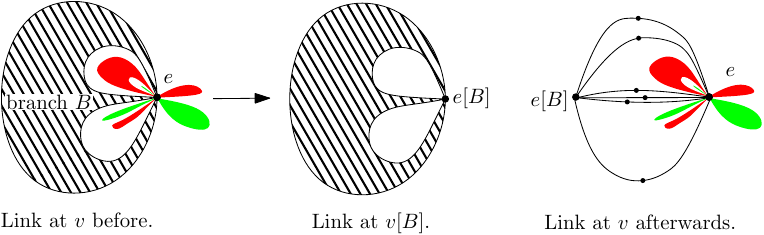}
   	  \caption{On the left we see the link graph of a vertex
$v$ with a cut-vertex $e$ and a branch $B$. On the right we see the
link graphs of the two vertices obtained from the link graph of $v$ by stretching
$B$.}\label{fig:stretch_branch_link}
\end{center}\vspace{-0.7cm}
   \end{figure}

\vspace{.3cm}

Until the rest of this section we fix a simplicial complex $C$ with a vertex $v$ such that the link
graph $L(v)$ is connected and let $B$ be a branch of that link. We denote the simplicial complex
obtained
from $C$ by stretching $B$ by $C[B]$.

\begin{rem}\label{inverse}
The simplicial complex $C$ can be obtained from $C[B]$ as follows. First we contract the edge
$e[B]$. This makes the
faces incident with $e[B]$ in $C[B]$ have size two. Then we contract these faces. This gives $C$.
The
contracted faces are incident with an edge that is incident with only one other face.
\end{rem}

   \begin{lem}\label{PRS1}
The simplicial complex $C[B]$ has a planar rotation system if and only if $C$ has a planar rotation
system.
\end{lem}

\begin{proof}
Let $\Sigma$ be a planar rotation system of the simplicial complex $C$.
We define\footnote{Faces incident with the edge $e$ in $C$ correspond in $C[B]$ either to a single
face of size three
or two faces of size three obtained from a face of size four by subdivision. This induces a
bijective map from the faces incident $e$ in $C$ to the faces incident with $e$ in $C[B]$, and an
injective
partial map from the faces incident $e$ in $C$ to the faces incident with $e[B]$ in $C[B]$. In
order to simplify the presentation of the definition we suppress these two maps.} a
rotation system $\Sigma'$ of the simplicial complex $C[B]$ by taking the same rotator as $\Sigma$
at
every edge except for $e[B]$ and other new edges, which are incident with two faces.
 At the edges incident with two faces we take the unique cyclic ordering of size two. The rotator
at the edge $e[B]$ is constructed from the rotator at the edge $e$ for $\Sigma$ by restricting it
to the faces incident with the edge $e[B]$.

This rotation system is obviously planar at all vertices of $C[B]$ except for the vertex $v$ and
$v[B]$.
We denote by $\Pi$ the rotation system  induced by
$\Sigma$ of the link graph of $C$ at the vertex $v$. By the construction given
directly after \autoref{minimal_minor}, $\Pi$ induces a planar rotation
system $\Pi_1$ at the branch $B$, and $\Pi$ induces a planar rotation system $\Pi_2$
at the minor of the link graph at $v$ in $C$ obtained by contracting $B-e$ to a single
vertex.
It is immediate that the rotation system induced by $\Sigma'$ at $v[B]$ is $\Pi_1$, and the
rotation system induced by $\Sigma'$ at $v$ is $\Pi_2$.
Hence the rotation system $\Sigma'$ is planar for the simplicial complex $C[B]$.

By \autoref{contr_pres_planar} contracting an edge preserves the existence of
planar rotation systems, and by \autoref{rot_closed_down} contracting a face of
size two preserves the existence of planar rotation
systems. Hence by \autoref{inverse} if $C[B]$ has a planar rotation system, then $C$ has a planar
rotation system.
\end{proof}

The following is geometrically clear, see \autoref{fig:stretch_branch}, and we will not use it in
our proofs.

\begin{lem}\label{stretch_branch_embed} If $C$ embeds in 3-space, then also $C[B]$
embeds in 3-space.
\qed
\end{lem}

\begin{rem}
 Also the converse of \autoref{stretch_branch_embed} is true.
\end{rem}

\section{Increasing local connectivity}\label{s3}

In the first three subsections of this section we define stretchings and prove basic properties;
these are necessary for \autoref{main_streching}.
The forth subsection is a preparation for the last subsection, in which we prove
\autoref{main_streching}, and \autoref{Kura_gen}.
\subsection{The operation of stretching edges}
Let $C$ be a 2-complex and let $e$ be an edge of $C$ incident with two faces $f_1$ and $f_2$.
Assume that there is an endvertex $v$ of the edge $e$ such that in any planar rotation system of
the link graph $L(v)$ at $v$ the edges $f_1$ and $f_2$ are adjacent in the rotator at $e$.
The complex $C'$ obtained from $C$ by \emph{pre-stretching the edge $e$ in
the direction of $f_1$ and $f_2$} is obtained from $C$ as follows, see \autoref{fig:stretch_edge}.
We replace the edge $e$ by two edges new edges $e_1$ and $e_2$, both with the same endvertices as
$e$. We add a face of size two only incident with $e_1$ and $e_2$. The faces $f_1$ and $f_2$ are
incident with $e_1$ instead of $e$, all other faces incident with $e$ in $C$ are incident with
$e_2$
instead.
This completes the definition of pre-stretching an edge. \emph{Stretching} an edge is defined the
same way except that additionally we subdivide the new face of size two
to obtain a simplicial complex, see \autoref{fig:subdiv}.

   \begin{figure} [htpb]
\begin{center}
   	  \includegraphics[height=4cm]{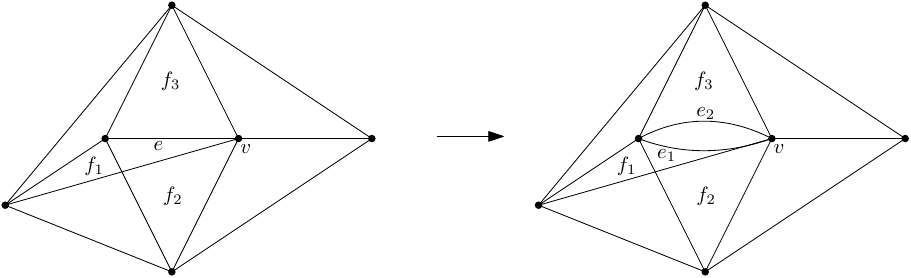}
   	  \caption{The graph on the left defines a simplicial complex by
adding faces on all cycles of size three. We obtain the 2-complex on the
right by pre-stretching the edge $e$ in the direction of the faces $f_1$ and
$f_2$. Its faces are all triangles of the graph on the left except that
the edge $e_1$ is only incident with the two faces $f_1$ and $f_2$ and the new face  $\{e_1,e_2\}$
and the edge $e_2$ is only incident with the face $f_3$ and the new face
$\{e_1,e_2\}$.}\label{fig:stretch_edge}
\end{center}
   \end{figure}

      \begin{figure} [htpb]
\begin{center}
   	  \includegraphics[height=1.3cm]{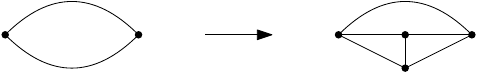}
   	  \caption{Subdivision of a face of size two to a simplicial complex.}\label{fig:subdiv}
\end{center}\vspace{-0.7cm}
   \end{figure}

\begin{eg}
The assumption for stretching an edge $e$ is particularly easy to verify if the link graph $L(v)$
is 3-connected. Indeed, then by a theorem of Whitney, we just need to check whether for a
particular embedding of the link graph $L(v)$ the edges $f_1$ and $f_2$ are adjacent.
\end{eg}

\begin{rem}\label{rem27}
The inverse operation of pre-stretching an edge $e$ to a face $\{e_1,e_2\}$ is contracting the face
$\{e_1,e_2\}$ to the edge $e$ as defined in \autoref{chapterI}.
\end{rem}

\begin{lem}\label{stretch_edge_rotn}
Let $C'$ be obtained from $C$ by pre-stretching an edge $e$. Then $C'$ has a planar rotation system
if
and only if $C$ has a planar rotation system.
\end{lem}

\begin{proof}
Let $\Sigma$ be a planar rotation system of the 2-complex $C$.
We denote the two new edges of the 2-complex $C'$ by $e_1$ and $e_2$.
We obtain a rotation system $\Sigma'$ of the 2-complex
$C'$ from $\Sigma$ by taking the same rotators at all edges of the 2-complex $C'$ except for $e_1$
and $e_2$.
By assumption, the faces $f_1$ and
$f_2$ along which we pre-stretch the edge $e$ are adjacent in the rotator at the edge $e$.
We define the new rotator at the edge $e_1$ to be the rotator of the edge $e$ restricted to the
adjacent faces $f_1$ and $f_2$ and we add the new face $\{e_1,e_2\}$ in place of the interval
formed by the deleted faces. Similarly, we define a
rotator at the edge $e_2$: we delete from the rotator at $e$ the faces $f_1$ and $f_2$ and add the
face $\{e_1,e_2\}$ in the interval formed by the two deleted faces.
It remains to check that the rotation system $\Sigma'$ is planar. This is immediate at all vertices
except for the two endvertices of the edge $e$.
For the two endvertices, note that pre-stretching the edge $e$ has the effect on the link graph as
coadding\footnote{Recall that \emph{addition} is the reverse operation of deletion of an edge, and \emph{coaddition} is the reverse operation of contraction of an edge.} an edge at the vertex $e$. As the edges $f_1$ and $f_2$ of the link graphs are
adjacent, the coaddition can be done within the embeddings  of the link graphs given  by $\Sigma$.

By \autoref{rot_closed_down}, contracting a
face of size two preserves the existence of planar rotation systems. Hence by \autoref{rem27} if
$C'$ has a planar rotation system, then $C$ has a planar rotation system.
\end{proof}

The following is geometrically clear, see \autoref{fig:stretch_edge}, and we will not use it in
our proofs.

\begin{lem}\label{stretch_edge_embed}
Let $C'$ be obtained from $C$ by stretching an edge $e$. If $C$ embeds in 3-space, then also $C'$
embeds in 3-space.
\qed
\end{lem}

\begin{rem}
 Also the converse of \autoref{stretch_edge_embed} is true.
\end{rem}

\subsection{The operation of contracting edges}
An edge $e$ in a 2-complex $C$ is \emph{reversible} if the 2-complex $C$
has a planar rotation system if and only if the 2-complex $C/e$ has a
planar rotation system.

A \emph{para-star} is a graph obtained from a family of disjoint parallel graphs by gluing them
together at a single vertex.

\begin{lem}\label{is_rev}
Let $e$ be a non-loop edge with endvertices $v$ and $w$ of a 2-complex $C$
such that the link graphs $L(v)$ and $L(w)$ are para-stars and the
vertex $e$ is a maximum degree vertex in both of them.
Then the edge $e$ is reversible.
\end{lem}

\begin{proof}
By \autoref{contr_pres_planar}, it suffices to show how any
planar rotation
system $\Sigma$ on the 2-complex $C/e$ induces a planar rotation
system $\Sigma'$ on the 2-complex $C$.
Letting $\Sigma'$ to be equal to $\Sigma$ at all edges of $C$ not
incident with $v$ or $w$, it suffices to show the following.

\begin{sublem}\label{sublem28}
Let $L(v)$ and $L(w)$ be para-stars and let the vertex
$e$ have maximal degree in both of them. Let $L(e)$ be the vertex sum
of $L(v)$ and $L(w)$ along $e$.
For any planar rotation system $\Pi$ of the graph $L(e)$, there are
planar rotation systems of the graphs $L(v)$ and $L(w)$ that are
reverse of one another at the vertex $e$, and otherwise agree
with the rotation system $\Pi$.
\end{sublem}

\begin{cproof}
Throughout we assume in the graphs $L(v)$ and $L(w)$, the vertex $e$
is adjacent to any other vertex. This easily implies the general case
by suppressing suitable degree two vertices as rotators
at such vertices are unique.

We prove this by induction on the number of branches of the graph $L(v)$.
The base case is that the graph $L(v)$ is a parallel graph. Then the
graph $L(e)$ is isomorphic to the graph $L(w)$. So a planar rotation
system on the graph $L(e)$ induces a planar rotation
system on the graph $L(w)$. There is a unique planar rotation
system on the graph $L(v)$ whose rotator at $e$ is reverse to the
rotator at $e$ in that planar rotation system of $L(w)$.

So we may assume that the graph $L(v)$ has at least two branches.
We split into two cases.

{\bf Case 1:} the graph $L(e)$ is disconnected.
We consider $L(e)$ as a bipartite graph with the vertex set of
$L(v)-e$ on the left and the vertex set of $L(w)-e$ on the right.
As every vertex of $L(e)$ is incident with an edge, there are two
vertices of $L(v)-e$ in different connected components of the
bipartite graph $L(e)$. Denote these two vertices by $y$ and $z$.
We obtain $L(v)'$ from $L(v)$ by identifying the vertices $y$ and $z$
into a single vertex. Denote that new vertex by $u$. We denote the
vertex sum of $L(v)'$ and $L(w)$ along $e$ by $L(e)'$. The
graph $L(e)'$ is equal to the graph obtained from $L(e)$ by
identifying the vertices $y$ and $z$. Thus any planar rotation system
of the graph $L(e)$ induces a planar rotation system of the graph
$L(e)'$ by sticking the rotation systems at the vertices $y$ and $z$
together so that the rotator at the new vertex $u$ contains the edges
incident with the vertex $y$ or $z$, respectively, as a
subinterval.
By induction such a rotation system induces planar rotation systems at
the graphs $L(v)'$ and $L(w)$. This planar rotation system at the
graph $L(v)'$ induces a rotation system on the graph $L(v)$
by splitting the rotator at $u$ into the two subintervals for the
vertices $y$ and $z$. This induced rotation system is planar for
$L(v)$ as the rotators for $y$ and $z$ are subintervals of the
rotator for $u$.

{\bf Case 2:} not Case 1, so the graph $L(e)$ is connected. As above,
we consider $L(e)$ as a bipartite graph, and let a planar rotation
system of the graph $L(e)$ be given. Since the left side
has at least two vertices, there is a vertex on the right of the
connected bipartite graph $L(e)$ that has two neighbours on the left.
Pick such a vertex $x$. We pick neighbours $y$ and $z$ of $x$
in $L(v)-e$ such that there are edges $e_y$ between $y$ and $x$ and
$e_z$ between $z$ and $x$ such that these two edges are incident in
the rotator at the vertex $x$.
Let $L(v)'$ be the graph obtained from $L(v)$ by identifying the
vertices $y$ and $z$ to a single vertex. Call that new vertex $u$. We
denote the vertex sum of $L(v)'$ and $L(w)$ along $e$ by $L(e)'$. The graph $L(e)'$ is equal to the
graph obtained from $L(e)$ by
identifying the vertices $y$ and $z$. The chosen planar rotation
system of the graph $L(e)$ induces a rotation system for the
graph $L(e)'$ by sticking the rotation systems at the vertices $y$ and
$z$ together so that the rotator at the new vertex $u$ contains the
edges incident with the vertex $y$ or $z$, respectively,
as a subinterval.  By the choice of $y$ and $z$ this rotation system
is planar. By induction this planar rotation system on $L(e)'$ induces
planar rotation systems on the graphs $L(v)'$ and $L(w)$. This planar rotation system at the graph
$L(v)'$ induces a
rotation system on the graph $L(v)$ by splitting the rotator at $u$
into the two subintervals for the vertices $y$ and $z$. This induced rotation system is planar for
$L(v)$ as the rotators for $y$
and $z$ are subintervals of the rotator for $u$.
\end{cproof}
To summarise the proof of \autoref{is_rev}, we define the planar rotation system $\Sigma'$ for
the 2-complex $C$ as indicated above, and we choose the rotators at the edges incident with the
vertices $v$ or $w$ as induced in the sense of \autoref{sublem28} by the rotation system of the link
graph $L(e)$ at the vertex $e$ of the 2-complex $C/e$.
\end{proof}

\subsection{The definition of stretching}

We say that a simplicial complex $\tilde C$ is obtained from a simplicial complex $C$ by
\emph{stretching}, if it is obtained from $C$ by applying successively operations of the following
types:
\begin{enumerate}
 \item stretching local branches at connected link graphs;
 \item 2-stretching at local 2-separators of 2-connected link graphs;
 \item stretching edges;
\item contracting reversible edges that are not loops;
\item splitting vertices.
\end{enumerate}
We also call $\tilde C$ a \emph{stretching} of $C$.

\begin{lem}\label{planar_rot_preserve}
 Assume $C'$ is a stretching of $C$. Then $C$ has a planar rotation system if and only if $C'$ has
a planar rotation system.
\end{lem}

\begin{proof}
In the language introduced above we are to show that all five stretching operations are
equivalences for the property `existence
of planar rotation systems'.
For the first operation it is proved in \autoref{PRS1}, for the second it is proved in
\autoref{PRS2}, and for the third it is proved in \autoref{stretch_edge_rotn} for pre-stretchings
of edges, and so the result for stretchings follows. For the forth
operation it is true by the definition of reversible.
Splitting vertices is clearly an equivalence for the existence of planar rotation systems.
\end{proof}

\subsection{Increasing the local connectivity a bit}

A 2-complex is \emph{locally almost 2-connected} if all its link graphs are 2-connected or free
graphs. The following is a key step towards
\autoref{main_streching}.

\begin{thm}\label{reduce_to loc_2-con}
Any simplicial complex $C$ has a stretching $C'$ that is a simplicial complex so that $C'$ is
locally almost 2-connected or $C'$ has a non-planar link graph.
\end{thm}

Before we prove \autoref{reduce_to loc_2-con}, we need some preparation.
A \emph{star of parallel graphs} is a graph that is not 2-connected and is obtained from a set of
disjoint parallel graphs by
gluing them together at a single vertex.
\begin{eg}
The only parallel graphs that are stars of parallel graphs are paths. Stars of parallel graphs are
2-connected para-stars.
\end{eg}

\begin{lem}\label{stretching_loc_1-cuts}
 Let $C$ be a simplicial complex that is locally connected. Then there is a simplicial complex
$\tilde{C}$ that is obtained
from $C$ by stretching local branches such that every link graph of $\tilde C$ is 2-connected
or a
star of parallel graphs.
\end{lem}

\begin{proof}
We will prove this by induction. The base case is that every link graph is 2-connected or a star of
parallel graphs. Next we consider the case that each link graph has at most one cut-vertex. Let $v$
be a vertex of the simplicial complex $C$ such that its link graph has a cut-vertex $e$. Then all
branches of $e$ are 2-connected graphs. We stretch all branches of $e$ that are not parallel
graphs, one after the other. Then the link at $v$ becomes a star of parallel graphs and all other
new link graphs are 2-connected. The old link graphs, those at vertices of $C$ aside from $v$,
do not change except possibly for subdividing edges\footnote{The vertices $v'$ of $C$ where those
subdivisions occur are those such that there is an edge $e'$ between $v$ and $v'$ such that
$e'$ is a vertex of one of the branches we stretch.}. We apply this recursively to all
link graphs
with cut-vertices, and so reduce this
case to the base case.

Next suppose that there is a vertex $v$ such that its link graph has at least two cut-vertices.
Let $e_1$ be an arbitrary cut-vertex of that link graph. Let $B$ be a branch of $e_1$
containing another cut-vertex $e_2$. Then we stretch $B$. All link graphs at vertices of $C$
aside from $v$ are not changed (except for possibly subdividing edges). The vertex $v$ is replaced
by two new vertices. Each cut-vertex of
the link graph of $v$ is in precisely one of the two new link graphs, and $e_1$ and $e_2$ are in
different link graphs. Hence both new link graphs have strictly less cut-vertices than the link
graph of $v$. Hence we can apply induction (on the sequence of numbers of cut-vertices of link
graphs, ordered by size and compared in lexicographical order).

\
\end{proof}

\begin{proof}[Proof of \autoref{reduce_to loc_2-con}.]
The \emph{cutvertex-degree} of a simplicial complex $C$ is
the maximal degree of a cutvertex of a link graph of the simplicial complex $C$.
We prove \autoref{reduce_to loc_2-con} by induction on the cutvertex-degree.
So let $C$ be a simplicial complex with cutvertex-degree $a$.

We obtain $C_1$ from $C$ by splitting all vertices whose link graphs are disconnected. The
simplicial complex $C_1$ is locally connected.
By \autoref{stretching_loc_1-cuts} there is a stretching $C_2$ of the simplicial complex $C_1$ such
that all its link graphs are 2-connected or stars of parallel graphs.

If the cutvertex-degree $a$ is at most three, then all link graphs of $C_2$ are
2-connected or stars of parallel graphs where the unique cut-vertex has degree at most three.
Graphs of the second type are always free, see
\autoref{fig:starstarround}. This completes the proof if  the cutvertex-degree is at most
three, so from now on let the cutvertex-degree $a$ be at least four.

We say that a simplicial complex $C$ is \emph{$a$-nice} if all its link graphs are 2-connected, or
stars of parallel graphs whose cutvertex has degree precisely $a$ or else have maximum degree
strictly less than $a$. For example the simplicial complex $C_2$ is $a$-nice.
We say that a simplicial complex $C$ is \emph{$a$-structured} if all its link graphs are
parallel graphs, or stars of
parallel graphs whose cutvertex has degree precisely $a$ or else have maximum degree strictly less
than $a$. For example, every $a$-structured simplicial complex is $a$-nice.

\begin{sublem}\label{be_nice}
Assume $C_2$ is $a$-nice.
 There is a stretching $C_3$ of $C_2$ that is $a$-structured or else has a non-planar link.
\end{sublem}

\begin{cproof}
We prove this claim by induction on the degree-parameter as defined in \autoref{s2}.
So let $C_2$ be an $a$-nice simplicial complex such that all $a$-nice simplicial complexes with
strictly smaller degree-parameter have a stretching that is $a$-structured or has a non-planar
link.

We may assume that the simplicial complex $C_2$ is not $a$-structured; that is, it has a vertex $v$
such that the link graph $L(v)$ is 2-connected but no parallel graph and $L(v)$ has vertex $e$ of
degree at
least $a$. Also we may assume that $L(v)$ is planar.

{\bf Case 1:} the vertex $e$ is not contained in a proper 2-separator of $L(v)$.
Since embeddings of 2-connected graphs in the plane are unique up to flipping at 2-separators by a
theorem of Whitney, any embedding of the graph $L(v)$ in the plane has the same rotator at the
vertex $e$ (up to reversing). Take two edges $f_1$ and $f_2$ incident with the vertex $e$ that
are adjacent in the rotator.
Now we stretch the edge $e$ of $C_2$ in the direction of the faces corresponding to $f_1$ and
$f_2$. The link graphs at all vertices of $C_2$ except for $v$ and the
other
endvertex $w$ of the edge $e$ of $C_2$ do not change. In the link
graphs for $v$ and $w$ the vertex $e$ is replaced by two new vertices (and a path of
length
two
joining them), each of strictly smaller
degree than $e$, as its degree $a$ is at least four. This new simplicial complex $C_3$ has strictly
smaller degree-parameter than
$C_2$.

In order to be able to apply induction, we need to show that $C_3$ is $a$-nice. The link graph at
$v$ is still 2-connected in $C_3$. If the link graph at $w$ in $C_2$ is 2-connected, this is
still true in $C_3$. Hence it remains to consider the case that it is a star of parallel
graphs. In this case the vertex $e$ must be the cutvertex of $L(w)$ by the choice of $a$. Then in
the simplicial complex $C_3$, the link graph at $w$ has maximum degree less than $a$. Thus $C_3$ is
$a$-nice and we can apply the induction hypothesis. So $C_3$ has a stretching of the desired type,
and $C_3$ is a stretching of $C_2$. This completes the induction step in this case.

{\bf Case 2:} not Case 1. Then the vertex $e$ is contained in a proper 2-separator of $L(v)$. Let
$x$ be the other vertex in that 2-separator. We obtain $C_3$ from $C_2$ by stretching at the
2-separator $\{e,x\}$. The simplicial complex $C_3$ has strictly smaller degree-sequence
than
$C_2$ by \autoref{2-stretch-change}. We verify that $C_3$ is $a$-nice. All link graphs at new
vertices are still 2-connected in $C_3$. Let $w$ and $w'$ be the endvertices of the edges $e$ and
$x$ aside from $v$, respectively. Hence it remains so show
that the link graphs at $w$ and $w'$ in $C_3$ are 2-connected, stars of parallel graphs whose
cutvertex has degree $a$ or have maximal degree less than $a$. If the link graph at $w$ in $C_2$ is
2-connected, it is also 2-connected in $C_3$ (as coadding a star preserves 2-connectedness). Hence
we may assume that the link graph at $w$ in $C_2$ is a star of parallel graphs. and $e$ is its
cutvertex by the choice of $a$. Then either $L(w)$ is still a star of parallel graphs in $C_3$ or
else it has maximum degree less than $a$. The same analysis applies to the vertex `$w'$' in place
of `$w$'. Thus $C_3$ is $a$-nice. So by induction there is a stretching of $C_3$ of the desired
type, and $C_3$ is a stretching of $C_2$. This completes the induction step, and hence the proof of
this claim.
\end{cproof}

Let $C_3$ be stretching of the simplicial complex $C_2$ as in \autoref{be_nice}. If $C_3$ has a
non-planar link, we are done. Hence we may assume that the simplicial complex $C_3$ is
$a$-structured. If $C_3$ has cutvertex-degree less than $a$, we can apply the induction
hypothesis.
Hence we may assume that $C_3$ has a vertex $v$ such that the link graph at $v$ is a star of
parallel graphs whose cutvertex $e$ has degree precisely $a$. Now we show how the property
`$a$-structured' implies the existence of certain paths, which can then be
contracted to reduce the cutvertex-degree.

\begin{sublem}\label{path_exists}
 There is a path $P_n$ from the vertex $v$ starting with $e$ to another vertex $w_n$ whose link
graph is a star of parallel graphs. All link graphs at internal vertices of the path are parallel
graphs and all edges of the path have the same face-degree.
\end{sublem}

\begin{cproof}
We build the path $P_n=w_0e_1w_1...e_nw_{n}$ recursively as follows.
We start with $e_1=e$ and $w_0=v$ and let $w_1$ be the endvertex of the edge $e$ aside from $v$.
Assume we already constructed $w_0e_1w_1...e_iw_{i}$.
If the link graph $L(w_i)$ is a star of parallel graphs we stop and let $i=n$ and
$w_i=w_n$.
Otherwise by assumption, the link graph $L(w_i)$ must be 2-connected. As the edge $e_i$ has
degree precisely $a$ the link graph $L(w_i)$ of the $a$-structured simplicial complex $C_3$ is a
parallel graph.
So the link graph $L(w_i)$ contains a unique vertex except from $e_i$ that has degree larger than
two, and this vertex has the same degree as the vertex $e_i$. We pick this vertex for $e_{i+1}$.
Note that
$e_{i+1}$ is an edge of the simplicial complex $C$. We let $w_{i+1}$ be the endvertex of
$e_{i+1}$ different from $w_i$.
Note that all edges $e_i$ have the same face-degree by construction.
Since any path\footnote{A \emph{path} in a graph is a sequence alternating between vertices and
edges such that adjacent members are incident, and all vertices (and edges) are distinct.} in $C$
must be finite, it suffices to prove the following:

\begin{fact}
For all $i$ the walk $P_i$ is a path.
\end{fact}

\begin{proof}
We prove this by induction on $i$. The base case is that $i=1$.
Suppose for a contradiction there is some $j< i$ such that
$w_i=w_j$.

{\bf Case 1:} $j=0$: then $e_i$ must be equal to the only vertex of $L(v)$ of the same degree; that
is, $e_i$ is equal to the edge $e_1$. But then the endvertex $w_{i-1}$ of the edge $e_i$ is equal
to the vertex $w_1$. This is a contradiction to the induction hypothesis. Hence $w_i$
cannot be equal to $w_0$.

{\bf Case 2:} $j\geq 1$: as in the link graph $L(w_j)$ the only two vertices with the same degree
as the vertex $e_i$ are $e_{j}$ and $e_{j+1}$, it must be that the edge $e_i$ is equal to one of
these two edges; that is, the endvertex $w_{i-1}$ of $e_i$ must be equal to
$w_{j-1}$ or $w_{j+1}$.
The vertex $w_{j-1}$ cannot be an option by the induction hypothesis.
Similarly, the vertex $w_{j+1}$ cannot be an option by the induction hypothesis if $j+1< i-1$. So
$j+2\geq i$, so
$j=i-2$ or $j=i-1$.
Since the simplicial complex $C$  has no loops or parallel edges any three consecutive
vertices on $P_i$, such as $w_{i-2}$, $w_{i-1}$ and $w_i$, are distinct. Hence neither  $j=i-2$ nor
$j=i-1$
are possible. Thus we have also reached a
contradiction in this case. Hence the vertex $w_i$ is distinct from all previous vertices on the
walk $P_i$.
\end{proof}

\end{cproof}

Given a path $P_n$ with endvertex $w_n$ as in \autoref{path_exists}, whose link graph $L(w_n)$ at
$w_n$ is a star of parallel
graphs, denote the (unique) cut-vertex of the link graph $L(w_n)$ by $x$.
\begin{sublem}\label{cutvx}
 The cut-vertex $x$ is equal to the last edge $e_n$ on the path $P_n$.
\end{sublem}

\begin{cproof}
We denote the degree of the cut-vertex $x$ by $a'$. By the definition of $a$, we have, $a'\leq a$.

On the other hand by \autoref{path_exists} the vertices $e$ and $e_n$ have the same degree in
the graphs $L(v)$ and $L(w_n)$, and this
degree is equal to $a$ by the choice of the vertex $v$.
As $L(w_n)$ is a star of parallel graphs with a
cut-vertex, the degree of the cut-vertex $x$ is strictly larger than the degree of any other vertex
of
$L(w_n)$.
Hence it must be that $a=a'$ and $x=e_n$.
\end{cproof}

By \autoref{path_exists} and \autoref{cutvx}, there is a set of vertex-disjoint paths in $C_3$
such that any of their endvertices has a link graph that is a star of parallel graphs whose
cutvertex has degree $a$. All internal vertices of these paths are parallel graphs. By taking
this collection maximal, we ensure that any vertex whose link graph is a star of parallel
graphs whose
cutvertex has degree $a$ is an endvertex of one of these paths. We denote the set of these paths by
$\Pcal$.
We obtain the 2-complex $C_4$ from $C_3$ by contracting all edges on these paths of $\Pcal$.
Contracting the edges on the paths recursively, we note at each step that these edges are
reversible by \autoref{is_rev}.
At all vertices of $C$ except
for those vertices on the paths, the two $2$-complexes $C_3$ and $C_4$ have the same link graphs.
In
addition, $C_4$ has the contraction vertices, one for each of the vertex-disjoint paths. These
link graphs are the vertex-sum of the link graphs at the vertices on its path, see
\autoref{sec_vertex_sum} for background on vertex-sums.
So the link graph at a new contraction vertex is (isomorphic to) the vertex sum of the link graphs
at the two endvertices plus various subdivision vertices coming from the parallel graphs at
internal vertices of the path. By \autoref{cutvx}, each
of
these vertices in the link graph has degree strictly less than $a$. Hence all new contraction
vertices have maximum degree less than $a$. Hence the cutvertex-degree of $C_4$ is
 strictly smaller than $a$.
 So the 2-complex $C_4$ satisfies all the conditions to apply the induction hypothesis except that
it may not be a simplicial complex as it may have edges that are loops or parallel edges.

Now we
show how we can stretch local branches of $C_3$ to get a simplicial complex $C_3'$ so that the
simplicial complex $C_4'$ obtained from $C_3'$ by contracting all the paths in $\Pcal$ is a
simplicial complex. We obtain $C_3'$ from $C_3$ by stretching at each endvertex of a path in $\Pcal$
all the branches and at each interior vertex of a path in $\Pcal$ we stretch at the 2-separator
consisting of the two branching vertices of its parallel graph. We obtain $C_4'$ from $C_3'$ by
contracting the above defined family of paths $\Pcal$. It is straightforward to check that $C_4'$
is a simplicial complex -- and is a stretching of $C$ with smaller cutvertex-degree. This completes
the induction step, and hence this proof.
\end{proof}

\subsection{Proofs of \autoref{Kura_simply_con} and \autoref{main_streching}}

We conclude this section by proving the following theorems mentioned in the Introduction.

\begin{proof}[Proof of \autoref{main_streching}]
Let $C$ be a simplicial complex. Recall that  \autoref{main_streching} says there is a simplicial
complex $C'''$ obtained from $C$ by stretching
so that $C'''$ is
locally almost 3-connected and stretched out or $C'''$ has a non-planar link; moreover $C$ has a
planar rotation system if and
only if $C'''$ has a planar rotation system.

 By \autoref{reduce_to loc_2-con} there is a stretching $C'$
of $C$ that is a simplicial complex that is locally almost 2-connected or has a non-planar link.
As we are done otherwise, we may assume that $C'$ is locally almost 2-connected.
By
\autoref{stretch_loc_con}
there is a stretching
$C''$ of $C'$ that is a locally almost 3-connected simplicial complex.

 \begin{sublem}\label{make_strected_out}
 Let $C''$ be a locally almost 3-connected simplicial complex. Then there is a
stretching $C'''$ of $C''$ that has additionally the property that it is stretched out.
\end{sublem}

\begin{cproof}
We say that an edge of face-degree two is \emph{stretched out} if it has one endvertex that is not
a subdivision of a 3-connected graph or a parallel graph whose branch vertices have degree at least
three. Note that a simplicial complex in which every edge of degree two is stretched out is
stretched out itself.
We prove this claim by induction on the number of edges of degree two that are not stretched
out.
So assume there is an edge $e$ that is not stretched out. Let $v$ be one of its endvertices.

{\bf Case 1:} the link graph at $v$ is a parallel graph whose two branch vertices $x_1$ and $x_2$
have degree at least three. Then we stretch at the 2-separator $(x_1,x_2)$ at $v$.
This gives a simplicial complex $\tilde C$ that in addition to the vertex $v$ has also one new
vertex for
every component of $L(v)-x_1-x_2$. The link graphs at these new vertices are cycles. Hence every
edge of degree two incident with these new vertices is stretched out. Thus $\tilde C$ has strictly
less
edges of degree two that are not stretched out.

{\bf Case 2:} the link graph at $v$ is a subdivision of a 3-connected graph.
Then the vertex $e$ of $L(v)$ is contained in a subdivided edge. Let $P$ be the path of that
subdivided edge,
and let $x_1$ and $x_2$ be its endvertices. Then we stretch at the 2-separator $(x_1,x_2)$ at $v$.
The rest of the analysis is analogue to Case 1. This completes the proof of the claim.
\end{cproof}

By \autoref{make_strected_out} we may assume that $C$ has a stretching $C'''$ that is a locally
almost 3-connected
and stretched out simplicial complex.
The `Moreover'-part follows from the fact that $C'''$ is a stretching of $C$ as shown in
\autoref{planar_rot_preserve}. This completes the proof.
 \end{proof}

 \begin{proof}[Proof of \autoref{Kura_simply_con}]
Let $C$ be a simplicial complex such that the first homology
group $H_1(C,\Fbb_p)$ is trivial for some prime $p$.
By \autoref{combi_intro_extended} $C$ embeds into $\mathbb{S}^3$ if and only if it is simply connected and has a planar rotation system; so from now on assume that $C$ is simply connected.

Recall that \autoref{Kura_simply_con} says that $C$ has an embedding in 3-space if and
only if $C$ has no stretching that has a space minor in $\Ycal\cup \Tcal$.
By \autoref{combi_intro} $C$
is embeddable in 3-space if and only if it has a
planar rotation system.

By \autoref{main_streching} there is a simplicial complex $C'$ that is a stretching of $C$.
Moreover $C$ has a planar rotation system if and only if $C'$ has a planar rotation system.
By that theorem either the simplicial complex $C'$ has a non-planar link or it is locally almost
3-connected and stretched out. In the first case,
by
Kuratowski's theorem, \autoref{get-cone} and
\autoref{cone-red}, the simplicial complex $C'$ has a minor in the finite list
$\Ycal$ -- so the theorem is true in this case. In the second case by \autoref{Kura_almost} $C'$ has
a planar rotation system if and only if it has no space minor in
$\Ycal\cup \Tcal$.
This completes the proof.
  \end{proof}

\section{Algorithmic consequences}\label{algo_sec}

Our proofs give a quadratic algorithm that verifies whether a given 2-dimensional
simplicial complex has a planar rotation system. This gives a quadratic algorithm that checks
whether a given 2-dimensional simplicial complex has an embedding in a (compact)
orientable 3-manifold by
\autoref{is_manifold} (for the general, not necessarily orientable, case see
\autoref{non-or}). In particular, for simply connected 2-complexes this
gives
a quadratic algorithm
that tests embeddability in 3-space by Perelman's theorem. The algorithm has several components.
Next we explain them in reverse order and prove the relevant lemmas afterwards.
We say that a simplicial complex $D$ is \emph{equivalent} to $C$ if $D$ has a planar rotation system if and only if $C$ has one.

\begin{enumerate}
 \item[4.] The locally almost 3-connected and stretched out case.
 This subroutine checks whether a locally almost 3-connected and stretched out simplicial complex admits a planar rotation system.
 The corresponding fact in the paper
is \autoref{41_2}. This clearly has a linear time algorithm.

\item[3.] Reduction of the locally almost 3-connected case to the locally almost 3-connected and
stretched out case.
This subroutine takes as input a locally almost 3-connected simplicial complex $D_2$ and the output is a simplicial complex $D_3$ equivalent to $D_2$ with the property that it is locally almost 3-connected and stretched out.
The corresponding fact in the paper
is \autoref{make_strected_out}. This clearly has a linear time algorithm.

 \item[2.] Reduction of the locally almost 2-connected case to the locally almost 3-connected case.
 This subroutine takes as input a locally almost 2-connected simplicial complex $D_1$ and the output is a simplicial complex $D_2$ equivalent to $D_1$ and that additionally is locally almost 3-connected.
 The corresponding fact in the paper
is \autoref{stretch_loc_con}.

To analyse the running time, we do this step slightly differently
than in the paper. First we compute a Tutte-decomposition\footnote{A Tutte-decomposition is a
decomposition of a graph (or matroid) into its 3-connected components along 2-separators. In the
special case of graphs it is also known as the \emph{SPQR tree}.} at every 2-connected link graph.
This
tells us precisely how we can stretch that vertex along 2-separators. Doing these stretchings at
different vertices may affect the link graphs at other vertices. Indeed, it may affect other
vertices in that we coadd stars at their link graphs.
However, once a link graph is a
subdivision of a 3-connected graph, it will stay that. So the vertices we may have to look at
multiple times are vertices where the link graphs are parallel graphs. But if we need to stretch
there again, the maximum degree goes down. Using \autoref{coadd_star} below, it is straightforward
to show that this step can be done
in linear time.
 \item[1.] Reduction of the general case to the locally almost 2-connected case.    This subroutine takes as input a simplicial complex $C$ and obtains a witness that $C$ has no planar rotation system or the output is a simplicial complex $D_1$ equivalent to $C$ and that additionally is locally almost 2-connected. The corresponding
fact in the paper is \autoref{reduce_to loc_2-con}.

This is done by
recursion on the cutvertex-degree $a$. So let us analyse the step from $a$ to $a-1$ in detail.
The input is a simplicial complex $C$ and we measure its size by $\sum (deg(e)-2)$, where the sum
ranges over all edges $e$ of $C$ of degree at least three. We refer to that sum as the \emph{degree
parameter}.\footnote{
We remark that at edges of degree at most two the compatibility conditions for planar rotation
systems is always satisfied and hence we do not need to take them into account in the definition of
the degree parameter.} Here $deg(e)$ denotes the number of faces incident with $e$.

The stretching related to \autoref{stretching_loc_1-cuts} can be done in linear time as computing
the block-cutvertex-tree of link graphs can be done in linear time. For the part
corresponding to \autoref{be_nice} we compute the stretching via Tutte-decompositions as in step 3
explained above, and then we check for planarity for each 3-connected link graph. If it is planar,
we remember a planar rotation system and if we stretch later an edge incident with that vertex at
the other endvertex we check whether this stretching is compatible with the chosen planar rotation
system. This can be done in
linear time. The construction of the set $\Pcal$ of paths can clearly be done
in linear time.
Hence the whole recursion step from $a$ to $a-1$ just takes linear time. The output is the
simplicial complex $C_4'$ from the last paragraph of the proof of \autoref{reduce_to loc_2-con}.

However, with the current argument, the degree parameter of $C_4'$ might be larger
than the degree parameter of the input $C$. Indeed, stretching a local branch may increase the
degree parameter. Hence here
we explain how we modify the construction of the simplicial complex $C_4'$ so that the degree
parameter does not increase.
First note that none of the stretching operations except for stretching a branch increases the
degree parameter, compare \autoref{degpar}.
Recall the definition of $C_3'$ from the last paragraph of the proof of \autoref{reduce_to loc_2-con}.
We obtain $C''$ from $C_3'$ by contracting all edges
$e[B]$ of degree at least three that
were added by stretching a local branch, and contracting the resulting faces of size two (this
has the effect of reversing the stretching operations at those edges $e[B]$); additionally we
stretch so that no edge of degree two has both endvertices on the same path -- similarly as in the
construction of stretched out in \autoref{make_strected_out} (this ensures that edges of degree two
do not make a problem
later. This does
not increase the degree parameter). It is easy to see that $C''$ is a simplicial complex and that
$C$ has a planar rotation system if and only if $C''$ has one.
Each path $P\in \Pcal$ of $C_3'$ contracts onto a closed trail of $C''$ (in fact this closed trail is a path or a cycle as edges incident with interior vertices of $P$ but not on $P$ have degree two and are not contracted when going from $C_3''$ to $C''$ by definition). We obtain $C_3''$ from
$C''$ by stretching branches for each closed trail $P\in \Pcal$ as follows. Recall that the endvertices of the path $P$ are denoted by $v$ and $w$.

{\bf Case 1:} in the simplicial complex $C''$ the trail $P$ has at least one internal vertex.
Denote the edge of $P$ incident with $v$ by
$e_1$ and the edge of $P$ incident with $w$ by
$e_2$. Then we stretch at the link graphs of
 $v$ and $w$ all branches at $e_1$ and $e_2$, respectively.

{\bf Case 2:} not Case 1. Note that $P$ must consist of at least one edge by
\autoref{path_exists}. And that edge is not a loop as $C''$ is a simplicial complex.
So $P$ consists of a single edge $e$.
Then in the simplicial complex $C''$ there is no edge in parallel to $e$. We stretch all branches
of the link graph at  $v$ at the vertex $e$ (but not for $w$).

We obtain $C_4''$ from  $C_3''$ by contracting all $P\in \Pcal$. It is
straightforward to check that $C_4''$ is a simplicial complex and that the degree parameter of
$C_4''$ is at
most that of $C$, compare \autoref{C4}. By construction the cutvertex degree of $C_4''$ is
strictly smaller than that of $C$. So $C_4''$ is a suitable output of the recursion step. As each
recursion step takes linear time, all of them together take at most quadratic time.
\end{enumerate}

This completes the description of the steps of the algorithm. Together they form the following algorithm. The input is a simplicial complex $C$ and we want to check whether it has a planar rotation system.
In step 1, we either obtain a witness that $C$ has no planar rotation system or the output is a simplicial complex $D_1$ equivalent to $C$ and that additionally is locally almost 2-connected.
In step 2, we take $D_1$ has input and the output is a simplicial complex $D_2$ equivalent to $D_1$ and that additionally is locally almost 3-connected.
Step 3 works like step 2 but here the output has the additional property that it is stretched out.
Finally in step 4, we check the existence of a planar rotation system for locally almost 3-connected and stretched out simplicial complexes.

\subsection{Some lemmas for the algorithm above}

Here we prove the lemmas referred to in the beginning of \autoref{algo_sec}.

The \emph{degree-parameter} of a graph $G$ is $\sum(deg(v)-2)$, where the sum ranges
over all vertices.

\begin{lem}\label{coadd_star}
 Coadding a star at a vertex $v$ preserves the degree parameter.
 \end{lem}

\begin{proof}
 Let $k$ be the degree of the center of the coadded star and $v_1, .., v_k$ be the leaves of the
coadded star. The degree parameter of the graph before minus the degree parameter after the
coaddition is:
\[
deg(v)-2 -\left(\sum_{i=1}^k (deg(v_i)-2)+(k-2)\right)
\]
As $deg(v)=\sum_{i=1}^k deg(v_i)-k$ the above sum evaluates to zero, completing the proof.
\end{proof}

\begin{lem}\label{degpar}
All stretching operations except for possibly stretching a local branch do not increase the degree
parameter.
\end{lem}

\begin{proof}
 This lemma is immediate for splitting vertices and contracting reversible non-loops.
 If we stretch an edge, first note that pre-stretching does not change the degree parameter as
$deg(e)=deg(e_1)+deg(e_2)-2$ if $e$ is stretched to $e_1$ and $e_2$.  Then note that subdivisions
of faces do not change the degree parameter.

The fact that 2-stretching does not change the degree parameter is proved similarly as
\autoref{coadd_star}.
\end{proof}

 \begin{lem}\label{C4}
The complex $C_4''$ is a simplicial complex whose degree parameter is not larger
than that of $C$.
\end{lem}

\begin{proof}
 As mentioned above $C''$ is a simplicial complex. By \autoref{degpar} the degree parameter of
$C''$ is not larger than that of $C$.
Hence it suffices to show for each trail $P\in \Pcal$ that the construction given in the
two cases plus the contraction afterwards preserves being a simplicial complex and does not
increase the degree parameter.

First we treat Case 1; that is, $P$ has an internal vertex. In the construction of $C''$ we never
contract an edge incident with an internal vertex of $P$ that is not on the path $P$, as they all have face-degree two. Hence $P$ is
a path in $C''$ or a cycle. Then we stretch all the branches at the local cutvertices $e_1$ and
$e_2$ in the link graphs at $v$ and $w$, respectively. The sum of degrees of the new edges $e[B]$
is at most $deg(e_1)+deg(e_2)$. So the degree parameter increased by at most $2deg(e_1)-4$ (as
$deg(e_1)=deg(e_2)$). Then we contract all edges on $P$, and resulting faces of degree two. This decreases the degree parameter at
least by that amount, as $P$ contains at least two edges. Thus in total the degree parameter does not increase. The stretching
before
the contraction ensures that all edges incident with an interior vertex of $P$ have an endvertex whose link graph is a cycle and thus on none of the paths from $\Pcal$. So we do not create parallel edges or loops. So $C_4''$ is a simplicial complex.

The analysis in Case 2 is similar. Here note that $P$ is a path of length one, and so $v\neq w$. Thus one stretching suffices to ensure that $C_4''$ is a simplicial complex.
\end{proof}

   \section*{Acknowledgement}

I thank Radoslav Fulek for pointing out an error in an earlier version of this paper.
I thank Arnaud de Mesmay for useful discussions that led to \autoref{non-or} in the extended version of \cite{3space2}.
I thank Nathan Bowler and Reinhard Diestel for useful discussions on this topic.

\bibliographystyle{plain}
\bibliography{literatur}

\end{document}